\documentclass[11pt]{article}
\usepackage{amsthm}
\usepackage{amssymb}
\usepackage{amsmath}
\usepackage{hyperref}
\usepackage{graphicx}
\usepackage{caption}
\usepackage{array}
\usepackage{enumitem}
\usepackage{mathtools}
\usepackage{etoolbox}
\usepackage{bbm}
\usepackage{mathdots}
\usepackage[backend=biber, style=alphabetic, maxbibnames=9, maxcitenames=9, maxalphanames=9]{biblatex}
\bibliography{TypesMonodromy}
\usepackage[all,cmtip]{xy}
\usepackage{tensor}

\usepackage{tocloft}
\pretocmd{\section}{%
  }{}{}

\numberwithin{table}{section}

\newtheorem{theorem}{Theorem}[section]
\newtheorem*{theorem*}{Theorem}
\newtheorem{proposition}[theorem]{Proposition}
\newtheorem*{proposition*}{Proposition}
\newtheorem{corollary}[theorem]{Corollary}
\newtheorem{lemma}[theorem]{Lemma}

\newtheorem*{lemma*}{Lemma}

\theoremstyle{definition}
\newtheorem{definition}[theorem]{Definition}

\newtheorem*{exercise*}{Exercise}
\newtheorem{remark}[theorem]{Remark}

\newtheorem{assumption}[theorem]{Assumption}
\newtheorem{notation}[theorem]{Notation}

\numberwithin{equation}{section}
\numberwithin{figure}{section}

\let\hom\relax
\let\det\relax
\let\L\relax

\newcommand{\mb}{\mathbb}
\newcommand{\mbf}{\mathbf}
\newcommand{\mc}{\mathcal}
\newcommand{\mf}{\mathfrak}
\newcommand{\mr}{\mathrm}

\newcommand\pmat[1]{\begin{pmatrix}#1\end{pmatrix}}
\newcommand\L[1]{\prescript{L}{}{#1}}
\newcommand\tp[1]{\prescript{t}{}{#1}}

\newcommand{\Z}{\mathbb{Z}}
\newcommand{\Q}{\mathbb{Q}}
\newcommand{\R}{\mathbb{R}}
\newcommand{\C}{\mathbb{C}}

\newcommand{\A}{\mathbb{A}}

\newcommand{\GL}{\mathrm{GL}}

\newcommand{\Sp}{\mathrm{Sp}}

\newcommand{\sset}[2]{\lbrace{#1}\,\,|\,\,{#2}\rbrace}

\DeclareMathOperator{\aut}{Aut}

\DeclareMathOperator{\chars}{char}

\DeclareMathOperator{\cyc}{cyc}

\DeclareMathOperator{\det}{det}
\DeclareMathOperator{\diag}{diag}

\DeclareMathOperator{\disc}{disc}

\DeclareMathOperator{\End}{End}

\DeclareMathOperator{\Frac}{Frac}
\DeclareMathOperator{\Frob}{Frob}

\DeclareMathOperator{\gal}{Gal}

\DeclareMathOperator{\hom}{Hom}
\newcommand{\id}{\mathrm{id}}

\DeclareMathOperator{\Ind}{Ind}

\DeclareMathOperator{\Jac}{Jac}

\newcommand{\modulo}[1]{\,\,(\mathrm{mod}\,\,{#1})}

\DeclareMathOperator{\nn}{nn}

\DeclareMathOperator{\restr}{res}

\DeclareMathOperator{\Span}{Span}
\DeclareMathOperator{\spec}{Spec}

\DeclareMathOperator{\sss}{ss}

\DeclareMathOperator{\St}{St}
\DeclareMathOperator{\std}{Std}

\DeclareMathOperator{\Sym}{Sym}

\DeclareMathOperator{\tr}{Tr}

\DeclareMathOperator{\vol}{Vol}
\DeclareMathOperator{\WD}{WD}

\begin{document}
\title{Degeneration of monodromy in the theory of Eisenstein congruences for classical groups}
\author{Sam Mundy}
\date{}
\maketitle
\begin{abstract}
We give a new Hecke-theoretic method to control the growth of monodromy in $p$-adic families of automorphic Galois representations for unitary groups and quasisplit classical groups.
\end{abstract}
\tableofcontents

\section*{Introduction}
\subsubsection*{The motivating problem}

This paper has been written to supply a technical ingredient for use in the theory of Eisenstein congruences to construct elements in Selmer groups for automorphic Galois representations. We have provided a tool which, in the context of various classical groups, can be used to ensure that certain Galois cohomology classes constructed via variants of Ribet's method actually satisfy the local conditions required of Selmer classes. Let us be more precise now.

Assume we are given a family of Galois representations parametrized by, say, an affinoid rigid space $\mf{X}$ over $\Q_p$ for some prime $p$. So we have a number field $F$ with absolute Galois group $G_F$, and a Zariski dense subset $\Sigma\subset\mf{X}(\overline\Q_p)$ for which, whenever $x\in\Sigma$, we are given a continuous representation $\rho_x:G_F\to GL_N(\overline\Q_p)$ for some positive integer $N$. These Galois representations should be interpolated by $\mf{X}$, and hence should satisfy that $x\mapsto\tr(\rho_x(g))$ is analytic for every $g\in G_F$, i.e., there is an element $F_g\in\mc{O}(\mf{X})$ so that, when $x\in\Sigma$, we have $F_g(x)=\tr(\rho_x(g))$.

Assume that the representations $\rho_x$ are generically irreducible, but that there is a point $x_0\in\mf{X}(\overline\Q_p)$ where the function $g\mapsto F_g(x_0)$ is the trace of a reducible Galois representation, say $\rho_{x_0}$. Let $\rho_1$ and $\rho_2$ be two different constituents of $\rho_{x_0}$. Then there are methods, which in some way all seem to go back to the classical paper of Ribet \cite{ribet}, by which one can try to construct nontrivial extensions of $\rho_2$ by $\rho_1$ by leveraging the generic irreducibility of this family of Galois representations. The result may be interpreted as a class in the Galois cohomology group $H^1(F,\rho_1\otimes\rho_2^\vee)$.

Following Bloch and Kato, one can define a distinguished subgroup,
\[H_f^1(F,\rho_1\otimes\rho_2^\vee)\subset H^1(F,\rho_1\otimes\rho_2^\vee)\]
called the Bloch--Kato Selmer group. The celebrated Bloch--Kato conjectures \cite{BK} tell us, at least when $\rho_1\otimes\rho_2^\vee$ is geometric in the sense of Fontaine--Mazur, that the size of this group should be dictated by the theory of $L$-values. So it is of interest to construct elements in $H_f^1(F,\rho_1\otimes\rho_2^\vee)$ when these conjectures tell us they should exist.

This Bloch--Kato Selmer group is defined by conditions at every finite place of $F$. When $v$ is such a place and $v\nmid p$, the condition is that we must have $c|_{I_v}=0$ in order for a class $c\in H^1(F,\rho_1\otimes\rho_2^\vee)$ to be in $H_f^1(F,\rho_1\otimes\rho_2^\vee)$; here $I_v$ is the inertia group at $v$.

Now if the Galois representations in the family parametrized by $\mf{X}$ are rather arbitrary, we have little hope of establishing this condition for a class $c$ constructed in this way. However, often such families come from families of automorphic forms or automorphic representations, meaning that there is a reductive group $G$ and for every $x\in\Sigma$, there is an automorphic representation $\pi_x$ for $G$ such that $\rho_x$ is the Galois representation attached to $\pi_x$. The reducible Galois representation $\rho_{x_0}$ might then be attached to an automorphic representation $\pi_{x_0}$ consisting of Eisenstein series, while the $\pi_x$ for $x\in\Sigma$ should be cuspidal. This puts us in the setting of the theory of Eisenstein congruences. In this automorphic situation, one can then hope to influence the behavior of the representations $\rho_x$ for $x$ near $x_0$ by automorphic means, by influencing the behavior of the Eisenstein representation $\pi_{x_0}$ in some corresponding way. This paper is an exercise in this practice.

More precisely, what we have done here is the following. We work in the setting where the group $G$ above is among a large class of classical groups, and the Eisenstein representation $\pi_{x_0}$ comes by induction from an automorphic representation $\sigma$ of a maximal Levi subgroup of $G$ which has a factor $H$ of the same type as $G$ but one rank lower. For $v\nmid 2p$, we then single out a certain constituent, say $\Pi_v$, of the parabolically induced representation to $G$ from the local component $\sigma_v$ of $\sigma$ at $v$, and we use type theory to produce a certain Hecke operator $\phi_v$ acting nontrivially on this constituent $\Pi_v$. We show that, if the eigenvalues of $\phi_v$ are interpolated in this family on $\mf{X}$, and if $\pi_{x_0,v}$ contains this constituent $\Pi_v$, then the Galois representations $\rho_x$ for $x\in\Sigma$ in the same irreducible component of $\mf{X}$ as $x_0$ can be no more ramified at $v$ than the Galois representation attached to $\sigma$ itself.

This result offers enough control on the classes $c$ constructed in the way described above to ensure that they satisfy the Selmer condition at $v$.

\subsubsection*{Past work and future applications}

The idea to extend Ribet's method in the way described above as an approach towards the rank part of the Bloch--Kato conjecture first appears in the work of Skinner--Urban \cite{SUannounce} and \cite{SUgsp4}. There, the authors carry out this \textit{Skinner--Urban method} when the group $G$ above is $GSp_4$ over $\Q$. The representation $\pi_{x_0}$ at the point we called $x_0$ is then not actually generated by Eisenstein series, but rather by Siegel parabolic CAP forms. But in any case, the finite adelic component $\pi_{x_0,f}$ of this $\pi_{x_0}$ is isomorphic to the finite adelic component of an Eisenstein representation.

Subsequent work of Skinner--Urban \cite{SUunitary} and Bella\"{i}che--Chenevier \cite{BCU3} and \cite{BCbook}, and later the author \cite{MundyThesis}, carry out the method outlined above in other settings. Most notably, the works \cite{SUunitary} and \cite{BCbook} work with $G$ a unitary group of arbitrary rank. In this setting, controlling the ramification of the Galois representations $\rho_x$ in the families $\mf{X}$ considered in these works becomes an important and difficult problem. Above, we mentioned that we solve in this paper problems like this one using Hecke operators constructed via type theory. In fact, the idea to use type theory for this purpose first appears in writing in \cite{BCbook}. There, Bella\"{i}che and Chenevier use Hecke idempotents constructed by Schneider and Zink \cite{SZ} to control the ramification of the Galois representations $\rho_x$ at places $v$ which are split in the quadratic field, call it $E$, from which the unitary group $G$ is defined. At these places, $G$ becomes a copy of $GL_n$ over a local field, and indeed Schneider and Zink work only on $GL_n$ in their paper. A full fleshing out of \cite{SUunitary} would presumably have to use the same idempotents.

However, for the results \cite{SUunitary}, Skinner and Urban had to assume that $\pi_{x_0}$ is spherical at places $v$ which are inert or ramified in the quadratic field $E$. (Admittedly this condition is not so clearly spelled out directly in the reference \cite{SUunitary}. The author learned the Skinner--Urban method as it applies to unitary groups in 2018 from a course taught by Urban then. There, Chao Li took excellent notes \cite{LiSU} and in these notes this condition is clearly put.) However, one application of the main theorem of this paper is that one can remove the condition that $\pi_{x_0,v}$ be spherical if $v$ is inert or ramified, as long as $v\nmid 2$. So at finite places $v$ not dividing $2p$, one could thus allow $\pi_{x_0}$ to be arbitrary in order for the main theorems of \cite{SUunitary} to hold.

In \cite{BCbook}, Bella\"{i}che and Chenevier make a weaker assumption at inert places $v$, but one which is still rather strong. Presumably one can remove this condition there as well when $v\nmid 2$, as long as the eigenvarieties whose construction is relevant for their main results can be made to interpolate the eigenvalues of the operators $\phi_v$ constructed here.

We remark that it seems reasonable to the author that one could carry out the Skinner--Urban method for split odd special orthogonal groups over $\Q$, using techniques similar to \cite{MundyThesis} to make the relevant $p$-adic deformation of automorphic representations needed there. The main result of this paper could then be applied so that there are no conditions on the local representation $\pi_{x_0,v}$ in this case as long as $v\nmid 2p$.

Finally, work in progress of the author on \textit{ramified Eisenstein congruences} aims to construct cuspidal Hida families for the symplectic group $Sp_{2n}$ over $\Q$ degenerating to nonclassical Eisenstein series which are given extra level at an auxiliary prime $\ell$. The main result of this paper is set up to still apply in this nonclassical case, and will be a crucial ingredient in this work to show that any Galois cohomology classes constructed via these families are Selmer away from $2p\ell$.

\subsubsection*{Discussion of the main result}

Having discussed what the main theorem of this paper is useful for, we now work towards making a precise statement of it. The objects to which our theorem applies are \textit{Hecke families}, for which we give a precise and quite general definition in Definition \ref{defheckefamilies} below.

First, we fix throughout a totally real number field $F$ and a reductive group $G$ over $F$ which is either a quasisplit classical group or unitary group. In the unitary case, we let $E$ be the imaginary quadratic extension field over the field $F$ used to define $G$; otherwise we let $E=F$ in what follows. We assume the $F$-rank of $G$ is at least $1$. Then $G$ has a maximal parabolic subgroup, call it $P$, whose Levi factor, call it $M$, is the group $M\cong GL_{1/E}\times H$ with $H$ of the same type as $G$ but one rank lower. We also fix throughout this introduction an isomorphism $\iota:\C\to\overline\Q_p$ and use it to identify these two fields.

We remark here that will need to assume the full endoscopic classification for $G$ and $H$; see below for more remarks on this.

Now as we have defined it, a Hecke family $\mc{F}$ for $G$ consists of a quadruple of data $\mc{F}=(\mb{T},\mf{X},\Sigma,\Psi)$, and it comes with a \textit{level subgroup}, that is, a compact open subgroup $K_f\subset G(\A_{F,f})$ where $\A_{F,f}$ is the finite adele ring of $F$, and we require $K_f$ to factorize over all finite places of $F$. The entries of the quadruple defining $\mc{F}$ are described as follows: First, the entry $\mb{T}$ is a Hecke algebra; by definition, it is a $\Q_p$-subalgebra of $C_c^\infty(K_f\backslash G(\A_{F,f})/K_f,\overline\Q_p)$ which is generated over $\Q_p$ by all the spherical Hecke operators at all but finitely many finite places of $F$ at which $K_f$ is hyperspecial, as well as possibly by finitely many other operators in $C_c^\infty(K_f\backslash G(\A_{F,f})/K_f,\overline\Q_p)$. The family $\mc{F}$ is then supposed to interpolate eigenvalues or, more generally, traces of the operators in $\mb{T}$ using the other three entries of the tuple defining it.

The second entry $\mf{X}$ is an affinoid rigid space over $\Q_p$, with affinoid ring $\mc{O}(\mf{X})$. The third entry $\Sigma$ is a Zariski dense subset of points in $\mf{X}(\overline\Q_p)$, which is supposed to be viewed as a set of \textit{classical points} in $\mc{F}$; by definition, attached to each $x\in\Sigma$ is a discrete, cohomological automorphic representation $\pi_x$ for $G$. We note that we also assume the mild condition that $\pi_x$ is $\std$-regular at every archimedean place of $F$ if $G$ is even orthogonal; see Definition \ref{defstdreg}.

Finally, the fourth entry is the mechanism which ties the previous three together. It is a $\Q_p$-linear map $\Psi:\mb{T}\to\mc{O}(\mf{X})$ with the following interpolation property: For every $x\in\Sigma$, there exists a nontrivial constituent $V_x$ of $\pi_{x,f}^{K_f}$, the $K_f$-fixed vectors in the finite adelic component of $\pi_x$, such that for any $\phi\in\mb{T}$, we have
\[x(\Psi(\phi))=\tr(\phi|V_x).\]
In other words, via $\Psi$, we interpolate partial traces of all operators $\phi\in\mb{T}$ by analytic functions on $\mf{X}$.

As an example, if $\mb{T}$ consists of spherical Hecke operators and each $V_x$ is $1$-dimensional, then $\mc{F}$ could come from a piece of an eigenvariety, in the usual sense. Though in this case, we would be forgetting the weight space and the $U$-operators at places above $p$.

Now let us fix such a Hecke family $\mc{F}=(\mb{T},\mf{X},\Sigma,\Psi)$ as above, with level $K_f$. To get nice Galois representations, we must assume some conditions on the automorphic representations $\pi_x$ for $x\in\Sigma$; we will assume that $\mc{F}$ is what we have called \textit{strongly} $\Sigma$\textit{-generic} in Definition \ref{defheckefamilies}; see also Definition \ref{defstrgen}.

The assumption that each $\pi_x$ for $x\in\Sigma$ is cohomological, along with our $\std$-regularity assumption if $G$ is even orthogonal, are used to ensure the existence of standard Galois representations $\rho_x:G_E\to GL_N(\overline\Q_p)$ attached to each $\pi_x$; here $N$ is the dimension of the standard representation of the dual group $G^\vee$ to $G$. In fact, we assemble the proper references and construct these Galois representations in Theorem \ref{thmconstofgalois}. The strong genericity assumption then guarantees that the representations $\rho_x$ are pure of the appropriate weight at all places of $F$ not above $p$. This purity will later be crucial for our method.

We remark that in most applications to situations based on Ribet's method, we will want the Galois representations $\rho_x$ to be irreducible for generic $x\in\Sigma$. This irreducibility then implies that the lift of $\pi_x$ to $GL_N$ is cuspidal for $x$ in a Zariski dense subset $\Sigma'$ of $\Sigma$, which in turn implies our strong genericity hypothesis for the new family $\mc{F}'=(\mb{T},\mf{X},\Sigma',\Psi)$. So for the applications we have in mind, strong genericity constitutes no loss of generality.

Now that we have our strongly $\Sigma$-generic Hecke family $\mc{F}$, we put it into the Ribet-style setting described previously by introducing an Eisenstein point. With $P$ the parabolic subgroup of $G$ we fixed above with Levi factor $M\cong GL_{1/E}\times H$, we let $x_0\in\mf{X}(\overline\Q_p)$ be a point which we assume is \textit{strongly} $M$\textit{-Eisenstein}, as in Definition \ref{defheckefamilies}. This means the following: First, there is a Hecke character $\chi$ of $GL_1(\A_E)$ and an automorphic representation $\sigma$ for $H$, and $\sigma$ is assumed to be discrete and cohomological, and to satisfy the same strong genericity and $\std$-regularity hypotheses as the representations $\pi_x$ for $x\in\Sigma$. Moreover, let $\Pi_{x_0}$ be the (unitarily normalized) parabolically induced representation,
\[\Pi_{x_0}=\Ind_{P(\A_F)}^{G(\A_F)}(\chi\boxtimes\sigma).\]
Then there is a nontrivial constituent $V_{x_0}'$ of $\Pi_{x_0,f}^{K_f}$, and another finite dimensional module $V_{x_0}''$ on which $\mb{T}$ acts, such that
\[x_0(\Psi(\phi))=\tr(\phi|V_{x_0}'\oplus V_{x_0}'').\]
Finally, there is a parity condition on the infinitesimal character of $\Pi_{x_0}$ at the archimedean places of $F$ which we do not wish to spell out here.

We remark that sometimes in such Hecke families, the constituents $V_x$ of $\pi_{x,f}^{K_f}$ may often be irreducible under the action of $\mb{T}$, but if $\mc{F}$ degenerates to an strongly $M$-Eisenstein point $x_0$, then $\Psi$ degenerates to the trace of a reducible representation. This is why we have allowed for the second module $V_{x_0}''$. We will remark below that Hecke families which parametrize traces, and not just eigenvalues themselves, can come from a construction of Urban in \cite{urbanev}, and often such families exhibit this behavior under degeneration to an Eisenstein point.

Now the parity condition on the infinitesimal character of $\Pi_{x_0}$ that we mentioned above is put there so that $x_0$ has attached to is a Galois representation $\rho_{x_0}$. After suitable normalizations, it may be described as follows. Let $\rho_\sigma$ be the Galois representation attached to $\sigma$. We will write below $\delta_{\mr{eo}}$ for the integer given by $1$ if $G$ is even orthogonal, and $0$ otherwise. The aforementioned parity condition implies that the Hecke characters for given by $\chi\Vert\cdot\Vert^{(1+\delta_{\mr{eo}}-N)/2}$ and $\chi^*\Vert\cdot\Vert^{(1+\delta_{\mr{eo}}-N)/2}$ are algebraic, where $\chi^*$ is the (conjugate-)dual of $\chi$, i.e., it is the dual unless $G$ is unitary, in which case it is the conjugate-dual. Let $\rho_{\chi\Vert\cdot\Vert^{(1+\delta_{\mr{eo}}-N)/2}}$ and $\rho_{\chi^*\Vert\cdot\Vert^{(1+\delta_{\mr{eo}}-N)/2}}$ denote the respective characters of $G_E$ into $\overline\Q_p^\times$. Then
\[\rho_{x_0}=\rho_{\sigma}(-1)\oplus \rho_{\chi\Vert\cdot\Vert^{(1+\delta_{\mr{eo}}-N)/2}}\oplus\rho_{\chi^*\Vert\cdot\Vert^{(1+\delta_{\mr{eo}}-N)/2}}.\]
Here $\rho_{\sigma}(-1)$ denotes a Tate twist of $\rho_\sigma$.

Given a finite place $v$ of $F$ with $v\nmid 2p$ which is inert or ramified in $E$ (this is only a condition in the unitary case) with $w$ denoting the place of $E$ above $v$, the main theorem below will give a condition on our Hecke family $\mc{F}$ to ensure that each $\rho_x$ for $x\in\Sigma$ is \textit{no more ramified} than $\rho_{x_0}$ at $w$. Let us explain now precisely what we mean by this.

Fix such places $v$ and $w$. Then for each $x\in\Sigma$, we can associate an $N$-dimensional Frobenius-semisimple Weil--Deligne representation
\[(r_x,N_x)=\WD(\rho_x|_{G_{E_w}})^{\varphi-\sss},\]
where $E_w$ is the completion of $E$ at $w$ and $G_{E_w}$ is a corresponding decomposition group. We recall the relevant definitions here in Section \ref{secparams}, suffice it to say $r_x$ is a representation of the Weil group $W_{E_w}$ with open kernel, and $N_x$ is a nilpotent operator on that representation. The operator $N_x$ will be used as our measure of ramification at $w$.

Similarly we have
\[(r_{x_0},N_{x_0})=\WD(\rho_{x_0}|_{G_{E_w}})^{\varphi-\sss}.\]
Then we necessarily have
\[(r_{x_0},N_{x_0})=\WD(\rho_{\sigma}(-1)|_{G_{E_w}})^{\varphi-\sss}\oplus\WD(\rho_{\chi\Vert\cdot\Vert^{(1+\delta_{\mr{eo}}-N)/2}}|_{G_{E_w}})\oplus \WD(\rho_{\chi^*\Vert\cdot\Vert^{(1+\delta_{\mr{eo}}-N)/2}}|_{G_{E_w}}).\]
Of course, the latter two summands above are $1$-dimensional and have no nontrivial monodromy operator.

Let $\tau$ be any irreducible, finite dimensional representation of the inertia group $I_{E_w}$ in $W_{E_w}$ over $\overline\Q_p$ with open kernel. For any such $\tau$, we can consider the $\tau$-isotypic component of $(r_x,N_x)[\tau]=(r_x[\tau],N_x[\tau])$ of $(r_x,N_x)$ for any $x\in\Sigma$, and similarly for $(r_{x_0},N_{x_0})$. If there is a matching of the underlying representations $r_x[\tau]\cong r_{x_0}[\tau]$, then we write $N_x[\tau]\sim N_{x_0}[\tau]$ if there is an invertible linear transformation on either which, after conjugation by this transformation, identifies $N_x[\tau]$ and $N_{x_0}[\tau]$. This condition is what we mean when we say $\rho_{x}$ is no more ramified than $\rho_{x_0}$ at $w$.

We then have the following theorem, which is a (slightly less precise) version of our main result, Theorem \ref{thmmainthm}, below.

\begin{theorem*}
We keep the notation above. So we have our strongly $\Sigma$-generic Hecke family $\mc{F}=(\mb{T},\mf{X},\Sigma,\Psi)$ for $G$ of level $K_f$, and we have a strongly $M$-Eisenstein point $x_0\in\mf{X}(\overline\Q_p)$ on it which has attached to it the parabolically induced representation $\Pi_{x_0}$ from above. We also have our $\mb{T}$-stable spaces $V_x$ for $x\in\Sigma$ and $V_{x_0}=V_{x_0}'\oplus V_{x_0}''$. Fixing a finite place $v$ of $F$ with $v\nmid 2p$ which is inert or ramified in $E$, and letting $w$ be the place of $E$ above it, we then also have the Weil--Deligne representations $(r_x,N_x)$ for $W_{E_w}$ when $x\in\Sigma$, as well as $(r_{x_0},N_{x_0})$.

Then there is a Hecke operator $\phi\in C_c^\infty(G(\A_{F,f}),\overline\Q_p)$ and an eigenvalue $\lambda(\phi)$ for $\phi$ on $\Pi_{x_0,f}$ such that, if $\phi\in\mb{T}$ and $\phi$ acts on $V_{x_0}$ with $\lambda(\phi)$ as an eigenvalue, then we have that
\[r_x[\tau]\cong r_x[\tau]\qquad \textrm{and}\qquad N_x[\tau]\sim N_{x_0}[\tau]\]
for every $x\in\mf{X}(\overline\Q_p)$ outside of a fixed, proper Zariski closed subset of $\mf{X}$, and for every irreducible, finite dimensional representation $\tau$ of $I_{E_w}$ over $\overline\Q_p$ with open kernel.
\end{theorem*}

We remark that actually, Theorem \ref{thmmainthm} below is stated so that it is clear that the construction of the operator $\phi$ can be taken to be purely local at $v$. This allows one to adjoin many such operators for different $v$'s to a spherical Hecke algebra while still allowing the resulting algebra to be commutative.

The analogous statement of the theorem when $G$ is unitary and $v$ is split is easier to show; in fact, all the relevant work to do this already exists in \cite{BCbook} using the types of Schneider--Zink \cite{SZ} discussed above.

\subsubsection*{Ideas of the proof}

We now discuss some of the ideas which go into the proof of the theorem above. The two main points will be to, first, construct links
\[\textrm{Hecke operators}\longleftrightarrow\textrm{Jacquet modules}\longleftrightarrow L\textrm{-parameters}\longleftrightarrow\textrm{Weil--Deligne reps}\]
and then, second, to interpolate the eigenvalues of Frobenius acting on the Weil--Deligne representations $(r_x,N_x)$ and $(r_{x_0},N_{x_0})$, and thus to match these eigenvalues with the eigenvalues of $\phi$ which are interpolated by hypothesis. At this point, a crucial numerical comparison will appear between the possible weights of such eigenvalues and the possible weights of eigenvalues of $\phi$; one will be bounded above and the other below by the same integer $k_0$. These weights thus both coincide with $k_0$, and from there, some combinatorics (albeit nontrivial) on the Weil--Deligne side will complete the proof. Let us describe in more detail how this plays out.

We recall for what follows that the Eisenstein point $x_0$ is induced from a Hecke character $\chi$ for $E$ and an automorphic representation $\sigma$ for $H$. 

The first thing to do is construct the operator $\phi$, which we do locally at $v$ by constructing an operator $\phi_v\in C_c^\infty(G(F_v),\C)$ and setting
\[\phi=\phi_v\otimes 1_{K_f^v},\]
where $K_f^v$ is the factor of $K_f$ away from $v$, and $1_{K_f^v}=\chars(K_f^v)/\vol(K_f^v)$ is the unit in the Hecke algebra for $K_f^v$. We construct $\phi_v$ in Section \ref{sectypes}, and do so in the Hecke algebra for a \textit{type} for $\Pi_{x_0,v}$. Such types are constructed away from residue characteristic $2$ (whence the hypothesis that $v\nmid 2$) by Miyauchi--Stevens \cite{MS}. In fact, they are \textit{covers}, in the sense of Bushnell--Kutzko \cite{BuKu}, of their restrictions to various Levi subgroups. The theory of covers constitutes the first link
\[\textrm{Hecke operators}\longleftrightarrow\textrm{Jacquet modules}\]
above, in the sense that $\phi_v$ will act on any smooth admissible representation $\pi$ of $G(F_v)$ in the same way some simple operators act on the Jacquet module $\Jac_M(\pi)$. We recall here that $\Jac_M(\pi)$ is the smooth admissible representation of $M(F_v)$ given by the $N(F_v)$-coinvariants of $\pi$, where $N$ is the unipotent radical of $P$. These aforementioned simple operators can be made to pick out the value on a uniformizer $\varpi_w$ at $w$ of the $GL_1(E_w)$-factor of the constituents of $\Jac_M(\pi)$.

We next wish to understand the constituents of $\Jac_M(\pi_{v,x})$ for $x\in\Sigma$ in terms of $L$-parameters for such $\pi_{x,v}$, and therefore establish the second link
\[\textrm{Jacquet modules}\longleftrightarrow L\textrm{-parameters}\]
above. Bootstrapping off results of Moeglin \cite{Moeglin} and Moeglin--Renard \cite{MR}, we do this throughout Section \ref{secjacmods}. In Proposition \ref{propjacoftempered}, we write down a finite list, in terms of the local $L$-parameter of $\pi_{x,v}$, of possibilities for $GL_1(E_w)$-factor of any constituent of $\Jac_M(\pi_{v,x})$. It turns they are all of the form
\[\rho\Vert\cdot\Vert_w^{b}\quad\textrm{or}\quad\rho^*\Vert\cdot\Vert_w^{b},\]
where $\rho:E_w^\times\to\C^\times$ is associated by class field theory with a character of $W_{E_w}$ occurring in the $L$-parameter for $\pi_{x,v}$, and $b$ is a half-integer from a finite list with lower bound $\tfrac{N-\delta_{\mr{eo}}-1}{2}$; this is the negative of the half integer that appeared above in describing the Galois representation $\rho_{x_0}$. The possible eigenvalues of $\phi_{v}$ acting on $\pi_{x,v}$ are thus given by these characters evaluated on $\varpi_w$. We note that by purity, these eigenvalues both have weight $-b$ as $q_w$-Weil numbers, where $q_w=\Vert\varpi_w^{-1}\Vert_w$ is the order of the residue field of $E_w$.

Moreover, we show in Corollary \ref{corjacofeisenstein} that $\Jac_M(\Pi_{x_0,v})$ contains the constituents
\[\chi_w\Vert\cdot\Vert_w^{(N-\delta_{\mr{eo}}-1)/2}\boxtimes\sigma_v\quad\textrm{and}\quad \chi_w^{*}\Vert\cdot\Vert_w^{(N-\delta_{\mr{eo}}-1)/2}\boxtimes\sigma_v,\]
and thus $\phi_v$ acts on $\Pi_{x_0,v}$ with eigenvalues among the list
\[\chi_w(\varpi_w)q_w^{(1+\delta_{\mr{eo}}-N)/2}\quad\textrm{and}\quad \chi_w^*(\varpi_w)q_w^{(1+\delta_{\mr{eo}}-N)/2}.\]
Let $k_0$ be the absolute value of the weight of the $q_w$-Weil number $\chi_w(\varpi_w)$. Then one of these two eigenvalues has weight $\tfrac{1+\delta_{\mr{eo}}-N}{2}+k_0$, which is thus $k_0$ greater than the upper bound for the weight of any eigenvalue of $\phi_v$ acting on $\pi_{x,v}$ for $x\in\Sigma$. This comparison of weights is the crucial one alluded to above (shifted for now by the number $\tfrac{1+\delta_{\mr{eo}}-N}{2}$).

Now the final link
\[L\textrm{-parameters}\longleftrightarrow\textrm{Weil--Deligne representations}\]
above is naturally provided by local-global compatibility for the Galois representations attached to $\pi_x$ and $\Pi_{x_0}$; see Section \ref{secgaloisreps}. These Galois representations are interpolated into a family, as we show in Section \ref{secgaloisfamilies}. After passing to a generically finite cover $\mf{Z}$ of $\mf{X}$, we obtain a family $(r_\mf{Z},N_{\mf{Z}})$ of Weil--Deligne representations for $W_{E_w}$ interpolating $(r_x,N_x)$ for $x\in\Sigma$ and $(r_{x_0},N_{x_0})$; see Lemma \ref{lemexistsWDforrhoY} in particular. With the full link
\[\textrm{Hecke operators}\longleftrightarrow\textrm{Weil--Deligne representations}\]
established, we thus find that the analytic function which interpolates the eigenvalues of $\phi_v$ described above is exactly given by an eigenvalue, call it $\lambda_\mf{Z}$, of (geometric) Frobenius on $r_\mf{Z}$. From the numerical comparison above, the weight of $\lambda_{\mf{Z}}$ is constant on $\Sigma$ but then jumps by at least $k_0$ at $x_0$.

In Section \ref{secweights}, we then examine weights in these families combinatorially. After twisting by $\Vert\cdot\Vert_w^{(N-\delta_{\mr{eo}}-1)/2}$, we may assume that the Weil--Deligne representations $(r_x,N_x)$ are pure of weight $0$, and then
\[(r_{x_0},N_{x_0})=(\bar{r}_0,\overline{N}_0)\oplus\chi_w\oplus\chi_w^*,\]
where $r_0$ is pure of weight $0$, and $\chi_w$ and $\chi_w^*$ are viewed as $W_{E_w}$-characters via class field theory. In Proposition \ref{proppossibsforWDdegen}, we conclude two things in light of the weights of these representations: First, we conclude that the weights of any constituent of the representations $r_x$ in this family can only shift by at most $k_0$ up or down upon specialization to $x_0$. Second, we conclude that to have $N_{x_0}\sim N_x$ generically, it suffices to show that there are two character constituents of $r_\mf{Z}$, which are pure of weight $0$ when specialized at generic $x\in\Sigma$, but which degenerate respectively to characters of weights $k_0$ and $-k_0$ at $x_0$. But the weight of the Frobenius eigenvalue $\lambda_\mf{Z}$ from above jumps by at least $k_0$ in this way, hence by exactly $k_0$, and by (conjugate-)duality, another drops by exactly $k_0$, thus implying that indeed $N_{x_0}\sim N_x$ generically.

Intuitively, what is happening is this: These aforementioned character constituents of $r_\mf{Z}$ will degenerate to $\chi_w$ and $\chi_w^*$ at $x_0$, but for the cited combinatorial reasons, cannot be connected to other constituents of $r_\mf{Z}$ via $N_{\mf{Z}}$. And the possibility \textit{a priori} that they are connected in such a way turns out to be the only obstruction to the monodromy shrinking upon specialization at $x_0$.

\subsubsection*{Remark on the endoscopic classification}

Throughout this paper we will need to assume the endoscopic classification of discrete automorphic representations for the classical groups we consider. This is work of many authors, notably culminating in \cite{arthurbook} for quasisplit classical groups, \cite{Mok} for quasisplit unitary groups, \cite{KMSW} for unitary groups beyond the quasisplit case, and later \cite{AGIKMS} filling in many crucial details that had been left. According to \cite[\S 0.4]{AGIKMS}, what is left is the weighted twisted fundamental lemma, and this is in progress; see \textit{loc. cit.} for an account of what is left to be done and who is doing it.

Although the conclusion of our main result, Theorem \ref{thmmainthm}, can be stated without mentioning global Galois representations, we use the endoscopic classification here most notably to construct such Galois representations. The reason for this is that, in Hecke families, it is not immediate that the Weil--Deligne representations associated with the $L$-parameters of $\pi_{x,v}$ for $x\in\Sigma$ at bad places $v$ can be interpolated. We show that they actually can be interpolated by constructing a family of global Galois representations first, and then restricting to the decomposition group and using techniques from \cite{BCbook} to get families of Weil--Deligne representations. This is notable because, during the interpolation process, we never invoke any properties of the Galois representations being interpolated at bad primes, nor at inert or ramified primes in the unitary case. Nevertheless, local-global compatibility at bad finite places for these Galois representations is crucial for us.

We are also implicitly invoking the endoscopic classification in Section \ref{secjacmods}, because we need to describe the Jacquet modules of representations contained in the local $L$-packets constructed in \cite{arthurbook}, \cite{Mok}, or \cite{KMSW}.

\subsubsection*{Acknowledgements}

We thank Chris Skinner for several helpful conversations from when the ideas for this paper were still in their infancy. We also thank Xin Wan for his interest in the method presented here.

\subsubsection*{Notation and conventions}

We fix a prime number $p$ throughout this paper. At the end of Section \ref{secgroups} and from then on, the symbols $G$, $P$, $M$, $H$, $F$, $E$, $n$, and $N$ will have a fixed meaning. Thus we often use the symbol $G'$ whenever discussing other groups, such as below.

Given an algebraic group $G'$ over a field $k$ and an extension field $K/k$, we write $G_{/K}'$ to denote the base change of that group to one over $K$.

Given a reductive group $G'$ over a field of characteristic $0$, we denote by $(G')^\vee$ its dual group, which we always view as a complex Lie group. We also use the symbol $(\cdot)^\vee$ for representations to denote the dual.

If $k$ is a number field, and $v$ is a place of $k$, then $k_v$ denotes the completion of $k$ at $v$.

If $k$ is a number field, then $\A_k$ denotes the adeles of $k$, and $\A_{k,f}$ the finite adeles. In this case, automorphic representations for a reductive group $G'$ over $k$ are viewed in the category of admissible $G'(\A_{k,f})\times(\mf{g}',K_\infty)$-modules which are irreducible by definition; here $\mf{g}$ denotes the complex Lie algebra of $G'(k\otimes_\Q\R)$ and $K_\infty$ denotes a maximal compact subgroup in $G'(k\otimes_\Q\R)$. We make the usual abuse of language and say ``automorphic representation of $G'(\A_k)$'' in this case, despite there being no full action of $G'(k\otimes_\Q\R)$ on such.

If $\pi$ is such an automorphic representation, then $\pi^\vee$ denotes its contragredient. If $v$ is a place of $k$, then $\pi_v$ will denote the smooth admissible representation of $G(k_v)$ given by the local component of $\pi$ at $v$.

Given a number field $k$, the symbol $\Vert\cdot\Vert$ denotes the norm character on $\A_k^\times$, and given a finite place $v$ of $k$, the symbol $\Vert\cdot\Vert_v$ denotes the norm on $k_v$ that sends the inverse of a uniformizer in $k_v$ to the number of elements in the residue field of $k_v$

The group $G$ fixed in Section \ref{secgroups} may or may not be unitary. We reserve the symbol $(\cdot)^*$ to mean conjugate-dual, defined depending on context, if $G$ is unitary; otherwise it denotes the ordinary dual. We correspondingly often write ``(conjugate-)dual,'' or ``(conjugate-)self dual'' to mean that the word ``conjugate'' applies only if $G$ is unitary.

Given a field $k$ of characteristic $0$, we denote its absolute Galois group by $G_k$.

Given a local field $L$ of characteristic $0$, we denote by $W_L$ its Weil group. If $L$ is nonarchimedean, this is the subgroup of the Galois group $G_L$ containing the inertia group $I_L$ of $L$ as a subgroup, as well as only integer powers of a given Frobenius element. It is given the topology where $I_L$ is profinite and open. If $L$ is real then $W_\R=\C^\times\rtimes\{1,j\}$, where $jzj=\bar{z}$ for $z\in\C^\times$. If $L$ is complex then $W_\C=\C^\times$.

Class field theory is normalized in the paper to send uniformizers to geometric Frobenius elements, and the nonarchimedean local Langlands correspondence is normalized correspondingly. Then the norm character $\Vert\cdot\Vert$ is sent to the cyclotomic character.

Despite this, when discussing Weil--Deligne representations $(r,N)$, we will often fix an arithmetic Frobenius $\varphi$. Then $N$ raises weights by $2$. We have tried to be clear about this throughout the paper.

Let $G'$ be a reductive group over a nonarchimedean local field $L$ of characteristic $0$. Let $P'$ be an $L$-parabolic subgroup of $G'$ with Levi factor $M'$ and unipotent radical $N'$. We denote by $\delta_{P'(L)}$ the modulus character of $P'(L)$. Given a smooth admissible representation $\sigma$ of $M'$, we write
\[\Ind_{P'(L)}^{G'(L)}(\pi)\]
for the unitarily normalized parabolic induction. Similar conventions hold in the adelic setting.

Moreover, given a smooth admissible representation $\pi$ of $G'(L)$, the \textit{Jacquet module} of $\pi$ down to $M'$ is given as usual by the $N'(L)$-coinvariants of $\pi$. We denote it by $\Jac_{M'}(\pi)$. It is naturally a representation of $M'(L)$. Some authors introduce a normalizing factor of $\delta_{P'(L)}^{-1/2}$, but we do not.

Given a subset $T$ of a set $S$, we denote by $\chars(T)$ the characteristic function of $T$ in $S$.

Given a rigid space $\mf{X}$ over $\Q_p$, we denote by $\mc{O}(\mf{X})$ the ring of analytic functions on $\mf{X}$, and we denote by $\mf{X}(\overline\Q_p)$ the set of $\overline\Q_p$-points of $\mf{X}$. A subset $\Sigma\subset\mf{X}(\overline\Q_p)$ is said to be Zariski dense if the set of points in $\mf{X}$ over which some $x\in\Sigma$ lies is so. Thus, if $\mf{X}$ is affinoid, then $\mf{X}$ is the maximal spectrum of $\mc{O}(\mf{X})$, and each point $x\in\mf{X}(\overline\Q_p)$ is a $\Q_p$-algebra map $\mc{O}(\mf{X})\to\overline\Q_p$. Each such $x$ thus has a kernel $\mf{m}_x$, which is a point of $\mf{X}$. Then $\Sigma\subset\mf{X}(\overline\Q_p)$ is said to be Zariski dense if $\sset{\mf{m}_x}{x\in\Sigma}$ is so.

\section{Groups}
\label{secgroups}

We start by introducing the classical groups on which we will work throughout this paper. Fix throughout this section a field $k$ of characteristic $0$.

\subsubsection*{Unitary groups}

Fix a quadratic field extension $K/k$. Let $V$ be a nontrivial, nondegenerate Hermitian space over $K$ and $n$ its dimension over $K$. Let $\langle\cdot,\cdot\rangle:V\times V\to K$ be the associated pairing. We write $U(V)$ for the unitary group associated with $V$, defined to be the subgroup of $GL(V)$ consisting of $K$-linear transformations which preserve $\langle\cdot,\cdot\rangle$. It naturally defines a group scheme over $k$.

The group $U(V)$ is of absolute rank $n$, where the center of the group is counted in the rank; in fact, the group $U(V)$ is a form of $GL_{n/k}$. The subgroup of $U(V)$ consisting of transformations of determinant $1$ is denoted by $SU(V)$, which has absolute rank $n-1$.

The space $V$ splits as
\[U\oplus W\oplus U^*,\]
where $U$ and $U^*$ are totally isotropic for $\langle\cdot,\cdot\rangle$ of the same dimension, where $W$ is anisotropic, and where the induced pairing on $U\times U^*$ is totally nondegenerate. Let $n_k$ be the $K$-dimension of $U$. Then $n_k$ is the relative rank of $U(V)$ over $k$, i.e., the rank of the largest $k$-split torus in $U(V)$.

In fact, fix a basis $\{u_1,\dotsc,u_{n_k}\}$ of $U$. Then the subgroup $T$ of $U(V)$ which acts by scalars in $GL_{1/k}$ on each $u_i$, and as the identity on $W$, is a maximal $k$-split torus in $U(V)$. Moreover, for each $i=1,\dotsc,n_k$, we can define a cocharacter $\tilde{e}_i^\vee:GL_{1/K}\to U(V)$ which sends $\gamma\in GL_{1/K}$ to the unique transformation which scales $u_i$ by $\gamma$ and leaves all other $u_j$ for $j\ne i$, along with $W$, fixed. Then the images of the $k$-rational cocharacters $\tilde{e}_i^\vee|_{GL_{1/k}}$, as $i$ ranges through $\{1,\dotsc,n_k\}$, generate $T$.

Let $e_i$ for $i=1,\dotsc,n_k$ be the character of $T$ such that $e_i\circ(\tilde{e}_i^\vee|_{GL_{1/k}})=\delta_{i,j}$ (Kronecker delta). Then the characters $e_i$ span the character group $X^*(T)$ of $T$, which is thus isomorphic to $\Z^{n_k}$.

Let $\Phi$ be the set of roots of $T$ in $U(V)$, so that $\Phi$ forms the relative root system of $U(V)$. If $W$ is nontrivial, then $\Phi$ is nonreduced of type $BC_{n_k}$, and otherwise it is of type $C_{n_k}$. The roots in $\Phi$ are of the form $\pm(e_i+e_j)$ and $\pm(e_i-e_j)$ for $1\leq i<j\leq n_k$, along with $\pm 2e_i$ and, if $W\ne 0$, also $\pm e_i$, for $1\leq i\leq n_k$. A system of simple roots is given by
\[e_1-e_2, e_2-e_3,\dotsc, e_{n_k-1}-e_{n_k}, e_{n_k}.\]
This defines the system of positive roots which is given by the roots of the form $e_i\pm e_j$ for $1\leq i<j\leq n_k$, along with $2e_i$ and, if $W\ne 0$, also $e_i$, for $1\leq i\leq n_k$.

We will need to know the multiplicity, call it $m(\alpha)$, of each positive root listed above in $U(V)$, and we compute this on the level of the Lie algebra. The Lie algebra $\mf{su}(V)$ of $SU(V)$ contains all roots, and may be identified with the subspace of endomorphisms in $\End_K(V)$ which are skew-Hermitian with respect to $\langle\cdot,\cdot\rangle$.

For a given $i$, the space of endomorphisms in $\mf{su}(V)$ which kill each $u_j$ for all $j$, and which map $W$ into $Ku_i$, is the root space for $e_i$. Thus
\[m(e_i)=\dim_k(\hom_K(W,K))=2(n-2n_k);\]
note that we need to take $k$-dimensions instead of $K$-dimensions, whence the factor of $2$ above. Note also that the number above is $0$ if $W=0$.

For a given $i$ and $j$ with $i<j$, space of endomorphisms in $\mf{su}(V)$ which kill $W$ and each $u_l$ for all $l\ne j$, and which map $Ku_j$ to $Ku_i$, is the root space for $e_i-e_j$. Thus
\[m(e_i-e_j)=\dim_k(\hom_K(K,K))=2.\]

Let $u_1^*,\dotsc,u_n^*$ denote the dual basis in $U^*$ to $u_1,\dotsc,u_n$. Then for a given $i$ and $j$ with $i\leq j$, space of endomorphisms in $\mf{su}(V)$ which kill $W$ and each $u_l$ for all $l\ne j$, and which map $Ku_j^*$ to $Ku_i$, is the root space for $e_i+e_j$. Thus
\[m(e_i+e_j)=m(2e_i)=\dim_k(\hom_K(K,K))=2.\]

Finally, we note that the dual group to $U(V)$ is $GL_n(\C)$. A standard maximal torus in $GL_n(\C)$ is given by the usual group of diagonal matrices in $GL_n(\C)$. If $\rho$ then denotes half the sum of positive roots in the system that makes the group of upper triangular matrices in $GL_n(\C)$ a standard Borel subgroup, then
\[\rho=(\tfrac{n-1}{2},\tfrac{n-3}{2},\dotsc,\tfrac{3-n}{2},\tfrac{1-n}{2}),\]
where the $n$-tuple above is identified with a weight of the diagonal torus in the usual way, the $i$th entry corresponding to the $i$th diagonal entry of that torus for all $i=1,\dotsc,n$.

\subsubsection*{Odd special orthogonal groups}

Let $n\geq 1$ be an integer. Let $J_{2n+1}$ denote the $(2n+1)\times (2n+1)$ matrix with $1$'s along the anti-diagonal and $0$'s elsewhere:
\[J_{2n+1}=\pmat{&&1\\ &\iddots &\\ 1&&}.\]
We write $SO_{2n+1}$ for the split odd special orthogonal group over $k$ defined by
\[SO_{2n+1}=\sset{g\in GL_{2n+1}}{g J_{2n+1} \tp{g}=J_{2n+1},\,\,\det(g)=1}.\]
This definition presents $SO_{2n+1}$ as a subgroup of $GL_{2n+1}$ and thus defines a natural \textit{standard representation} $\std:SO_{2n+1}\to GL_{2n+1}$, given by inclusion.

The rank of $SO_{2n+1}$ is $n$ and the root system for $SO_{2n+1}$ is of type $B_n$ (the root at the end of the Dynkin diagram is shorter than the ones preceding it). In fact, a maximal torus $T$ in $SO_{2n+1}$ is given by
\[T=\sset{\diag(t_1,\dotsc,t_n,1,t_n^{-1},\dotsc,t_1^{-1})}{t_1,\dotsc,t_n\in GL_1}\subset SO_{2n+1}.\]
We can then define characters $e_i:T\to GL_1$ for $i=1,\dotsc,n$ by
\[e_i(\diag(t_1,\dotsc,t_n,1,t_n^{-1},\dotsc,t_1^{-1}))=t_i.\]
These characters span the character group $X^*(T)$ of $T$, which is thus isomorphic to $\Z^n$ under the identification
\[a_1 e_1+\dotsb+a_n e_n\mapsto (a_1,\dotsc,a_n),\]
for $a_1,\dotsc,a_n\in\Z$. We will therefore often identify weights in $X^*(T)\otimes\Q$ with $n$-tuples of rational numbers via this identification.

Let $\Phi$ be the set of roots of $T$ in $SO_{2n+1}$. Then the roots in $\Phi$ are those characters of the form $\pm(e_i+e_j)$ and $\pm(e_i-e_j)$ for $1\leq i<j\leq n$, along with $\pm e_i$ for $1\leq i\leq n$. A system of simple roots is given by
\[e_1-e_2, e_2-e_3,\dotsc, e_{n-1}-e_n, e_n.\]
This defines the system of positive roots which is given by the roots of the form $e_i\pm e_j$ for $1\leq i<j\leq n$, along with $e_i$ for $1\leq i\leq n$. This system makes the standard Borel in $SO_{2n+1}$ upper triangular. It follows that if $\rho$ denotes the half sum of positive roots in $\Phi$, then
\[\rho=(n-\tfrac{1}{2},n-\tfrac{3}{2},\dotsc,\tfrac{3}{2},\tfrac{1}{2}),\]
using the identification $X^*(T)\otimes\Q\cong \Q^n$ explained above.

Finally, we note that the dual group to $SO_{2n+1}$ is the symplectic group $Sp_{2n}(\C)$, which we define just below.

\subsubsection*{Symplectic groups}

Let $n\geq 1$ be an integer. Write $1_n$ for the $n\times n$ identity matrix. Let $\tilde{J}_{2n+1}$ denote the $2n\times 2n$ matrix given by
\[\tilde{J}_{2n}=\pmat{&1_n\\ -1_n&}.\]
Then we write $Sp_{2n}$ for the symplectic group over $k$ defined by
\[Sp_{2n}=\sset{g\in GL_{2n}}{g \tilde{J}_{2n} \tp{g}=\tilde{J}_{2n}}.\]
This definition presents $Sp_{2n}$ as a subgroup of $GL_{2n}$ and thus defines a natural \textit{standard representation} which, like in the odd special orthogonal case above, we continue to denote by $\std:Sp_{2n}\to GL_{2n}$, and which is given by inclusion.

The group $Sp_{2n}$ is split and its rank is $n$. The root system for $Sp_{2n}$ is of type $C_n$ (the root at the end of the Dynkin diagram is longer than the ones preceding it). In fact, a maximal torus $T$ in $Sp_{2n}$ is given by
\[T=\sset{\diag(t_1,\dotsc,t_n,t_1^{-1},\dotsc,t_n^{-1})}{t_1,\dotsc,t_n\in GL_1}\subset Sp_{2n}.\]
We can then define characters $e_i:T\to GL_1$ for $i=1,\dotsc,n$ by
\[e_i(\diag(t_1,\dotsc,t_n,t_1^{-1},\dotsc,t_n^{-1}))=t_i.\]
These characters span the character group $X^*(T)$ of $T$, which is thus isomorphic to $\Z^n$ under the identification
\[a_1 e_1+\dotsb+a_n e_n\mapsto (a_1,\dotsc,a_n),\]
for $a_1,\dotsc,a_n\in\Z$. We will therefore often identify weights in $X^*(T)\otimes\Q$ with $n$-tuples of rational numbers via this identification.

Let $\Phi$ be the set of roots of $T$ in $Sp_{2n}$. Then the roots in $\Phi$ are those characters of the form $\pm(e_i+e_j)$ and $\pm(e_i-e_j)$ for $1\leq i<j\leq n$, along with $\pm 2e_i$ for $1\leq i\leq n$. A system of simple roots is given by
\[e_1-e_2, e_2-e_3,\dotsc, e_{n-1}-e_n, 2e_n.\]
This defines the system of positive roots which is given by the roots of the form $e_i\pm e_j$ for $1\leq i<j\leq n$, along with $2e_i$ for $1\leq i\leq n$. It follows that if $\rho$ denotes the half sum of positive roots in $\Phi$, then
\[\rho=(n,n-1,\dotsc,2,1),\]
using the identification $X^*(T)\otimes\Q\cong \Q^n$ explained above.

Finally, we note that the dual group to $Sp_{2n}$ is the odd special orthogonal group $SO_{2n+1}(\C)$ defined above.

\subsubsection*{Even special orthogonal groups}

Let $n\geq 1$ be an integer, and let $\eta:G_k\to\{\pm 1\}$ be a continuous quadratic character of the absolute Galois group of $k$. We allow $\eta$ to be the trivial character $1$. If $\eta$ is nontrivial, let $k^\eta/k$ be the quadratic extension cut out by $\eta$, and choose an element $a\in k$ such that $k^\eta=k(\sqrt{a})$. Then define a $2\times 2$ matrix $J_2^\eta$ by
\[J_2^\eta=\pmat{1 &\\ & -a}\,\,\textrm{if }\eta\ne 1,\quad\textit{or}\quad J_2^\eta=\pmat{& 1 \\ 1 &}\,\,\textrm{if }\eta=1.\]
Let $J_{n-1}$ denote the $(n-1)\times (n-1)$ matrix with $1$'s along the anti-diagonal and $0$'s elsewhere:
\[J_{n-1}=\pmat{&&1\\ &\iddots &\\ 1&&}.\]
Then finally define $J_{2n}^\eta$ to be the $2n\times 2n$ matrix given in block form by
\[J_{2n}^\eta=\pmat{&& J_{n-1}\\ &J_2^\eta &\\ J_{n-1}&&}.\]
Note that if $n=1$, then we have given two definitions of $J_2^\eta$, and they coincide.

We define the even special orthogonal group $SO_{2n}^\eta$ over $k$ by
\[SO_{2n}^\eta=\sset{g\in GL_{2n}}{g J_{2n}^\eta \tp{g}=J_{2n}^\eta,\,\,\det(g)=1}.\]
It is a subgroup of the orthogonal group $O_{2n}^\eta$ defined by
\[O_{2n}^\eta=\sset{g\in GL_{2n}}{g J_{2n}^\eta \tp{g}=J_{2n}^\eta}.\]
This definition presents both $SO_{2n}^\eta$ and $O_{2n}^\eta$ as subgroups of $GL_{2n}$ and thus defines natural \textit{standard representations} which, like in the cases above, we continue to denote by $\std:SO_{2n}^\eta\to GL_{2n}$ or $\std:O_{2n}^\eta\to GL_{2n}$, and which is given by inclusion. In both cases, we omit the superscript $\eta$ when $\eta=1$.

The group $SO_{2n}^\eta$ is quasi-split and it is split if and only if $\eta=1$; otherwise it splits over $k^\eta$. Its absolute rank is $n$ in any case, and its relative rank is $n-1$ if $\eta\ne 1$. The relative root system for $SO_{2n}^{\eta}$ is of type $D_n$ if $\eta=1$, and otherwise it is of type $B_{n-1}$. In fact, a maximal split torus $T$ in $SO_{2n}^\eta$ is given by
\[T=\sset{\diag(t_1,\dotsc,t_n,t_n^{-1},\dotsc,t_1^{-1})}{t_1,\dotsc,t_n\in GL_1}\subset SO_{2n}^\eta,\textrm{ if }\eta=1,\]
and otherwise by
\[T=\sset{\diag(t_1,\dotsc,t_{n-1},1,1,t_{n-1}^{-1},\dotsc,t_1^{-1})}{t_1,\dotsc,t_{n-1}\in GL_1}\subset SO_{2n}^\eta,\textrm{ if }\eta\ne 1,\]

Let us write
\[n_\eta=\begin{cases}
n&\textrm{if }\eta=1;\\
n-1&\textrm{otherwise}
\end{cases}\]
for the relative rank of $SO_{2n}^\eta$. We can then define characters $e_i:T\to GL_1$ for $i=1,\dotsc,n_\eta$ by
\[e_i(\diag(t_1,\dotsc,t_n,t_n^{-1},\dotsc,t_1^{-1}))=t_i,\textrm{ if }\eta=1,\]
and otherwise similarly by
\[e_i(\diag(t_1,\dotsc,t_{n-1},1,1,t_{n-1}^{-1},\dotsc,t_1^{-1}))=t_i,\textrm{ if }\eta\ne 1,\]
These characters span the character group $X^*(T)$ of $T$, which is thus isomorphic to $\Z^{n_\eta}$ under the identification
\[a_1 e_1+\dotsb+a_{n_\eta} e_{n_\eta}\mapsto (a_1,\dotsc,a_{n_\eta}),\]
for $a_1,\dotsc,a_{n_\eta}\in\Z$. We will therefore often identify weights in $X^*(T)\otimes\Q$ with $n_\eta$-tuples of rational numbers via this identification.

Let $\Phi$ be the set of roots of $T$ in $SO_{2n}^\eta$. If $\eta=1$, then the roots in $\Phi$ are those characters of the form $\pm(e_i+e_j)$ and $\pm(e_i-e_j)$ for $1\leq i<j\leq n$; Otherwise, if $\eta\ne 1$, then the roots in $\Phi$ are those characters of the form $\pm(e_i+e_j)$ and $\pm(e_i-e_j)$ for $1\leq i<j\leq n-1$, along with $\pm e_i$ for $1\leq i\leq n-1$. This makes sense as the latter set of characters is obtained from the former by setting $e_n\mapsto 0$ everywhere.

We will need to know the multiplicity $m(\alpha)$ of these roots $\alpha$ in the group $SO_{2n}^\eta$. In fact, when $\eta\ne 1$ these multiplicities can also be obtained informally by setting $e_n\mapsto 0$ from the split case. So when $\eta=1$, of course all roots have multiplicity $1$. Otherwise, when $\eta\ne 1$, we have that $m(e_i\pm e_j)=1$ for $1\leq i<j\leq n-1$, and $m(e_i)=2$ for $1\leq i\leq n-1$ (since $e_i$ comes from both $e_i+e_n$ and $e_i-e_n$ by setting $e_n\mapsto 0$).

A system of simple roots is given by
\[e_1-e_2, e_2-e_3,\dotsc, e_{n-1}-e_n, e_{n-1}+e_n,\textrm{ if }\eta=1,\]
or otherwise by
\[e_1-e_2, e_2-e_3,\dotsc, e_{n-2}-e_{n-1}, e_{n-1},\textrm{ if }\eta\ne 1.\]
This defines the system of positive roots which is given by the roots of the form $e_i\pm e_j$ for $1\leq i<j\leq n_\eta$, along with $e_i$ for $1\leq i\leq n$ when $\eta\ne 1$.

In the case that $\eta=1$, it follows that if $\rho$ denotes the half sum of positive roots in $\Phi$, then
\[\rho=(n-1,n-2,\dotsc,1,0),\]
using the identification $X^*(T)\otimes\Q\cong \Q^n$ explained above.

Finally, we note that the dual group to $SO_{2n}^\eta$ is itself the even special orthogonal group $SO_{2n}(\C)$.

\subsubsection*{Weyl groups}

All of the $k$-groups just defined have relative root systems of type $B$, $C$, $BC$, or $D$. We make some comments about the Weyl groups of these root systems. Let $n$ be a positive integer, and let $\Phi$ be an irreducible root system of rank $n$ of type $B$, $C$, or $BC$, contained in a character group isomorphic to $\Z^n$ (we remark on type $D$ later below). For $i=1,\dotsc,n$, write $e_i$ for the character corresponding to the element $(0,\dotsc,0,1,0,\dotsc,0)\in\Z^n$, where the $1$ is in the $i$th coordinate. Let $W=W(\Phi)$ be the Weyl group of $\Phi$.

The group $W$ is generated by simple reflections. For each $i=1,\dotsc,n$, it contains the reflection $(-)_i$ across $e_i$, which fixes each $e_j$ for $j\ne i$ and sends $e_i$ to $-e_i$. The reflection $(-)_n$ is simple. The group $W$ also contains any permutation in the symmetric group $S_n$, acting in the obvious way by permuting the $e_i$'s. Then the permutations $(12),(23),\dotsc,(n-1,n)$ are all simple.

It follows that the Weyl group is the semidirect product $\{\pm\}^{n}\rtimes S_n$, where the minus sign in the $i$th factor of the group $\{\pm\}^{n}$ acts via $(-)_i$, and where $S_n$ acts naturally by permuting the coordinates in $\{\pm\}^{n}$.

Now let $\Phi'$ be the root system of the same type as $\Phi$ but one rank lower, obtained from $\Phi$ by omitting all roots which contain a multiple of $e_1$ in the basis $e_1,\dotsc,e_n$. Let $W'$ be the Weyl group of $\Phi'$, viewed in the natural way as the subgroup of elements in $W$ which fix $e_1$.

We will have occasion in Section \ref{secjacmods} to consider the set $[W/W']$ of minimal length representatives in $W$ of $W/W'$. The group $W'$ is generated by the same reflections as $W$ except that we omit the reflection $(12)$ about first simple root $e_1-e_2$, or instead the reflection $(-)_1$ about the smallest multiple of $e_1$ if $n=1$. It is then not too difficult to see that
\begin{multline*}
[W/W']=\{1,(21),(321),\dotsc,(n,n-1,\dotsc,1),(-)_n(n,n-1,\dotsc,1),\\
(n-1,n)(-)_n(n,n-1,\dotsc,1),\dotsc,(1,2,\dotsc,n)(-)_n(n,n-1,\dotsc,1)\}.
\end{multline*}
These are the shortest elements which, in order, send $e_1$ to the characters
\[e_1,e_2,e_3\dotsc,e_n,-e_n,-e_{n-1},\dotsc,-e_1.\]
Their lengths increase by $1$ consecutively in the order written above.

If instead $W$ is of type $D$ and $n\geq 2$, then the situation is similar except that $W$ does not contain the reflections $(-)_i$. Instead, since $e_{n-1}+e_n$ is the final simple root, the Weyl group $W$ contains the elements $(-)_i(-)_j$ for $1\leq i<j\leq n$. Thus, if we let $\Sigma^{(n)}$ be the subgroup of elements $\sigma$ in $\{\pm\}^{n}$ for which the product all entries of $\sigma$ is $+$, then naturally $W=\Sigma^{(n)}\rtimes S_n$. Now, the element $(-)_{n-1}(-)_n$ is a simple reflection.

In this case, when $W'$ is defined similarly as above, we have
\begin{multline*}
[W/W']=\{1,(21),(321),\dotsc,(n,n-1,\dotsc,1),(-)_{n-1}(-)_n(n,n-1,\dotsc,1),\\
(n-1,n)(-)_{n-1}(-)_n(n,n-1,\dotsc,1),\dotsc,(1,2,\dotsc,n)(-)_{n-1}(-)_n(n,n-1,\dotsc,1)\}.
\end{multline*}
Again, these are the shortest elements of $W$ moving $e_1$ in the manner described above.

\subsubsection*{The group $G$ in this paper}

Throughout this paper we fix a totally real number field $F$ and take $G$ to be any one of the four groups described above over $k=F$; in the case that $G$ is unitary, we denote the quadratic field used to define $G$ by $E$ (instead of $K$ as above), and we assume that $E/F$ is imaginary quadratic. Otherwise, we write $E=F$ throughout if $G$ is not unitary. We also assume in all cases that the relative root system of $G$ over $F$ is nonempty; thus, in the unitary case, the Hermitian space over $F$ defining $G$ cannot be anisotropic, and in the even orthogonal case, the group $G$ cannot be a form of $SO_2$.

With this assumption, we fix the relative root system $\Phi$ for the maximal split torus $T$ in $G$ exactly as described in all cases above. Then $\Phi$ contains a first simple root; it is given by $e_1-e_2$ in the notations above unless the relative rank of $G$ is $1$, in which case it is given by $e_1$ or $2e_1$, depending respectively on whether $G$ is odd orthogonal or not.

Therefore $G$ contains a standard parabolic subgroup which we call $P$, and whose Levi factor we denote by $M$ and whose unipotent radical we denote by $N$, which is obtained by omitting this first simple root. Thus $M$ contains all the roots which do not contain a multiple of $e_1$ in the basis $e_1,\dotsc,e_n$; the root system for $T$ in $M$ is therefore what we called $\Phi'$ in the discussion on Weyl groups above. In all cases we have that $M$ decomposes as
\[M=GL_{1/E}\times H,\]
for some group $H$ over $F$ of the same type as $G$ but with $F$-rank one lower than $G$.

Throughout we will denote by $n$ the absolute rank of $G$ (center included), i.e., the dimension of the maximal torus in $G(\overline{F})$. We also denote by $N$ the dimension of the standard representation of the dual group $G^\vee$ of $G$. Thus,
\[N=\begin{cases}
n&\textrm{if }G\textrm{ is unitary;}\\
2n&\textrm{if }G\textrm{ is odd orthogonal;}\\
2n+1&\textrm{if }G\textrm{ is symplectic;}\\
2n&\textrm{if }G\textrm{ is even orthogonal.}
\end{cases}\]
We believe no confusion will result from the notational overlap between this number and the unipotent radical of $P$.

\section{Parameters}
\label{secparams}

Let $G$ be the group fixed just above in Section \ref{secgroups}, and let $F$ and $E$ be the number fields from there as well. Recall that $n$ is the absolute rank of $G$ and $N$ is the dimension of the standard representation of $G^\vee$. This section will discuss global Arthur parameters for $G$, as well as the local parameters obtained from them.

\subsubsection*{Global parameters}

As noted in the introduction, this paper will make signifincant use of the endoscopic classification of discrete automorphic representations of the group $G$ due to many authors, especially \cite{arthurbook}, \cite{Mok}, \cite{KMSW}, and \cite{AGIKMS}. However, we will not need to be too precise about the exact parity and duality conditions which define the class of global parameters relevant for each possibility for the group $G$, nor the exact multiplicity formulas that these authors obtain. For our purposes we will be content to know the following.

Let $\pi$ be an automorphic representation of $G(\A_F)$ which occurs in the discrete spectrum for $G$,
\[\pi\subset L_{\disc}^2(G(F)\backslash G(\A_F)).\]
Then there exists a \textit{global Arthur parameter} $\psi$ (or possibly many) which, for us, is a concatenation of symbols of the following form:
\[\psi=\boxplus_{i=1}^r\widetilde\pi_i[d_i];\]
Here, each $\widetilde\pi_i$ is a cuspidal automorphic representation of $GL_{n_i}(\A_E)$ for some positive integer $n_i$, and each $d_i$ is a positive integer such that $\sum_{i=1}^r n_id_i=N$. Each $\widetilde\pi_i$ is required to be self dual if $G$ is not unitary. As mentioned, the $\widetilde\pi_i$'s and the $d_i$'s are required to satisfy various other conditions which will not concern us here, except that $\psi$ must be \textit{conjugate-self dual} if $G$ is unitary. This means the following.

Let $\widetilde{\pi}_i^{\vee}$ denote the contragredient of $\widetilde{\pi}_i$. Define $\widetilde{\pi}_i^{*}(g)=\widetilde{\pi}_i^{\vee}(c(g))$, where $c$ is the nontrivial element in $\gal(E/F)$. Then we say $\psi$ is (conjugate-)self dual if we have that the multisets of pairs $\{(\widetilde\pi_i,d_i)\}$ and $\{(\widetilde\pi_i^*,d_i)\}$ are the same. Thus the parameter
\[\psi^*=\boxplus_{i=1}^r\widetilde\pi_i^*[d_i]\]
may be rearranged to obtain $\psi$.

Below we will describe how, for any place $v$ of $F$, to associate with $\psi$ a local $A$-parameter $\psi_v$. The endoscopic classification provides, in particular, a \textit{local Arthur packet}, which is a finite set $\Pi_{\psi_v}$ consisting of irreducible admissible representations which are smooth if $v$ is nonarchimedean. This packet itself satisfies a number of properties relevant for endoscopy. For us, it suffices to know that the local component $\pi_v$ of $\pi$ at $v$ occurs in this packet $\Pi_{\psi_v}$. In Section \ref{secjacmods}, however, under certain circumstances we will need to compute the Jacquet modules associated with various representations in such packets at nonarchimedean places $v$. Fortunately, we will be able to retroactively offload most of this burden to other authors there, most notably \cite{Moeglin} and \cite{MR}, who will give us enough explicit information about such representations to be able to do this.

\subsubsection*{Local parameters}

We first recall the following definition.

\begin{definition}
Let $L$ be a nonarchimedean local field of characteristic $0$ and $G'$ a reductive group over $L$. Let $W_L$ be the Weil group of $L$ and let $\L{G}'=(G')^\vee\rtimes W_L$ be the $L$-group, where $(G')^\vee$ is, as usual, the complex dual group of $G'$. In the following we denote by $SL_{2,\mr{D}}(\C)$ and $SL_{2,\mr{A}}(\C)$ two different copies of $SL_2(\C)$ (respectively the \textit{Deligne} $SL_2$ and the \textit{Arthur} $SL_2$).

A (local) \textit{Arthur parameter}, or $A$\textit{-parameter}, for $G'$ is a homomorphism
\[\psi:W_L\times SL_{2,\mr{D}}(\C)\times SL_{2,\mr{A}}(\C)\to \L{G}',\textrm{ if }L\textrm{ is nonarchimedean},\]
or
\[\psi:W_L\times SL_{2,\mr{A}}(\C)\to \L{G}',\textrm{ if }L\textrm{ is archimedean},\]
considered up to $(G')^\vee$-conjugation, satisfying the following properties:
\begin{enumerate}[label=(\alph*)]
\item The restriction of $\psi$ to $W_L\times SL_{2,\mr{D}}(\C)$ in the nonarchimedean case (respectively, to $W_L$ in the archimedean case) is a tempered $L$-parameter;
\item The restriction of $\psi$ to $SL_{2,\mr{A}}(\C)$ is algebraic.
\end{enumerate}

If $\psi$ is an $A$-parameter for $G'$, we define the associated $L$-parameter $\phi_\psi$ by
\[\phi_\psi(w)=\psi\left(w\times\pmat {\vert w\vert^{1/2} & \\ & \vert w\vert^{-1/2}}\right),\]
where the notation $\vert w\vert$ is explained as follows: When $L$ is archimedean, we write $\vert\cdot\vert:W_L\to\R_{>0}$ for the character given by the usual absolute value on $\C^\times\subset W_L$ and, when $L=\R$ is moreover real, we define $\vert j\vert=1$ as well. Otherwise, when $L$ is nonarchimedean, write $\lambda$ for the number of elements in the residue field of $L$, and fix an arithmetic Frobenius element $\varphi\in L$. Then let $\vert\cdot\vert:W_L\to\R_{>0}$ be the character which is trivial on the inertia group $I_L$ in $W_L$, and such that $\vert\varphi\vert=\lambda$. We then extend $\vert\cdot\vert$, using the same notation, to $W_L\times SL_{2,\mr{D}}(\C)$ by declaring it to be trivial on $SL_{2,\mr{D}}(\C)$.
\end{definition}

The reader interested in the details involved in the general definitions of the $L$-group, as well as $L$-parameters may consult Borel's article \cite[\S 2, \S 8]{BorelCorv}. We remark that in \textit{loc. cit.}, in the nonarchimedean case, the group $W_L\times SL_{2,\mr{D}}(\C)$ used above is replaced by the \textit{Weil--Deligne group} of $L$ when discussing $L$-parameters. The Weil--Deligne group of $L$ is a certain semidirect product $W_L\ltimes\C$ and it may be viewed as a subgroup of $W_L\times SL_{2,\mr{D}}(\C)$ via
\[w\ltimes u\mapsto \left(w\times\pmat{\vert w\vert^{1/2} & \\ & \vert w\vert^{-1/2}}\right)\cdot\left(1\times\pmat{ 1 & u\\ & 1}\right).\]
An $L$-parameter for the Weil--Deligne group may then be viewed as the restriction of an $L$-parameter in the sense above to the Weil--Deligne group via this map, and these points of view are equivalent.

We remark that when a reductive group $G'$ is split over a local field $L$ of characteristic $0$, then $\L{G}'$ is the direct product $\L{G}'=W_L\times (G')^\vee$, and so an $A$-parameter $\psi$ or an $L$-parameter $\phi$ for $G'$ may be viewed as a homomorphism into $(G')^\vee$ instead of $\L{G}'$. We will often do this when $G'$ is $GL_d$ for some $d$.

In the setting of our classical group $G$ over the number field $F$, let $v$ be a place of $F$ and $w$ a place of $E$ above $v$. Assume we are given a global parameter $\psi=\boxplus_{i=1}^r\widetilde\pi_i[d_i]$ as above, with each $\widetilde\pi_i$ an automorphic representation of $GL_{n_i}(\A_E)$ for some $n_i$. Then we can form a local $A$-parameter $\psi_v$ for $G(F_v)$ attached to $\psi$, using the standard representation $\std$ of dimension $N$ for $G^\vee$ discussed in Section \ref{secgroups}, as follows.

First, let $w$ be a place of $E$ above $v$. Then the local Langlands correspondence of Harris--Taylor \cite{HT} and Henniart \cite{Henn} in the nonarchimedean case, or that of Langlands \cite{LanglandsArchLL} in the archimedean case (which works for arbitrary reductive groups, not just general linear groups), attaches to the local component $\widetilde\pi_{i,w}$ of $\widetilde\pi_i$ at $w$ an $L$-parameter
\[\tilde\phi_{i,w}:W_{E_w}\times SL_{2,\mr{D}}(\C)\to GL_{n_i}(\C)\textrm{ if }v\textrm{ is nonarchimedean},\]
or
\[\tilde\phi_{i,w}:W_{E_w}\to GL_{n_i}(\C)\textrm{ if }v\textrm{ is archimedean},\]
for $GL_{n_i}(E_w)$. Then we define an $A$-parameter $\tilde\psi_w$ for $GL_N(E_w)$ by
\[\tilde\psi_w=\bigoplus_{i=1}^r\tilde\phi_{i,w}\boxtimes\Sym^{d_i-1}\]
where
\[\Sym^{d_i-1}:SL_{2,\mr{A}}(\C)\to GL_{d_i}(\C)\]
is the $(d_i-1)$th symmetric power representation, and we view $\Sym^0$ as the trivial representation of $SL_{2}(\C)$ into $GL_1(\C)$.

Next, consider the index $2$ subgroup of $\L{G(F_v)}$ given by $G^\vee(\C)\rtimes W_{E_w}$. The action of $W_{E_w}$ on $G^\vee(\C)$ here is trivial unless $G=SO_{2n}^\eta$ with $\eta|_{W_{F_v}}\ne 1$; in this case, if we let $F_{v}^{\eta}$ be the quadratic extension of $F_v$ cut out by $\eta$, then the action of $W_{F_v^\eta}$ on $G^\vee(\C)$ is trivial, and there is a natural isomorphism
\begin{equation}
\label{eqnstdofLgps}
\L{G(F_v)}/(\{ 1\}\times W_{F_{v}^{\eta}})=SO_{2n}(\C)\rtimes(W_{F_v}/W_{F_{v}^{\eta}})\cong O_{2n}(\C).
\end{equation}
In all cases, there is a natural embedding which we write as
\[\std:G^\vee(\C)\rtimes W_{E_w}\hookrightarrow GL_N(\C)\times W_{E_w}\]
by abuse of notation; in the case when $G=SO_{2n}^\eta$ with $\eta|_{W_{F_v}}\ne 1$, this embedding is induced by the standard representation $\std:O_{2n}(\C)\to GL_{2n}(\C)$ along with the identity map on $W_{F_v}$, and otherwise it is induced by the standard representation $\std:G^\vee(\C)\to GL_N(\C)$ along with the identity map on $W_{E_w}$.

Then it turns out (and this is part of the endoscopic classification) that there is a unique $A$-parameter $\psi_v$ for $G(F_v)$ such that
\begin{equation}
\label{eqnpsiwtilde}
\tilde\psi_w=\begin{cases}
\std\circ(\psi_v|_{W_{E_w}\times SL_{2,\mr{D}}(\C)\times SL_{2,\mr{A}}(\C)})&\textrm{if }v\textrm{ is nonarchimedean;}\\
\std\circ(\psi_v|_{W_{E_w}\times SL_{2,\mr{D}}(\C)})&\textrm{otherwise.}
\end{cases}
\end{equation}
We take this $\psi_v$ to be the localization of the global parameter $\psi$ we started with. This parameter will be self dual if $G$ is not unitary, meaning that the dual representation $\tilde\psi_v^\vee$ of $\tilde\psi_v$ is conjugate to $\tilde\psi_v$ itself by $G^\vee(\C)$. Otherwise, when $G$ is unitary, we have the relation
\[\tilde\psi_w^\vee=\tilde\psi_{c(w)}\circ(c(\cdot)c),\]
where $c\in G_F$ is a complex conjugation, and conjugation by $c$ is interpreted as trivial on any $SL_2(\C)$ factor of the source group of $\tilde\psi_{c(w)}$.

As alluded to above, the endoscopic classification attaches a local Arthur packet $\Pi_{\psi_v'}$ to every local $A$-parameter $\psi_v'$ for $G(F_v)$, and these packets satisfy the following: If $\pi$ is an automorphic representation of $G(\A_F)$ occurring in the discrete spectrum for $G$ and $\psi$ is a global Arthur parameter for $\pi$, then for every place $v$ of $F$, the local component $\pi_v$ of $\pi$ at $v$ occurs in the local Arthur packet $\Pi_{\psi_v}$ attached to the localization $\psi_v$ of $\psi$ at $v$;
\[\pi_v\in\Pi_{\psi_v}.\]
Moreover, for each $v$ and each place $w$ of $E$ lying over $v$, let $\widetilde{\pi}_w$ be the representation of $GL_N(E_w)$ attached to the $L$-parameter $\phi_{\tilde\psi_w}$ by the local Langlands correspondence, where $\tilde\psi_w$ is defined by the right hand side of \eqref{eqnpsiwtilde}. Then the representation
\[\widetilde\pi=\sideset{}{'}{\bigotimes}_{w}\widetilde\pi_w\]
is an $L^2$ automorphic representation of $GL_N(\A_E)$. In fact, if $\psi=\boxplus_{i=1}^r\widetilde\pi_i[d_i]$, then $\widetilde\pi$ is actually the isobaric sum of the $\pi_i$'s, each counted with multiplicity $d_i$.

\begin{remark}
To avoid potential confusion later, we note the following. Assume $G$ is unitary and $v$ is split in $E$. There is a natural inclusion
\[\Delta:G(F_v)\hookrightarrow G(F_v\otimes_F E).\]
Letting $w$ and $w'$ be the places of $E$ above $v$, we then get a natural isomorphism
\[j:G(F_v\otimes_F E)\overset{\sim}{\longrightarrow}G(E_w)\times G(E_{w'}).\]
Let
\[p_w:G(E_w)\times G(E_{w'})\to G(E_w)\quad\textrm{and}\quad p_{w'}:G(E_w)\times G(E_{w'})\to G(E_{w'})\]
be the respective projection maps. Then there are natural isomorphisms
\[\iota_w:G(E_w)\overset{\sim}{\longrightarrow} GL_N(E_w)\quad\textrm{and}\quad \iota_{w'}:G(E_{w'})\overset{\sim}{\longrightarrow} GL_N(E_{w'})\]
coming from the fact that $G_{/E}\cong GL_{N/E}$; this latter isomorphism is due to the presentation of a unitary group as a subgroup of $GL(V)$ for some $E$-vector space $V$. Finally, the isomorphisms $F_v\cong E_w$ and $F_v\cong E_{w'}$ coming from the inclusion $F\hookrightarrow E$ induce natural isomorphisms
\[\nu_{w}:GL_N(E_{w})\overset{\sim}{\longrightarrow}GL_N(F_v)\quad\textrm{and}\quad \nu_{w'}:GL_N(E_{w'})\overset{\sim}{\longrightarrow}GL_N(F_v).\]
Altogether, composing in this order, we get two self-isomorphisms of $G(F_v)$,
\[\nu_{w}\circ\iota_{w}\circ p_{w}\circ j\circ\Delta \quad\textrm{and}\quad \nu_{w'}\circ\iota_{w'}\circ p_{w'}\circ j\circ\Delta.\]
If $G(F_v)$ is then identified with $GL_N(F_v)$ in any fixed way, then these isomorphisms differ by the outer automorphism $\tp{(\cdot)}^{-1}$ of $GL_N$.

Such an identification also allows us to view $\L{G(F_v)}\cong GL_N(\C)\times W_{F_v}$, but the isomorphisms $\iota_{w}$ and $\iota_{w'}$ are what we use to view $\L{G(E_w)}\cong GL_N(\C)\times W_{E_w}$ and $\L{G(E_{w'})}\cong GL_N(\C)\times W_{E_{w'}}$. Thus, it follows that any parameters $\tilde\psi_w$ and $\tilde\psi_{w'}$ satisfying \eqref{eqnpsiwtilde} for $w$ and $w'$, respectively, give rise together to a representation of $G(F_v\otimes_F E)$ which is conjugate-self dual.
\end{remark}

\subsubsection*{Infinitesimal characters and cohomological representations}

We recall a bit of the theory of infinitesimal characters as in \cite[\S 2]{NP}. For this subsection only, let $G'$ be a reductive group over $\R$ with fixed maximal torus $T_{G'}$, and let $\phi_\infty:W_{\R}\to\L{G}'$ be an archimedean $L$-parameter. Then we may conjugate $\phi_\infty$ so that $\phi_\infty|_{\C^\times}$ factors through the dual torus $T_{G'}^\vee\subset (G')^\vee$. Then there are $\lambda,\mu\in X_*(T_{G'}^\vee)\otimes\C$, with $\lambda-\mu\in X_*(T_{G'}^\vee)$ a genuine cocharacter, such that
\[\phi(z)=(\lambda+\mu)(\vert z\vert)\cdot(\lambda-\mu)(z/\vert z\vert).\]
Here, the first term in the product above has the following meaning: There are elements $\alpha_1,\dotsc,\alpha_n\in X_*(T_{G'}^\vee)$ and $a_1,\dotsc,a_n\in\C$ such that $\lambda+\mu=\sum_{i=1}^n \alpha_i\otimes a_i$, and then we define
\[(\lambda+\mu)(\vert z\vert)=\prod_{i=1}^n \alpha_i(\vert z\vert^{a_i}),\]
which makes sense because $\vert z\vert$ is real and positive. With this notation, we then define the \textit{infinitesimal character} of $\phi_\infty$ to be the orbit of the element $\lambda\in X_*(T_{G'}^\vee)\otimes\C$ under the Weyl group of $T_{G'}^\vee$ in $(G')^\vee$ or, what is the same, under the absolute Weyl group of $T$ in $G'$. Nair--Prasad \cite[Lemma 1]{NP} then prove that all elements in the $L$-packet of an archimedean $L$-parameter $\phi_\infty$ for $G'$ have infinitesimal character given by that of $\phi_\infty$.

If $G'$ were instead a complex Lie group, then the same process defines the infinitesimal character of a complex $L$-parameter of $G'$ when we take $T_{G'}$ to be a complex maximal torus.

Now if $\psi_\infty$ is a real $A$-parameter for $G'$, we define the \textit{infinitesimal character} of $\psi_\infty$ to be the infinitesimal character of the associated $L$-parameter $\phi_{\psi_\infty}$. (We take the same definition if $G'$ were complex.) We say, following \cite[Definition 3]{NP}, that an $A$-parameter for $G'$ is \textit{cohomological} if its infinitesimal character is that of a finite dimensional algebraic representation of $G'(\C)$. Nair--Prasad prove in \cite[Theorem 5]{NP} that the cohomological $A$-parameters for $G'$ are exactly the $A$-parameters for $G'$ of Adams--Johnson type (the definition of which we do not need to recall here). They then prove \cite[Theorems 6 and 7]{NP}, following \cite{adjo}, that if $G'(\R)$ is connected, then the union over cohomological $A$-parameters of the corresponding packets constructed in \cite{adjo} is exactly the set of cohomological representations of $G'(\R)$.

To an archimedean $A$-parameter $\psi_\infty$ for $G'$, one can attach an Arthur packet in potentially several ways. First, in general, one can attach a packet $\Pi_{\psi_\infty}^{\mathrm{ABV}}$, constructed by Adams--Barbasch--Vogan, as in \cite[Definition 22.6]{ABV}. This packet, which we will call the ABV-packet attached to $\psi$, has the property that all of its members have the same infinitesimal character which is given by that of $\psi_\infty$. This property is perhaps not so clear (at least to the author) from the definition given in \cite{ABV}, but it is clear at least from the exposition given in \cite[\S 3.1]{AAM}, since there it is defined as a subset of the set of representations with infinitesimal character $\prescript{\vee}{}{\mc{O}}$ in their notation.

Second, when $\psi_\infty$ is of Adams--Johnson type (equivalently, cohomological, by above) then one can attach a packet $\Pi_{\psi_\infty}^{\mathrm{AJ}}$, which we call the Adams--Johnson packet attached to $\psi_\infty$, using \cite{adjo}. Arancibia Robert has proved in \cite[Corollary 4.18]{aran} in this case that $\Pi_{\psi_\infty}^{\mathrm{AJ}}=\Pi_{\psi_\infty}^{\mathrm{ABV}}$.

Finally, when instead $G'$ is a quasisplit symplectic or special orthogonal group (resp. unitary group) but $\psi_\infty$ is general, then there is the Arthur packet $\Pi_{\psi_\infty}$ from \cite{arthurbook} (resp. \cite{Mok} and \cite{KMSW}) described above. Adams, Arancibia Robert, and Mezo have proved \cite[Theorems 8.0.2 and 9.0.3]{AAM} that if $G'$ is not a form of $SO_{2n}$, then in this case we have $\Pi_{\psi_\infty}=\Pi_{\psi_\infty}^{\mathrm{ABV}}$, and hence both equal $\Pi_{\psi_\infty}^{\mathrm{AJ}}$ if $\psi$ is moreover cohomological. Otherwise, when $G'$ is a form of $SO_{2n}$, then \textit{loc. cit.} proves that
\[\Pi_{\psi_\infty}=\Pi_{\psi_\infty}^{\mathrm{ABV}}\cup\Pi_{\vartheta\circ\psi_\infty}^{\mathrm{ABV}},\]
where $\vartheta$ is the outer automorphism of $SO_{2n}(\C)$ given by the adjoint action of an element in the non-identity component of $O_{2n}(\C)$. In either case, when $\psi_\infty$ is cohomological, all three packets contain only cohomological representations.

We use the information described above to prove the proposition below, which will be useful later in constructing Galois representations. But first we make a definition following \cite{KSeo} which we will use in the hypotheses of that proposition.

\begin{definition}
\label{defstdreg}
Assume $G$ is even orthogonal. Let $\psi_\infty$ be an Arthur parameter for $G(\R)$. We say $\psi_\infty$ is $\std$\textit{-regular} if its infinitesimal character is represented by a weight which, when written as an $n$-tuple as in Section \ref{secgroups}, has only nonzero entries.

As well, we say that a representation of $G(\R)$ is $\std$\textit{-regular} if it lies in the packet $\Pi_{\psi_\infty}$ associated with some $\std$-regular Arthur parameter $\psi_\infty$.
\end{definition}

Note that $\std$-regularity is indeed well defined for representations; indeed, if $\pi_\infty$ is such, then by above, we have that $\pi_\infty\in\Pi_{\psi_\infty}^{\mathrm{ABV}}$ or $\Pi_{\vartheta\circ\psi_\infty}^{\mathrm{ABV}}$ for some $\psi_\infty$. The infinitesimal characters of $\psi_\infty$ and $\vartheta\circ\psi_\infty$ differ only by a sign in one entry. Thus all choices of $\psi_\infty$ having $\pi_\infty$ in $\Pi_{\psi_\infty}$ have the same infinitesimal characters up to signs in each entry, and are therefore $\std$-regular if just one is.

We also make the following convenient notation for ourselves.

\begin{notation}
\label{notndeltaeo}
We write $\delta_{\mr{eo}}\in\Z$ for the integer $1$ if $G$ is even orthogonal, and for $0$ otherwise.
\end{notation}

\begin{proposition}
\label{propparamliftcoho}
Let $v$ be an archimedean place of $F$, and let $\psi_v$ be an Arthur parameter for $G(F_v)$, and assume $\Pi_{\psi_v}$ contains a cohomological representation. If $G$ is even orthogonal, assume in addition that $\psi_v$ (or equivalently, any representation in $\Pi_{\psi_{v}}$) is $\std$-regular in the sense defined above. Let $\tilde{\psi}_v$ be defined by the right hand side of \eqref{eqnpsiwtilde} above. Write $\chi$ for the infinitesimal character of $\tilde\psi_v$, which may thus be viewed as an unordered $N$-tuple of complex numbers. Then
\[\chi+(\tfrac{1+\delta_{\mr{eo}}-N}{2},\tfrac{1+\delta_{\mr{eo}}-N}{2},\dotsc,\tfrac{1+\delta_{\mr{eo}}-N}{2})\]
is integral and regular.
\end{proposition}

\begin{proof}
The assumption that $\Pi_{\psi_\infty}$ contains a cohomological representation implies, by the facts recalled above, that the infinitesimal character of $\psi_\infty$ equals that of a finite dimensional algebraic representation of $G(\overline{F}_v)$. Such an infinitesimal character is represented by $\lambda_0+\rho$ for some dominant integral weight
\[\lambda_0=(a_1,\dotsc,a_n)=a_1 e_1+\dotsb+a_n e_n,\]
where $n$ is the absolute rank of $G$ and where we have
\[a_1,\dotsc,a_n\in\Z,\quad a_1\geq\dotsb\geq a_n,\quad a_n\geq 0\textrm{ if }G\textrm{ is symplectic or odd orthogonal,}\]
and where, since we are assuming $\std$-regularity if $G$ is even orthogonal, we have $a_n\ne 0$ in that case. Recall that
\[\rho=\begin{cases}
(\tfrac{n-1}{2},\tfrac{n-3}{2},\dotsc,\tfrac{3-n}{2},\tfrac{1-n}{2})&\textrm{if }G\textrm{ is unitary;}\\
(n-\tfrac{1}{2},n-\tfrac{3}{2},\dotsc,\tfrac{3}{2},\tfrac{1}{2})&\textrm{if }G\textrm{ is odd orthogonal;}\\
(n,n-1,\dotsc,2,1)&\textrm{if }G\textrm{ is symplectic;}\\
(n-1,n-2,\dotsc,1,0)&\textrm{if }G\textrm{ is even orthogonal.}
\end{cases}\]
Let $\lambda$ be the coweight in $X^*(T^\vee)\otimes\Q$ corresponding to $\lambda_0+\rho$. Then the coweight of the diagonal torus in $GL_N(\C)$ given by
\[(\std\circ\lambda)+(\tfrac{1+\delta_{\mr{eo}}-N}{2},\tfrac{1+\delta_{\mr{eo}}-N}{2},\dotsc,\tfrac{1+\delta_{\mr{eo}}-N}{2})\] is thus given by the cocharacter $\C^\times\to GL_{N}(\C)$ as follows:
\[z\mapsto\begin{cases}
\diag(z^{a_1},z^{a_2-1}\dotsc,z^{a_n-n+1})&\textrm{if }G\textrm{ is unitary;}\\
\diag(z^{a_1},z^{a_2-1},\dotsc,z^{a_n-n+1},z^{-a_n-n},\dotsc,z^{-a_1-2n+1})&\textrm{if }G\textrm{ is odd orthogonal;}\\
\diag(z^{a_1},\dotsc,z^{a_n-n+1},z^{-n},z^{-a_n-N-\delta_{\mr{eo}}-1},\dotsc,z^{-a_1-2n})&\textrm{if }G\textrm{ is symplectic;}\\
\diag(z^{a_1},z^{a_2-1},\dotsc,z^{a_n-n+1},z^{-a_n-n+1},\dotsc,z^{-a_1-2n+2})&\textrm{if }G\textrm{ is even orthogonal.}\\
\end{cases}\]
Note that in the even orthogonal case above, we have that the exponents $a_n-n+1$ and $-a_n-n+1$ on the two middle entries are distinct by $\std$-regularity. So in all cases, this character corresponds to a regular integral character of the diagonal torus in $GL_{N}(\C)$. It follows easily from the definitions that the Weyl orbit of $\std\circ\lambda$ is the infinitesimal character of $\tilde\psi_v$, and so the proposition follows.
\end{proof}

\subsubsection*{Frobenius semisimple Weil--Deligne representations}
Let $L$ be a nonarchimedean local field of characteristic $0$ and residue characteristic different from the fixed prime $p$. Write $\lambda$ for the order of the residue field of $L$. Let $W_L$ be the Weil group of $L$. Let $\varphi\in W_L$ be a fixed arithmetic Frobenius element, and let $I_L\subset W_L$ be the inertia subgroup. Then there is a unique isomorphism $\nu:W_L/I_L\to\Z$ such that $\nu$ sends the image of $\varphi$ in $W_L/I_L$ to $1\in\Z$. 

Recall that a \textit{Weil--Deligne representation} for $L$ of dimension $d$ (with coefficients in $\overline\Q_p$) is a pair $(r,N)$, where $r$ is a representation of $W_L$ on a $d$-dimensional $\overline{\Q}_p$-vector space $V$ with open kernel, and $N:V\to V$ is a nilpotent linear operator such that
\[r(g)N=\lambda^{\nu(g)}Nr(g),\]
for all $g\in W_L$. The operator $N$ in this definition is called the \textit{monodromy operator}. Moreover, we call $(r,N)$ \textit{Frobenius-semisimple} if $\varphi$ acts semisimply on $V$; as is well known, it is equivalent to require that all elements of $W_L$ act semisimply on $V$, and it is equivalent to require $r$ to be semisimple as a representation.

For any $a\in\overline\Q_p^\times$, let $\chi_a$ denote the $1$-dimensional Weil-Deligne representation with (necessarily) trivial monodromy operator, and such that
\[\chi_a(g)=a^{\nu(g)}\]
for all $g\in W_L$. We sometimes abuse notation slightly and write $\chi_a$ for the underlying character of $W_L$.

Fix a square root $\lambda^{1/2}$ in $\overline\Q_p$. When $M$ is an odd integer, let us write $\lambda^{M/2}=\lambda^{1/2}\lambda^{(M-1)/2}$; since $(M-1)/2$ is an integer in this case, both expressions in the product make sense.

Then for any integer $m\geq 1$, the \textit{special representation} $\Sp(m)$ is the $m$-dimensional Weil--Deligne representation for $L$ with underlying representation of $W_L$ given by
\[\chi_{\lambda^{(m-1)/2}}\oplus\chi_{\lambda^{(m-3)/2}}\oplus\dotsb\oplus\chi_{\lambda^{-(m-1)/2}},\]
such that for any integer $i$ with $1\leq i<m$, the monodromy operator $N$ of $\Sp(m)$ sends a fixed nonzero vector in the summand $\chi_{\lambda^{i-1-(m-1)/2}}$ to one in the summand $\chi_{\lambda^{i-(m-1)/2}}$, and such that the $N$ kills summand $\chi_{\lambda^{(m-1)/2}}$. The special representation is obviously Frobenius-semisimple. Moreover, any indecomposable Frobenius-semisimple Weil--Deligne representation is of the form $r\otimes\Sp(m)$ for some $m\geq 1$ and some irreducible, finite dimensional representation $r$ of $W_L$ with open kernel, and any Frobenius-semisimple Weil--Deligne representation is a direct sum of indecomposable ones.

Given a Frobenius semisimple Weil--Deligne representation $(r,N)$ of dimension $d$ for $L$ and an isomorphism $\iota:\C\to\overline\Q_p$, we can define an $L$-parameter $\phi$ for $GL_d(L)$,
\[\phi:W_L\times SL_{2,\mr{D}}(\C)\to GL_d(\C),\]
which depends on $\iota$, as follows. By the above, we may write
\begin{equation}
\label{eqnWDrepdecompgen}
(r,N)=\bigoplus_{i\in I}r_i\otimes\Sp(m_i),
\end{equation}
where $I$ is a finite index set and for each $i\in I$, $r_i$ is an irreducible, finite dimensional $\overline\Q_p$-representation of $W_L$ with open kernel, and $m_i\geq 1$ is an integer. Then $\phi$ is defined, up to conjugation, by
\[\phi=\bigoplus_{i\in I}(\iota^{-1}\circ r_i)\boxtimes\Sym^{m_i-1},\]
which is a representation of $W_L\times SL_2(\C)$ on a complex vector space of dimension that of $(r,N)$. We will use this association later to describe the local behavior of the Galois representations we study.

Finally, we recall that a Frobenius-semisimple Weil--Deligne representation $(r,N)$ for $W_L$ is \textit{pure of weight} $w\in\Z$ if, when written in the form \eqref{eqnWDrepdecompgen}, the eigenvalues of $\varphi$ are all $\lambda$-\textit{Weil numbers of weight} $w$; this means that they are algebraic numbers which, under any embedding into the complex numbers, have absolute value $\lambda^{w/2}$.

\section{Jacquet modules}
\label{secjacmods}
The content of this section and the next will be purely local, but for sake of consistency we will use more global notation. Thus, continuing with the notation as in Section \ref{secgroups}, we have our classical group $G$ over our totally real number field $F$, and our extension $E/F$ which is trivial unless $G$ is unitary, in which case it is imaginary quadratic. We let $v$ be a finite place of $F$ (which, in this section, we allow to lie above $2$) such that there is a unique place, call it $w$, of $E$ lying over $v$ in $F$. So if $E=F$ then this uniqueness condition is vacuous, and otherwise we have that $v$ is either inert or ramified in $E$. In particular we have that $G(F_v)$ is not of the form $GL_{n}(F_v)$. We will work in this section only with the local group $G(F_v)$.

Let us, throughout this section, denote the $F_v$-rank of $G_{/F_v}$ by $n_v$. Let $P=MN$ be the maximal parabolic in $G$ fixed in Section \ref{secgroups}. Then $M(F_v)=GL_1(E_w)\times H(F_v)$, where $H$ is of the same type as $G$ but with $F_v$-rank smaller by $1$. The goal of this section is to partially compute the Jacquet modules down to $M(F_v)$ of two types of representations, namely those associated with tempered $L$-parameters, and those induced from representations of $M(F_v)$. We will use the information obtained in this section, along with the Hecke operators constructed in the next section, to control the local behavior of Hecke families at bad primes.

So fix a tempered $L$-parameter $\phi$ for $G(F_v)$, viewed via the standard representation as a (conjugate-)self dual homomorphism
\[W_{E_w}\times SL_2(\C)\to GL_N(\C),\]
as explained in Section \ref{secparams}; here the integer $N$ is the dimension of the standard representation of $G^\vee$ as in Section \ref{secgroups}. Then the parameter $\phi$ decomposes into
\[\phi=\phi_0\oplus(\phi_1\oplus\phi_1^{*})\]
where $\phi_1$ is a tempered $L$-parameter for $GL_{n_1}(E_w)$ for some integer $n_1\leq n_v$, and $\phi_0$ is a \textit{discrete} $L$-parameter for a group $G_{n_0}$ over $F_v$ of the same type as $G$ and $F_v$-rank $n_0=n_v-n_1$. We recall that a parameter is called discrete if it does not factor through any relevant proper parabolic subgroup of the target $L$-group. This decomposition is relatively standard and follows from the fact that the relevant maximal parabolic subgroups of $\L{G(F_v)}$ are those with Levis of the form $GL_{n_1}(\C)\times \L{G}_{n_0}$.

Now for any integer $m$ with $1\leq m\leq n_v$, the group $GL_{m/E_w}\times G_{n_v-m}$ is then the Levi factor of a standard parabolic subgroup of $G_{/F_v}$; we have written here $G_{n_v-m}$ for the group over $F_v$ of the same type as $G_{/F_v}$ but lower in $F_v$-rank by $m$. Let us call this associated parabolic subgroup $P_{m}$ and write $M_{m}$ for its Levi factor.

Given a smooth admissible representation $\pi_0$ of $G_{n_0}(F_v)$ and a smooth admissible representation $\pi_1$ of $GL_{n_1}(E_w)$, we write
\[\pi_1\rtimes\pi_0=\Ind_{P_{n_1}(F_v)}^{G(F_v)}(\pi_1\boxtimes\pi_0)\]
for the normalized parabolic induction of $\pi_1\boxtimes\pi_0$ from $P_{n_1}(F_v)$ to $G(F_v)$. We then have the following proposition about the relationships between the elements of the packets for $\phi$ and those for $\phi_0$ and $\phi_1$.

\begin{proposition}
\label{proppacketdecomp}
With $\phi$, $\phi_0$ and $\phi_1$ as above, let $\pi\in\Pi_{\phi}$. Let $\pi_1$ be representation of $GL_{n_1}(E_w)$ attached to $\phi_1$ by the local Langlands correspondence for $GL_{n_1}(E_w)$. Then there exists a $\pi_0\in\Pi_{\phi_0}$ such that $\pi$ is a direct summand of $\pi_1\rtimes\pi_0$.
\end{proposition}

\begin{proof}
See \cite[Proposition 7.1]{LLS} for a more precise statement, as well as the references given there for the proof.
\end{proof}

We begin to consider various Jacquet modules of representations of $G_{n_0}$ and $G$; see the notation section for the meaning of the functors $\Jac$.

\begin{proposition}
\label{propjacM1Sprep}
Let $n_0$ and $\phi_0$ be as above Proposition \ref{proppacketdecomp}. Assume $n_0\geq 1$. Decompose the parameter $\phi_0$ into indecomposable pieces as
\[\phi_0=\bigoplus_{i\in I_0} r_i\boxtimes\Sym^{a_i-1},\]
for some index set $I_0$; the summands are distinct because $\phi_0$ is discrete. Let $\pi_0\in\Pi_{\phi_0}$, and assume there is an irreducible constituent $\tau$ of $\Jac_{M}(\pi_0)$. Then there is an $i\in I_0$ such that $r_i$ is a character and, writing $\rho_i$ for the character of $GL_1(E_w)$ attached to $r_i$ by class field theory, we have
\[\tau\cong(\rho_i\Vert\cdot\Vert_{w}^{(a_i-1)/2}\boxtimes\sigma)\otimes \delta_{P(F_v)}^{1/2},\]
for some smooth irreducible representation $\sigma$ of $G_{n_v-1}(F_v)$.
\end{proposition}

\begin{proof}
In the unitary or orthogonal case, this follows from the theorem in \cite[\S 3.4]{MR}; the remaining symplectic case is proved in \cite[\S 2.5]{Moeglin}. We note that the factor $\delta_{P(F_v)}^{1/2}$ is inserted above to undo the normalization that the cited authors use for Jacquet modules.
\end{proof}

Fix an integer $m$ with $1\leq m\leq n_v$. We now begin to describe the Jacquet modules down to the Levi $M(F_v)$ of representations induced from the parabolic $P_{m}(F_v)$. Let $W_{m}$ be the relative Weyl group of the Levi $M_{m}$ over $F_v$, and $W_M$ that of $M$. If $m\ne n_v$, then $W_{m}$ is the group generated by all the simple reflections in the relative Weyl group $W$ of $G$ over $F_v$ except for $(m,m+1)$; if instead $m=n_v$, then $W_m$ is generated by all the simple refections except for $(-)_{n_v}$ except when $G$ is even orthogonal and split over $F_v$, in which case the omitted reflection is $(-)_{n_v-1}(-)_{n_v}$. Let us write
\[\delta_{\mr{seo},v}=\begin{cases}
1 &\textrm{if }G_{/F_v}\textrm{ is split even orthogonal;}\\
0 &\textrm{otherwise}
\end{cases}=\begin{cases}
1 &\textrm{if }G_{/F_v}\textrm{ is of type }D;\\
0 &\textrm{otherwise.}
\end{cases}\]
The set $[W/W_M]$ of minimal length representatives for $W/W_M$ in $W$ is given in Section \ref{secgroups}. It follows that the set $[W_{m}\backslash W/W_M]$ is given, if $m\ne n_v$, by the three element set $\{1,w_{m,s},w_{m,l}\}$ where $w_{m,s}$ and $w_{m,l}$ are defined by
\[w_{m,s}=(m+1,m,\dotsc,1),\qquad w_{m,l}=(m,m+1,\dotsc,n_v)(-)_{n_v-1}^{\delta_{\mr{seo},v}}(-)_{n_v}(n_v,n_v-1,\dotsc,1).\]
If $m=n_v$, then $[W_{m}\backslash W/W_M]=\{1,w_{n_v,l}\}$ where
\[w_{n_v,l}=(-)_{n_v-1}^{\delta_{\mr{seo},v}}(-)_{n_v}(n_v,n_v-1,\dotsc,1).\]

\begin{proposition}
\label{propjacM1gen}
Let $m$ be an integer with $1\leq m\leq n_v$. Let $\pi_{GL}$ be a smooth admissible representation of $GL_{m}(E_w)$ and $\pi_{G_{n_v-m}}$ a smooth admissible representation of $G_{n_v-m}(F_v)$. Let $\delta_{\mr{eo}}$ be the symbol from Notation \ref{notndeltaeo}. Then the constituents of the semisimplification of the Jacquet module,
\[\Jac_{M_1}(\pi_{GL}\rtimes\pi_{G_{n_v-m}})^{\sss},\]
are given by the following: If $m\ne n_v$, then there are three constituents, call them $\sigma_1$, $\sigma_l$ and $\sigma_s$, and if $m=n_v$, there are two constituents $\sigma_1$ and $\sigma_l$, whose descriptions in all cases we give now.

First, write
\[\Jac_{GL_1(E_w)\times GL_{m-1}(E_w)}(\pi_{GL})^{\sss}=\bigoplus_{j\in J_1} \chi_{1,j}\boxtimes\tau_{1,j},\]
for some index set $J_1$, where the $\chi_{1,j}$ are characters of $GL_1(E_w)$ and $\tau_{1,j}$ are irreducible smooth representations of $GL_{m-1}(E_w)$. Then
\[\sigma_1=\bigoplus_{j\in J_1}\chi_{1,j}\Vert\cdot\Vert_w^{(N-\delta_{\mr{eo}}-m)/2}\boxtimes ((\tau_{1,j}\otimes\Vert\det\Vert_w^{1/2})\rtimes\pi_{G_{n_v-m}})^{\sss}.\]

Next, write
\[\Jac_{GL_{m-1}(E_w)\times GL_{1}(E_w)}(\pi_{GL})^{\sss}=\bigoplus_{j\in J_l} \tau_{l,j}\boxtimes\chi_{l,j},\]
for some index set $J_l$, where the $\chi_{l,j}$ are characters of $GL_1(E_w)$ and $\tau_{l,j}$ are irreducible smooth representations of $GL_{m-1}(E_w)$. Then
\[\sigma_l=\bigoplus_{j\in J_l}\chi_{l,j}^{*}\Vert\cdot\Vert_w^{(N-\delta_{\mr{eo}}-m)/2}\boxtimes ((\tau_{l,j}\otimes\Vert\det\Vert_w^{-1/2})\rtimes\pi_{G_{n_v-m}})^{\sss},\]
where $(\cdot)^{*}$ denotes (conjugate-)dual as usual.

Finally, when $m\ne n_v$, write
\[\Jac_{GL_{1}(E_w)\times G_{n_v-m-1}(F_v)}(\pi_{G_{n_v-m}})^{\sss}=\bigoplus_{j\in J_s} \chi_{s,j}\boxtimes\tau_{s,j},\]
for some index set $J_s$, where the $\chi_{s,j}$ are characters of $GL_1(E_v)$ and $\tau_{s,j}$ are irreducible smooth representations of $G_{n_v-m-1}(F_v)$. Then
\[\sigma_s=\bigoplus_{j\in J_s}\chi_{s,j}\Vert\cdot\Vert_w^m\boxtimes (\pi_{GL}\rtimes\tau_{s,j})^{\sss}.\]
\end{proposition}

\begin{proof}
This is an exercise using \cite{casselman}, but we carry out the computation here because we need to be very careful about what powers of the norm character $\Vert\cdot\Vert_w$ appear in the constituents of these Jacquet modules.

Now \cite{casselman} writes the constituents of $\Jac_{M}(\pi_{GL}\rtimes\pi_{G_{n_v-m}})^{\sss}$ as a sum indexed by the set $[W_m\backslash W/W_M]$ of certain representations induced from other Jacquet modules as follows. First, for an $F_v$-rational root $\alpha$ of $G$, let $m_v(\alpha)$ denote the multiplicity of $\alpha$ in $G_{/F_v}$; so by what we recalled in Section \ref{secgroups}, we have
\[m_v(\alpha)=\begin{cases}
2(N-2n_v)&\textrm{if }G\textrm{ is unitary and }\alpha=\pm e_i;\\
2&\textrm{if }G\textrm{ is unitary and }\alpha\ne\pm e_i;\\
2&\textrm{if }G\textrm{ is nonsplit even orthogonal and }\alpha=\pm e_i;\\
1&\textrm{otherwise.}
\end{cases}\]
Let $2\rho_{P_m}$ denote the sum of positive roots in the unipotent radical $N_m$ of $P_m$, counted with multiplicity. For $w\in[W_m\backslash W/W_M]$, consider the sum of roots
\begin{equation}
\label{eqnsumofroots}
w^{-1}(2\rho_{P_m})+\sum_{\substack{\textrm{roots }\alpha>0\textrm{ not in }M_m \\ w^{-1}\alpha<0\textrm{ is not in }M}}-2w^{-1}m_v(\alpha)\alpha.
\end{equation}
This sum of roots defines a character of the minimal standard $F_v$-Levi subgroup of $G$, call it $M_0$, which is the restriction to $M_0$ of a character of the $F_v$-Levi subgroup $w^{-1}M_m w\cap M$. Let us denote by $\delta_w$ the modulus of this character on $(w^{-1}M_m w\cap M)(F_v)$. Then \cite[Proposition 6.3.3]{casselman} says that
\begin{multline}
\label{eqncasseljac}
\Jac_{M_1}(\pi_{GL}\rtimes\pi_{G_{n_v-m}})^{\sss}\cong\\
\bigoplus_{w\in[W_m\backslash W/W_M]}\prescript{\nn}{}\Ind_{(w^{-1}P_m w\cap M)(F_v)}^{M(F_v)}(w^{-1}(\Jac_{M_m\cap wM w^{-1}}(\pi_{GL}\boxtimes\pi_{G_{n_v-m}}))\otimes\delta_w^{1/2})^{\sss},
\end{multline}
where for a representation $\pi$ of $(M_m\cap wMw^{-1})(F_v)$, we denote by $w^{-1}(\pi)$ the representation of $(w^{-1}M_m w\cap M)(F_v)$ on the same space with action given by $w^{-1}(\pi)(g)=\pi(wgw^{-1})$ for a fixed representative of $w$ in the normalizer of $M_0(F_v)$, and where $\prescript{\mathrm{nn}}{}\Ind$ denotes non-normalized induction. Note that \textit{loc. cit.} has a typo which forgets the Weyl twist corresponding to the one which appears on the right hand side above; this twist appears in the proof of \textit{loc. cit.} but not in the statement.

Now the positive $F_v$-rational roots of $G$ which are not in $M_m$ are given by
\begin{equation}
\label{eqndifferentroots}
\begin{gathered}
e_i-e_j\textrm{ for }1\leq i\leq m,\,\,m+1\leq j\leq n_v;\\
e_i+e_j\textrm{ for }1\leq i\leq m,\,\,1\leq j\leq n_v,\,\,i<j;\\
\textrm{ all positive multiples of }e_i\textrm{ for }1\leq i\leq m.
\end{gathered}
\end{equation}
Note that the sum, with multiplicity, of the roots of the form $e_i-e_j$ and $e_i+e_j$ listed above is
\[[E_w:F_v](2n_v-m-1)(e_1+\dotsb+e_m),\]
Also, the sum, with multiplicity, of all positive multiples of $e_i$ is
\[\begin{cases}
2(N-2n_v+1)e_i&\textrm{if }G\textrm{ is unitary;}\\
e_i&\textrm{if }G\textrm{ is odd orthogonal;}\\
2e_i&\textrm{if }G\textrm{ is symplectic;}\\
0&\textrm{if }G\textrm{ is split even orthogonal;}\\
2e_i&\textrm{if }G\textrm{ is nonsplit even orthogonal.}
\end{cases}\]
Therefore, in all cases, the sum, with multiplicity, of the roots listed in \eqref{eqndifferentroots} is
\begin{equation}
\label{eqn2rhoPm}
2\rho_{P_m}=[E_w:F_v](N-\delta_{\mr{eo}}-m)(e_1+\dotsb+e_m).
\end{equation}
So when $w=1$, the formula \eqref{eqnsumofroots} gives $2\rho_{P_m}$ and the corresponding summand of \eqref{eqncasseljac} becomes
\begin{equation}
\label{eqncasseljacfor1}
\prescript{\mathrm{nn}}{}\Ind_{M(F_v)\cap P_{m}(F_v)}^{M(F_v)}(\Jac_{M\cap M_m}(\pi_{GL}\boxtimes\pi_{G_{n_v-m}})\otimes\delta_{P_m(F_v)}^{1/2})^{\sss}.
\end{equation}
Since
\begin{multline*}
\begin{aligned}
[E_w:F_v](N-\delta_{\mr{eo}}-m)(e_1+\dotsb+e_m)=&\,[E_w:F_v](N-\delta_{\mr{eo}}-m)e_1\\
&+((N-2-\delta_{\mr{eo}})-(m-1))(e_2+\dotsb+e_m)\\
&+[E_w:F_v](e_2+\dotsb+e_m),
\end{aligned}
\end{multline*}
we have that
\[\delta_{P_m(F_v)}|_{GL_1(E_w)\times GL_{m-1}(E_w)}=\Vert\cdot\Vert_w^{N-\delta_{\mr{eo}}-m}\boxtimes(\Vert\det\Vert_w\cdot\delta_{(P_{m}\cap G_{n_v-1})(F_v)}|_{GL_{m-1}(E_w)}).\]
Thus \eqref{eqncasseljacfor1} becomes
\[\Ind_{M(F)\cap P_{m}(F_v)}^{M(F)}(\Jac_{GL_1\times GL_{m-1}}(\pi_{GL})\otimes(\Vert\cdot\Vert_w^{(N-\delta_{\mr{eo}}-m)/2}\boxtimes\Vert\det\Vert_w^{1/2})\boxtimes\pi_{G_{n_v-m}})^{\sss},\]
where the induction is now unitarily normalized. This gives the constituent $\pi_1$ in the proposition.

Now we consider the summand of \eqref{eqncasseljac} indexed by $w_{m,l}$. The element $w_l^{-1}$ acts by
\begin{gather*}
w_{m,l}^{-1}e_i=e_{i+1}\textrm{ for }1\leq i\leq m-1,\qquad w_{m,l}^{-1}e_m=-e_1,\\
w_{m,l}^{-1}e_i=e_i\textrm{ for }m+1\leq i\leq n_v-1,\qquad w_{m,l}^{-1}e_{n_v}=(-1)^{\delta_{\mr{seo}}}e_{n_v}.
\end{gather*}
The positive roots $\alpha$ for which $w_{m,l}^{-1}\alpha<0$ are therefore the ones of the form $e_m\pm e_i$ for $i\ne m$, as well as the positive multiples of $e_m$. Then for any of these roots $\alpha$, we have automatically $w_{m,l}^{-1}\alpha$ is not in $M$. Also, of these roots $\alpha$, the ones for which $\alpha$ is not in $M_m$ are given by
\[e_i+e_m\textrm{ for }1\leq i\leq m-1,\qquad e_m\pm e_i\textrm{ for }m+1\leq i\leq n_v,\qquad \textrm{positive multiples of }e_m.\]
Thus in all cases, the sum, with multiplicity, of these roots is
\[[E_w:F_v](e_1+\dotsb+e_{m-1})+[E_w:F_v](N-\delta_{\mr{eo}}-m)e_m.\]
So \eqref{eqnsumofroots} for $w=w_{m,l}$ becomes
\begin{multline*}
w_{m,l}^{-1}([E_w:F_v](N-\delta_{\mr{eo}}-m-2)(e_1+\dotsb+e_{m-1})-[E_w:F_v](N-\delta_{\mr{eo}}-m)e_m)\\
=[E_w:F_v](N-\delta_{\mr{eo}}-m)e_1+[E_w:F_v]((N-2)-\delta_{\mr{eo}}-(m-1)-1)(e_2+\dotsb+e_{m}).
\end{multline*}
Therefore we have
\[\delta_w|_{GL_1(E_w)\times GL_{m-1}(E_w)}=\Vert\cdot\Vert_w^{N-\delta_{\mr{eo}}-m}\boxtimes(\Vert\det\Vert_w^{-1}\cdot\delta_{(P_{m}\cap G_{n_v-1})(F_v)}|_{GL_{m-1}(E_v)}).\]

Now we note that $w_{m,l}^{-1}P_m w_{m,l}\cap M=P_m\cap M$, and that $M_m\cap w_{m,l} M w_{m,l}^{-1}$ is the standard Levi of $G_{n_v-m/F_v}$ of the form $GL_{m-1/E_w}\times GL_{1/E_w}\times G_{n_v-m}$. Thus the summand of \eqref{eqncasseljac} indexed by $w_{m,l}$ becomes the unitary induction
\begin{equation}
\label{eqncasseljacforwl}
\Ind_{M(F_v)\cap P_{m}(F_v)}^{M(F_v)}(w_{m,l}^{-1}(\Jac_{GL_{m-1}\times GL_1}(\pi_{GL}))\otimes(\Vert\det\Vert_w^{-1/2}\boxtimes\Vert\cdot\Vert_w^{(N-\delta_{\mr{eo}}-m)/2})\boxtimes\pi_{G_{n_v-m}})^{\sss}.
\end{equation}
Since
\[w_l^{-1}(\Jac_{GL_{m-1}\times GL_1}(\pi_{GL}))^{\sss}=\bigoplus_{j\in J_l} w_l^{-1}(\tau_{l,j}\boxtimes\chi_{l,j})=\bigoplus_{j\in J_l} (\chi_{l,j}^{*}\boxtimes\tau_{l,j}),\]
we see that \eqref{eqncasseljacforwl} becomes the expression we have for $\pi_l$ in the proposition.

Finally, we assume $m\ne n_v$ and we compute the summand of \eqref{eqncasseljac} indexed by $w_{m,s}$. The only positive $F_v$-rational roots $\alpha$ of $G$ not in $M_m$ for which $w_{m,s}^{-1}\alpha<0$ are given by $\alpha=e_i-e_{m+1}$ for $1\leq i\leq m$. The sum, with multiplicity, of these roots is given by
\[[E_w:F_v](e_1+\dotsb+e_m)-[E_w:F_v] me_{m+1},\]
and so \eqref{eqnsumofroots} for $w=w_{m,s}$ becomes
\[[E_w:F_v](N-\delta_{\mr{eo}}-m-2)(e_2+\dotsb+e_{m+1})+2[E_w:F_v] me_{1}.\]
Therefore we have
\[\delta_w|_{GL_1(E_w)\times GL_{m}(E_w)}=\Vert\cdot\Vert_w^{2m}\boxtimes\delta_{P_{m}'(F_v)},\]
where we have written $P_m'$ for the standard $F_v$-parabolic of $G_{n_v-1}$ analogous to $P_m$ for $G$, having Levi $GL_{m/E_w}\times G_{n_v-m-1}$; note that $w_{m,s}^{-1}P_m w_{m,s}\cap M=GL_{1/E_w}\times P_m'$, 

Thus the summand of \eqref{eqncasseljac} indexed by $w_{m,s}$ becomes the unitary induction
\begin{equation*}
\Ind_{GL_{1}(E_w)\times P_m'(F_v)}^{M(F_v)}(w_{m,s}^{-1}(\pi_{GL}\rtimes[\Jac_{GL_{1}\times G_{n_v-m-1}}(\pi_{G_{n_v-m}})\otimes(\Vert\cdot\Vert_w^{m}\boxtimes 1)]))^{\sss},
\end{equation*}
and this is easily seen to yield the constituent $\sigma_s$ described in the proposition.
\end{proof}

\begin{corollary}
\label{corjacofeisenstein}
Let $\sigma$ be a smooth irreducible representation of $G_{n_v-1}(F_v)$ and $\chi$ any continuous character of $GL_1(E_w)$. Then the Jacquet module
\[\Jac_{M}(\chi\rtimes\sigma)\]
contains the constituents $\chi\Vert\cdot\Vert_w^{(N-\delta_{\mr{eo}}-1)/2}\boxtimes\sigma$ and $\chi^{*}\Vert\cdot\Vert_w^{(N-\delta_{\mr{eo}}-1)/2}\boxtimes\sigma$, where $\delta_{\mr{eo}}$ is as in Notation \ref{notndeltaeo}, and $(\cdot)^*$ denotes (conjugate-)dual as usual.
\end{corollary}

\begin{proof}
This follows from Proposition \ref{propjacM1gen} for $m=1$; the desired constituents are, respectively, exactly what are called $\sigma_1$ and $\sigma_l$ in that proposition.
\end{proof}

\begin{proposition}
\label{propjacoftempered}
Let $\phi$, $\phi_0$ and $\phi_1$ be as in the beginning of this section, so that $\phi=\phi_0\oplus(\phi_1\oplus\phi_1^{*})$, and decompose the parameters $\phi_0$ and $\phi_1$ as
\[\phi_0=\bigoplus_{i\in I_0} r_i\boxtimes\Sym^{a_i-1},\qquad\phi_1=\bigoplus_{i\in I_1} r_i\boxtimes\Sym^{a_i-1},\]
for some index sets $I_0$ and $I_1$. Let $\pi\in\Pi_\phi$ and write
\[\Jac_{M}(\pi)^{\sss}=\bigoplus_{j\in J}\chi_j\boxtimes\sigma_j,\]
for some index set $J$, where each $\chi_j$ is a character of $GL_1(E_w)$ and each $\sigma_j$ is a smooth irreducible representation of $G_{n_v-1}(F_v)$. Let $\delta_{\mr{eo}}$ be as in Notation \ref{notndeltaeo}.

Then for each $j\in J$, there is an $i\in I_0\sqcup I_1$ depending on $j$, and a $b_j\in\frac{1}{2}\Z$ with $\frac{N-\delta_{\mr{eo}}-1}{2}\leq b_j\leq \frac{N-\delta_{\mr{eo}}-1}{2}+\frac{n_v-1}{2}$ such that $r_i$ is $1$-dimensional and, writing $\rho_i$ for the character of $GL_1(E_w)$ associated with $r_i$ by class field theory, we have
\[\chi_j=\rho_i\Vert\cdot\Vert_w^{b_j}\quad\textrm{or}\quad\chi_j=\rho_i^*\Vert\cdot\Vert_w^{b_j}.\]
Here, the symbol $(\cdot)^*$ denotes (conjugate-)dual as usual.
\end{proposition}

\begin{proof}
We induct on the size of $I_1$. If $I_1$ is empty, then Proposition \ref{propjacM1Sprep} tells us that for any $j\in J$, there is an $i\in I_0$ with $r_i$ a character and $\chi_j=\rho_i\Vert\cdot\Vert_w^{(a_i-1)/2}\delta_{P(F_v)}^{1/2}$, where $\rho_i$ is associated with $r_i$ by class field theory. Since $a_i\geq 1$ and $\delta_{P(F_v)}^{1/2}$ is given by $\Vert\cdot\Vert_w^{(N-\delta_{\mr{eo}}-1)/2}$ on the $GL_1(E_w)$ factor of $M(F_v)$ (see the computation of \eqref{eqn2rhoPm} in the proof of Proposition \ref{propjacM1gen} and specialize to $m=1$ there), this implies the proposition when $I_1$ is empty.

Now assume $I_1$ is nonempty. Fix $i_1\in I_1$, and write
\[\phi_1'=\bigoplus_{\substack{i\in I_1 \\ i\ne i_1}}r_i\boxtimes \Sym^{a_i-1}.\]
Let $n_{i_1}=a_{i_1}\dim r_{i_1}$ and let $\pi_{i_1}$ be the representation of $GL_{n_{i_1}}(E_w)$ associated by the local Langlands correspondence with $r_{i_1}\boxtimes \Sym^{a_{i_1}}$; it is a generalized Steinberg representation. Similarly, let $n_1'=\dim\phi_1'$, and let $\pi_1'$ be the representation of $GL_{n_1'}(E_w)$ associated with $\phi_1'$. Then, writing $\pi_1$ for the representation of $GL_{n_1}(E_w)$ associated with $\phi_1$, we know that $\pi_1$ is the unitary parabolic induction of $\pi_{i_1}\boxtimes\pi_1'$ because $\phi_1$ is tempered (as this implies that blocks of the Bernstein--Zelevinsky presentation of $\pi_1$ are not linked). Let $\phi'=\phi_1'\oplus\phi_0\oplus(\phi_1')^*$, where $(\cdot)^*$ denotes (conjugate-)dual. Then by Proposition \ref{proppacketdecomp} and induction in stages, it follows that $\Pi_{\phi}$ contains only irreducible direct summands of $\pi_{i_1}\rtimes \pi'$ for $\pi'\in\Pi_{\phi'}$.

Now by induction, we may assume we know the following: For any $\pi'\in\Pi_{\phi'}$, any irreducible constituent $\chi\boxtimes\sigma$ of $\Jac_{GL_{1/E_w}\times G_{n_v-n_{i_1}-1}}(\pi')$ has the character $\chi$ of the form $\chi=\rho_i\Vert\cdot\Vert_w^b$ or $\chi=\rho_i^*\Vert\cdot\Vert_w^b$ with $b\in\frac{1}{2}\Z$ and $\frac{N-\delta_{\mr{eo}}-1}{2}\leq b\leq \frac{N-\delta_{\mr{eo}}-1}{2}+\frac{n_v-1}{2}$ and $\rho_i$ associated with $r_i$ via class field theory for some $i\in I_0\sqcup I_1$ with $\dim(r_i)=1$. With this assumption, we will use Proposition \ref{propjacM1gen} to compute the possible characters occurring in the constituents of $\Jac_{M}(\pi_{i_1}\rtimes \pi')$ for $\pi'\in\Pi_{\phi'}$.

First, if $\dim(r_{i_1})>1$, then the Jacquet modules
\[\Jac_{GL_1\times GL_{n_{i_1}-1}}(\pi_{i_1})\qquad\textrm{and}\qquad \Jac_{GL_{n_{i_1}-1}\times GL_1}(\pi_{i_1})\]
are both zero, because the smooth irreducible representation associated with $r_{i_1}$ is supercuspidal and is not a character. So by Proposition \ref{propjacM1gen}, the only possible characters occurring in the constituents of $\Jac_{M}(\pi_{i_1}\rtimes \pi')$ for $\pi'\in\Pi_{\phi'}$ are of the form $\chi\Vert\cdot\Vert_w^{n_{i_1}}$, with $\chi$ as above in our induction hypothesis. (These are the characters occurring in the constituent $\pi_s$ of Proposition \ref{propjacM1gen}.) Such characters are therefore of the form $\rho_i\Vert\cdot\Vert_w^{b+n_{i_1}}$ for $i\in I_0$ with $r_i$ a character, or $\rho_i^{*}\Vert\cdot\Vert_w^{b+n_{i_1}}$ for $i\in I_1$, $i\ne i_1$, where $b\in\frac{1}{2}\Z$ has
\[\frac{N-\delta_{\mr{eo}}-2n_{i_1}-1}{2}\leq b\leq \frac{N-\delta_{\mr{eo}}-2n_{i_1}-1}{2}+\frac{n_v-n_{i_1}-1}{2}.\]
Since
\[\frac{N-\delta_{\mr{eo}}-1}{2}\leq b+n_{i_1}\leq \frac{N-\delta_{\mr{eo}}-1}{2}+\frac{n_v-n_{i_1}-1}{2}<\frac{N-\delta_{\mr{eo}}-1}{2}+\frac{n_v-1}{2},\] this finishes the induction step when $\dim(r_{i_1})>1$.

If, on the other hand, $\dim(r_{i_1})=1$, then $n_{i_1}=a_{i_1}$, and so $\pi_{i_1}=\St_{n_{i_1}}\otimes \rho_{i_1}$, where $\St_{n_{i_1}}$ denotes the Steinberg representation of $GL_{n_{i_1}}(E_w)$ and $\rho_{i_1}$ is again the character of $GL_1(E_w)$ associated with $r_{i_1}$ by class field theory. Any $\pi'\in\Pi_{\phi'}$ still has the possible constituents as described above, but then also we have
\[\Jac_{GL_1\times GL_{n_{i_1}-1}}(\pi_{i_1})=\Jac_{GL_1\times GL_{n_{i_1}-1}}(\St_{n_{i_1}}\otimes \rho_{i_1})=\rho_{i_1}\Vert\cdot\Vert_w^{n_{i_1}-1}\boxtimes(\St_{n_{i_1}-1}\otimes\rho_i\Vert\det\Vert_w^{-1}),\]
and similarly,
\[\Jac_{GL_{n_{i_1}-1}\times GL_1}(\pi_{i_1})=(\St_{n_{i_1}-1}\otimes\rho_i\Vert\det\Vert_w)\boxtimes\rho_{i_1}\Vert\cdot\Vert_w^{1-n_{i_1}}.\]
The constituents of $\Jac_{M_1}(\pi_{i_1}\rtimes \pi')$ coming from the representations called $\pi_1$ and $\pi_l$, respectively, in Proposition \ref{propjacM1gen} then have character components given by
\[\rho_{i_1}\Vert\cdot\Vert_w^{n_{i_1}-1}\Vert\cdot\Vert_w^{(N-\delta_{\mr{eo}}-n_{i_1})/2}=\rho_{i_1}^{\pm 1}\Vert\cdot\Vert_w^{(N-\delta_{\mr{eo}}-1)/2+(n_{i_1}-1)/2},\]
and, respectively, by
\[\rho_{i_1}^{*}\Vert\cdot\Vert_w^{(N-\delta_{\mr{eo}}-1)/2+(n_{i_1}-1)/2},\]
Since $n_{i_1}\leq n_v$, we have
\[\frac{N-\delta_{\mr{eo}}-1}{2}\leq\frac{N-\delta_{\mr{eo}}-1}{2}+\frac{n_{i_1}-1}{2}\leq \frac{N-\delta_{\mr{eo}}-1}{2}+\frac{n_v-1}{2}.\]
This finishes the induction step when $\dim(r_{i_1})=1$, hence in general.
\end{proof}

\section{Types and Hecke operators}
\label{sectypes}

We continue with the notation of the previous section; So we have our group $G$ with parabolic subgroup $P=MN$, our extension of number fields $E/F$, and our finite place $v$ of $F$ with the unique place $w$ of $E$ above it. However, we assume in this section that $v\nmid 2$. We let $\varpi_w$ be a fixed uniformizer in $E_w$.

Our goal in this section is to construct certain Hecke operators using the theory of types and, in particular, the theory of covers \cite[\S 7-8]{BuKu}. These operators will be the ones which we will use below to control the local monodromy at $v$ in the families of Galois representations coming from our Hecke families.

So fix throughout this section a standard $F_v$-Levi subgroup $L$ of $G$ and supercuspidal representation $\tau$ of $L(F_v)$. Recall that a \textit{supercuspidal type} for $(\tau,L)$ in $G_{/F_v}$ consists of a smooth, finite dimensional (complex) representation $\varrho$ of an open compact subgroup $K_{\varrho}$ of $G(F_v)$ such that, for any smooth irreducible representation $\pi$ of $G(F_v)$, we have that $\varrho$ occurs in $\pi|_{K_{\varrho}}$ if and only if $\pi$ is a constituent of an induced representation $\Ind_{L(F_v)}^{G(F_v)}(\tau\otimes\chi)$ for some unramified character $\chi$ of $L(F_v)$.

Now Miyauchi and Stevens \cite{MS} have constructed supercuspidal types for any pair $(\tau,L)$ in $G_{/F_v}$ as above, and have done so in such a way that they are \textit{covers} of their restrictions to $L(F_v)$, in the sense of \cite[Definition 8.1]{BuKu}. We note that \cite{MS} assumes that the residue characteristic of the nonarchimedean local field over which they work is not $2$, and this is why we assume $v\nmid 2$ here. In any case, we do not need to recall the definition of covers here, but instead will be content with the following proposition which summarizes all of the properties associated with covers that we need for our purposes.

\begin{proposition}
\label{propexistenceofcovers}
Let $(\tau,L)$ be the fixed pair consisting of an $F_v$-Levi subgroup $L$ of $G$ and a supercuspidal representation $\tau$ of $L$, as above. We continue to assume $v\nmid 2$. Then there exists a supercuspidal type $\varrho$ for $(\tau,L)$ in $G_{/F_v}$, say defined on an open compact subgroup $K_{\varrho}$ of $G(F_v)$, with the following properties. Let $L'$ be any standard $F_v$-Levi subgroup of $G$ with $L\subset L'\subset G$. Then we have:
\begin{enumerate}[label=(\roman*)]
\item The representation $\varrho|_{K_{\varrho}\cap L'}$ is a supercuspidal type for $(\tau,L)$ in $L'$;
\item First, write $U'$ for the unipotent radical of the standard $F_v$-parabolic subgroup of $G_{/F_v}$ containing $L'$, and given smooth admissible representation $\pi$ of $G(F_v)$, let $p_{U'}$ be the natural projection map
\[p_{U'}:\pi\to\Jac_{L'}(\pi).\]
Next, write $e_{\varrho,L'}$ for the idempotent in the Hecke algebra $C_c^\infty(L'(F_v),\C)$ associated with $\varrho|_{K_{\varrho}\cap L'}$, so
\[e_{\varrho,L'}(g)=\begin{cases}
\frac{\dim\varrho}{\vol(K_{\varrho}\cap L'(F_v))}\tr(\rho(g^{-1}))&\textrm{if }g\in K_{\varrho}\cap L'(F_v);\\
0&\textrm{if }g\in L'(F_v)\textrm{ but }g\notin K_{\varrho}.
\end{cases}\]
Then there is an algebra map
\[t_{L',G}:e_{\varrho,L'}*C_c^\infty(L'(F_v),\C)*e_{\varrho,L'}\to e_{\varrho,G}*C_c^\infty(G(F_v),\C)*e_{\varrho,G}\]
such that, for any smooth admissible representation $\pi$ of $G(F_v)$, any $v\in\pi$, and any $\phi\in e_{\varrho,L'}*C_c^\infty(L'(F_v),\C)*e_{\varrho,L'}$, we have
\begin{equation}
\label{eqnBuKuJac}
\phi * p_{U'}(v)=p_{U'}(t_{L',G}(\phi)* v).
\end{equation}
\item In fact, the map $p_{U'}$ induces an isomorphism
\[e_{\varrho,G}*\pi\cong e_{\varrho,L'}*\Jac_{L'}(\pi).\]
\end{enumerate}
\end{proposition}

\begin{proof}
As already noted, the main construction of \cite{MS} gives a supercuspidal type $\varrho$ for $(\tau,L)$ in $G_{/F_v}$ which is a cover of its restriction to $L(F_v)$. Then properties (i)-(iii) follow from basic properties of covers; property (i) is part of the \textit{transitivity} of covers \cite[Proposition 8.5]{BuKu}, property (ii) follows from \cite[Corollary 7.12]{BuKu}, and property (iii) follows from \cite[Theorem 7.9]{BuKu}.
\end{proof}

We now construct the Hecke operators we will need using the types from the proposition above.

\begin{definition}
\label{defheckeoptype}
With $(\tau,L)$ as above, assume moreover that $L\subset M$. Let $\varrho$ be a supercuspidal type for $(\tau,L)$ in $G_{/F_v}$ satisfying the conclusion of Proposition \ref{propexistenceofcovers}; then we have idempotents $e_{\varrho,M}\in C_c^\infty(M(F_v),\C)$ and $e_{\varrho,G}\in C_c^\infty(G(F_v),\C)$, as well as the map $t_{M,G}$, from Proposition \ref{propexistenceofcovers} (ii) with $L'=M$.

Let $\varpi_w$ be a fixed uniformizer in $E_w$. Identifying $M(F_v)=E_w^\times\times H(F_v)$, write $\varpi_w\times 1\in M(F_v)$ for the corresponding element of $E_w^\times\times H(F_v)$.

Write $K'$ for the kernel of the representation $\varrho$; it is a finite index open subgroup of the group $K_\varrho$ on which $\varrho$ is defined, and hence $K'$ is an open compact subgroup of $G(F_v)$. Then we define the Hecke operator $\phi_{\varrho}\in e_{\varrho,G}*C_c^\infty(G(F_v),\C)*e_{\varrho,G}$ by
\[\phi_{\varrho}=t_{M,G}\left(e_{\varrho,M}*\frac{\chars((\varpi_w\times 1)(K'\cap M(F_v)))}{\vol(K'\cap M(F_v))}*e_{\varrho,M}\right).\]
\end{definition}

\begin{remark}
The operator $\phi_{\varrho}$ just defined is so constructed as to allow us to compare its action on a given representation $\pi$ of $G(F_v)$ with the action of the element $(\varpi_w\times 1)\in M(F_v)$ on the Jacquet module $\Jac_{M}(\pi)$; such a comparison is provided by the relation \eqref{eqnBuKuJac} in property (ii) of Proposition \ref{propexistenceofcovers}.

More precisely, the constituents of $\Jac_{M}(\pi)$ are of the form $\chi\boxtimes\sigma$ for representations $\sigma$ of $H(F_v)$ and characters $\chi$ of $E_w^\times$. The operator
\[\frac{\chars((\varpi_w\times 1)(K'\cap M(F_v)))}{\vol(K'\cap M(F_v))}\]
will act on the $(K'\cap M(F_v))$-fixed vectors of such a constituent by the element $\varpi_w\times 1\in M(F_v)$, and hence by the eigenvalue $\chi(\varpi_w)$. Therefore, the intermediate operator
\[e_{\varrho,M}*\frac{\chars((\varpi_w\times 1)(K'\cap M(F_v)))}{\vol(K'\cap M(F_v))}*e_{\varrho,M}\]
will act nontrivially on those constituents $\chi\boxtimes\sigma$ which are of type $(\tau,L)$ with eigenvalues $\chi(\varpi_w)$ and $0$, and it will act on any other constituent as $0$. Thus the operator $\phi_\varrho$ acts on $\pi$ with eigenvalues $0$ and $\chi(\varpi_w)$ as $\chi\boxtimes\sigma$ ranges over all the constituents of $\Jac_{M}(\pi)$ of type $(\tau,L)$.

Now if $\pi$ is tempered, the numbers $\chi(\varpi_w)$ for such $\chi$, and hence the possible eigenvalues of $\phi_\varrho$ on $\pi$, are restricted to a particular finite set by Proposition \ref{propjacoftempered}, depending on the parameter for $\pi$. Moreover, in the same way, the eigenvalues of $\phi_\varrho$ on Eisenstein series induced from representations of type $(\tau,L)$ at $v$ are described (partially) by Corollary \ref{corjacofeisenstein}. The comparison of these possible eigenvalues will then be the key to the proof of our main theorem, Theorem \ref{thmmainthm}.
\end{remark}

\section{Construction of Galois representations}
\label{secgaloisreps}
We now proceed with the global part of our setup. From Section \ref{secgroups} we have our totally real number field $F$, our extension of number fields $E/F$, and our group $G$. We recall that the extension $E/F$ is trivial if $G$ is special orthogonal or symplectic, and otherwise it is the imaginary quadratic extension used to define $G$ if $G$ is unitary.

As usual, let $N$ be the dimension of the standard representation of $G^\vee$. We will construct the Galois representations attached to cohomological automorphic representations of $G(\A_F)$ via the functoriality results between $G$ and $GL_N$ available to us. The content of this section is rather standard, but we provide the construction of Galois representations in some detail. In particular, it will is crucial for our method to be able to link the local parameters of automorphic representations at bad places $v$ of $F$ to the restrictions of their Galois representations to the decomposition groups at places above $v$ in $E$.

Throughout this section, fix any prime number $p$ and an isomorphism $\iota:\C\overset{\sim}{\longrightarrow}\overline{\Q}_p$. Let $\pi$ be an automorphic representation of $G(\A_F)$ appearing in the discrete spectrum
\[L_{\disc}^2(G(F)\backslash G(\A_F)).\]
Since $\pi$ appears in $L_{\disc}^2(G(F)\backslash G(\A_F))$, by the facts about Arthur's classification recalled in Section \ref{secparams}, there is a global parameter
\begin{equation}
\label{eqnartparamforpi}
\psi=\boxplus_{i=1}^r\widetilde\pi_i[d_i],
\end{equation}
where each $\widetilde\pi_i$ is a cuspidal automorphic representation of some $GL_{n_i}(\A_E)$, each $d_i\geq 1$ is a positive integer, and $\sum n_i d_i=N$; here, the representations $\widetilde\pi_i$ are required to be self dual if $G$ is not unitary, and otherwise in the unitary case, we only require that the parameter $\psi$ itself is conjugate-self dual. This parameter has the following property under localization. Let $v$ be any place of $F$ and $w$ any place of $E$ above $v$. For each $i$, let $\tilde\phi_{i,w}$ be the local $L$-parameter at $w$ for the local component $\widetilde\pi_{i,w}$ of $\pi_i$. Let $\tilde\phi_w$ be the local $L$-parameter for $GL_{N}(E_w)$ given by
\[\tilde\phi_w=\bigoplus_{i=1}^r\left((\tilde\phi_{i,w}\otimes\Vert\cdot\Vert_w^{(d_i-1)/2})\oplus(\tilde\phi_{i,w}\otimes\Vert\cdot\Vert_w^{(d_i-3)/2})\oplus\dotsb\oplus(\tilde\phi_{i,w}\otimes\Vert\cdot\Vert_w^{(1-d_i/2)})\right).\]
This is the $L$-parameter associated with the $A$-parameter that was written as $\tilde\psi_w$ in Section \ref{secparams}. Finally let $\restr_w(\phi_{\psi_v})$ be the restriction of the $L$-parameter $\phi_{\psi_v}$ associated with $\psi_v$ to either $W_{E_w}$ or $W_{E_w}\times SL_{2,\mr{D}}(\C)$, depending on whether $v$ is archimedean or not. Then
\[\std\circ(\restr_w(\phi_{\psi_w}))=\tilde\phi_w,\]
where, as in formula \eqref{eqnpsiwtilde}, the map $\std$ denotes the standard embedding of $L$-groups defined as in that formula.

\begin{theorem}
\label{thmconstofgalois}
Let the notation be as above, so in particular, we have our automorphic representation $\pi$ of $G(\A_F)$ appearing in
\[L_{\disc}^2(G(F)\backslash G(\A_F)).\]
Assume that $\pi_{v_\infty}$ is cohomological for any archimedean place $v_\infty$ of $F$. If $G$ is even orthogonal, assume moreover that each $\pi_{v_\infty}$ is $\std$-regular (see Definition \ref{defstdreg}). Then there is a continuous, semisimple Galois representation $\rho_\pi:G_E\to GL_{N}(\overline\Q_p)$ (which depends on the embedding $\iota$ just fixed, though we suppress this from the notation) satisfying the following.
\begin{enumerate}[label=(\alph*)]
\item We have that $\rho_{\pi}^*(1+\delta_{\mr{eo}}-N)\cong\rho_\pi$, where $\delta_{\mr{eo}}\in\{0,1\}$ is as in Notation \ref{notndeltaeo}, and $(\cdot)^*$ denotes the linear dual unless $G$ is unitary, in which case it denote the conjugate dual, i.e., if $c$ is any complex conjugation in $G_F$, then $\rho_{\pi}^*(g)=\rho_\pi^\vee(cgc)$ for $g\in G_E$.
\item If $G$ is unitary, assume in addition that each $\widetilde{\pi}_i$ is conjugate-self dual. If $v$ is any finite place of $F$ with $v\nmid p$ and $w$ is a place in $E$ with $w|v$, then the Frobenius semisimple Weil--Deligne representation $\WD(\rho_{\pi}|_{G_{E_w}})^{\mr{F-ss}}$ has associated $L$-parameter $\tilde\phi_{w}\otimes\Vert\cdot\Vert_w^{(1+\delta_{\mr{eo}}-N)/2}$, in the sense of Section \ref{secparams}. If instead some $\widetilde{\pi}_i$ is not conjugate self-dual, this conclusion still holds for all unramified places for $\pi$ in $F$ which are split in $E$.
\end{enumerate}
\end{theorem}

\begin{proof}
First note that, since $\pi_{v_\infty}$ is assumed cohomological for any archimedean place $v_\infty$ of $F$, by Proposition \ref{propparamliftcoho}, we have that $\pi_{v_\infty}\otimes\Vert\cdot\Vert_{v_\infty}^{(1+\delta_{\mr{eo}}-N)/2}$ has regular integral infinitesimal character for any such $v_\infty$. If $w_\infty$ is the place of $E$ lying over $v_\infty$, then it follows that for $i=1,\dotsc r$ and any $j=1-d_i,3-d_i,\dotsc,d_i-1$ (with $r$ and $d_i$ as in \eqref{eqnartparamforpi}), we have that $\tilde{\pi}_{i,w_\infty}\otimes\Vert\cdot\Vert_{w_\infty}^{(j+1+\delta_{\mr{eo}}-N)/2}$ has regular integral infinitesimal character.

Then there are two cases. If $n_i$ is odd, the results of Chenevier--Harris, namely  \cite[Theorem 4.2]{CH} if $E=F$ or \cite[Theorem 3.2.3]{CH} if $E/F$ is imaginary quadratic, attach a continuous, semisimple Galois representation $\rho_{i,j}$ to $\widetilde\pi_i\otimes\Vert\cdot\Vert_E^{(j+1+\delta_{\mr{eo}}-N)/2}$. This representation has the compatibility property that for any finite place $w$ of $E$ with $w\nmid p$, we have that $\WD(\rho_{i,j}|_{G_{E_w}})^{\mr{F-ss}}$ has associated $L$-parameter $\tilde\phi_{i,w}\otimes\Vert\cdot\Vert_w^{(j+1+\delta_{\mr{eo}}-N)/2}$.

If instead $n_i$ is even, then the representation
\[\widetilde\pi_i\otimes\Vert\cdot\Vert_E^{(j+1+\delta_{\mr{eo}}-N)/2}\otimes \Vert\cdot\Vert_E^{(n_i-1)/2}\]
has half-integral, but not integral, infinitesimal character at any infinite place of $E$, and is thus cohomological (as the weight $\rho$ for $GL_{n_i}$ is half-integral but not integral). Then \cite{CH} when each $\widetilde{\pi}_i$ is (conjugate-)self dual, along with \cite{caravnep} for the compatibility property, or otherwise \cite{HLTT} attach a Galois representation to the twist of this by $\Vert\cdot\Vert_E^{(1-n_i)/2}$, satisfying analogous properties; hence a Galois representation $\rho_{i,j}$ with the right properties exists for $\widetilde\pi_i\otimes\Vert\cdot\Vert_E^{(j+1+\delta_{\mr{eo}}-N)/2}$.

We then define
\[\rho_\pi=\bigoplus_{i}\bigoplus_{j}\rho_{i,j},\]
where $i$ ranges over $1,\dotsc,r$, and $j$ over $1-d_i,3-d_i,\dotsc,d_i-1$. Then $\rho_\pi$ satisfies (b) by construction and the facts recalled in the previous paragraph. It also satisfies (a); this follows from Chebotarev along with the (conjugate-)self duality of each local Galois representation up to twist, which itself follows from the (conjugate-)self duality of the local parameters.
\end{proof}

In the case that $G$ is not unitary, so when $E=F$, the individual Galois representations attached to each $\widetilde\pi_i$ as above are constructed via a \textit{patching} argument, described as follows. One first constructs the Galois representation attached to the base change $\widetilde{\pi}_{i,E'}$ of a certain twist of $\widetilde\pi_i$ to various imaginary quadratic fields $E'$ for which $\widetilde\pi_{i,E'}$ is still cuspidal, and then proves that there is a Galois representation over $F$ that restricts to those over the varying fields $E'$. All of the local--global compatibility properties follow from those for $E'$, since one can always choose $E'$ so that a given finite place $v$ of $F$ splits in $E'$; in this case, if $w$ is a place of $E'$ above $v$, then $F_v=E_w'$ and $(\widetilde\pi_{i,E'})_w$ is thus a twist of $\widetilde\pi_{i,v}$. See \cite[\S 4.4]{sorensen} for a reference where this is done carefully.

Recall that in this setting, each $\widetilde\pi_i$ is an automorphic representation of $GL_{n_i}(\A_F)$ for some $n_i$. For the Galois representations $\rho_{\widetilde\pi_{i,E'}}$ attached to the $\widetilde\pi_{i,F}$ described above, Shin \cite[Corollary 7.15]{shin} and Caraiani \cite[Corollary 5.9]{caravnep} prove that for any finite place $w$ of $E'$ not dividing $p$, the associated Frobenius-semisimple Weil--Deligne representation
\[\WD(\rho_{\widetilde\pi_{i,E'}}|_{G_{E_w'}})^{\varphi-\sss}\]
is pure, in the sense defined at the end of Section \ref{secparams}. We thus obtain the following.

\begin{proposition}
\label{proppurityofgal}
Notation as set at the beginning of this section, assume that the parameter $\psi=\boxplus_{i=1}^r\widetilde\pi_i[d_i]$ is generic, i.e., has $d_i=1$ for all $i$. If $G$ is unitary, assume in addition that $\widetilde{\pi}_i$ is conjugate-self dual for all $i$. Assume also that $\pi_{v_\infty}$ is cohomological for any archimedean place $v_\infty$ of $F$, and if $G$ is even orthogonal, assume moreover that each $\pi_{v_\infty}$ is $\std$-regular. Then the Galois representation $\rho_\pi$ constructed in Theorem \ref{thmconstofgalois} has that $\WD(\rho_{\pi}|_{G_{E_w}})^{\varphi-\sss}$ is pure of weight $1+\delta_{\mr{eo}}-N$ for every finite place $w\nmid p$ of $E$.
\end{proposition}

\begin{proof}
If $G$ is not unitary, then $E=F$ and, taking $v=w$ with $w$ as in the statement of the proposition, the purity follows from the above discussion upon taking one of the fields $E'$ to split the finite place $v$ under consideration. If $G$ is unitary, then we do not need to base change any $\widetilde\pi_i$ as each such representation is already over the imaginary quadratic extension $E$ of $F$. So the same discussion for $\widetilde\pi_i$ in place of $\widetilde{\pi}_{i,E'}$ implies the purity. The claim about the weight in either case follows from property (a) of Theorem \ref{thmconstofgalois}.
\end{proof}

Let us give a name to some of the hypotheses used in this proposition.

\begin{definition}
\label{defstrgen}
Let $\pi$ be a discrete automorphic representation of $G(\A_F)$ with parameter $\psi=\boxplus_{i=1}^r\widetilde\pi_i[d_i]$, as above. We will say $\pi$ is \textit{strongly generic} if $\pi$ is generic (i.e., $d_i=1$ for all $i$) and, moreover, if each $\widetilde{\pi}_i$ is conjugate-self dual if $G$ is unitary.
\end{definition}

This notion is important for us as purity will play a strong role in the combinatorial aspects of the proof of our main theorem. For the applications we have in mind which use forms of Ribet's method, the relevant Galois representations which are attached to such $\pi$ will be irreducible; this will force $r=1$ and $d_1=1$, making the lift to $GL_N(\A_E)$ not only strongly generic, but actually cuspidal, equal to $\widetilde{\pi}_1$.

\section{Hecke families}
The purpose of this section is to simply make and discuss the definition of Hecke families that we will use in this paper. We continue to use the group $G$ and the fields $E$ and $F$, and we identify the Levi $M$ of the parabolic subgroup $P$ with $GL_{1/E}\times H$ as usual. As in the previous section we fix a prime number $p$ and an isomorphism $\iota:\C\overset{\sim}{\longrightarrow}\overline{\Q}_p$. The definition we study is the following.

\begin{definition}
\label{defheckefamilies}
Let $K_f$ be a factorizable, open compact subgroup of $G(\A_{F,f})$. A \textit{Hecke family of level} $K_f$ \textit{for} $G$ (or simply \textit{Hecke family}, if the rest is implicit) is a tuple $\mc{F}=(\mb{T},\mf{X},\Sigma,\Psi)$, where
\begin{itemize}
\item $\mb{T}$ is a sub-$\Q_p$-algebra (with unit) of $C_c^\infty(K_f\backslash G(\A_{F,f})/K_f,\overline{\Q}_p)$ generated over $\Q_p$ by all the $\Q_p$-valued spherical operators at all but finitely many places where $K_f$ is hyperspecial, as well as finitely many other operators in $C_c^\infty(K_f\backslash G(\A_{F,f})/K_f,\overline{\Q}_p)$;
\item $\mf{X}$ is an affinoid rigid space over $\Q_p$;
\item $\Sigma$ is a Zariski dense subset of $\mf{X}(\overline \Q_p)$;
\item $\Psi:\mb{T}\to\mc{O}(\mf{X})$ is a $\Q_p$-linear map;
\end{itemize}
such that the following holds: First, let us write $\iota^{-1}(\mb{T})$ for the $\C$-subalgebra of
\[C_c^\infty(K_f\backslash G(\A_{F,f})/K_f,\C)\]
generated by the operators given by composing functions in $\mb{T}$ with $\iota^{-1}$, so:
\[\iota^{-1}(\mb{T})=\Span_\C\sset{\iota^{-1}\circ\phi}{\phi\in\mb{T}}.\]
Then we require that for any $x\in\Sigma$, there is a cohomological automorphic representation $\pi_{x}$ of $G(\A_F)$, which is $\std$-regular at every archimedean place of $F$ if $G$ is even orthogonal (see Definition \ref{defstdreg}), and a nontrivial constituent $V_x$ of the semisimplification of $\pi_{x,f}^{K_f}$ as an $\iota^{-1}(\mb{T})$-module, which we write as $(\pi_{x,f}^{K_f})^{\mb{T}-\sss}$, with the property that
\begin{equation}
\label{eqnheckefamcomp}
x(\Psi(\phi))=\iota(\tr(\iota^{-1}\circ\phi|V_x)),
\end{equation}
for all $\phi\in\mb{T}$.

We say the Hecke family $\mc{F}$ is \textit{strongly} $\Sigma$-\textit{generic} if each of the representations $\pi_x$ for $x\in\Sigma$ can be taken to be discrete and strongly generic in the sense of Definition \ref{defstrgen}.

Finally, we say a point $x_0\in\mf{X}(\overline{\Q}_p)$ is \textit{strongly} $M$-\textit{Eisenstein} if there is an automorphic representation $\sigma$ of $H(\A_F)$ and a Hecke character $\chi$ of $GL_1(\A_E)$ such that the following are satisfied:
\begin{itemize}
\item The automorphic representation $\sigma$ is discrete, strongly generic, and cohomological at all archimedean places of $F$, and it is $\std$-regular at all archimedean places of $F$ if $G$ is even orthogonal;
\item The parabolically induced representation
\[\Pi_{x_0}=\Ind_{P(\A_F)}^{G(\A_F)}(\chi\boxtimes \sigma)\]
has, at all archimedean places $v_\infty$ of $F$, infinitesimal character of the form $\lambda_{v_\infty}+\rho$ for some integral weight $\lambda_{v_\infty}$ of $G^\vee$;
\item There is a nontrivial constituent $V_{x_0}'$ of $(\Pi_{x_0,f}^{K_f})^{\mb{T}-\sss}$ as a $\iota^{-1}(\mb{T})$-representation, and there is another finite dimensional $\iota^{-1}(\mb{T})$-representation $V_{x_0}''$, such that, writing $V_{x_0}=V_{x_0}'\oplus V_{x_0}''$, we then have
\[x_0(\Psi(\phi))=\iota(\tr(\iota^{-1}\circ\phi|V_{x_0})),\]
for all $\phi\in\mb{T}$.
\end{itemize}
For clarity, the weight $\rho$ is the half sum of positive roots for $G^\vee$, and $(\cdot)^{\mb{T}-\sss}$ has the same meaning as above. If these conditions hold, we say that the point $x_0$ is \textit{induced from} $(\chi,\sigma)$.
\end{definition}

\begin{remark}
We make several remarks about this definition in order to acquaint the reader with the concepts just posed.
\begin{enumerate}[label=(\arabic*)]
\item The definition depends on the fixed isomorphism $\iota$.

\item For clarity, the unit in $C_c^\infty(K_f\backslash G(\A_{F,f})/K_f,\overline{\Q}_p)$ is the operator
\[1_{K_{f}}=\frac{1}{\vol(K_{f})}\chars(K_f).\]
Thus, in particular, we are assuming $\mb{T}$ contains this operator $1_{K_f}$.

\item Again for clarity, let us explain the condition on $\mb{T}$ about spherical operators in the definition. Let $S$ be the set of finite places away from which $\mb{T}$ is supposed to contain all spherical Hecke operators; by assumption, this set contains finitely many finite places of $F$, and $K_f$ is hyperspecial at any place not in $S$. Since $K_f$ is assumed to be factorizable, we may write $K_f=K_{S}K_f^{S}$ where $K_{S}$ is an open compact subgroup of $G(\A_{F,S})$ and $K_f^S$ is a factorizable open compact subgroup of $G(\A_{F,f}^S)$ which is hyperspecial at every finite place not in $S$. Correspondingly, we have
\[C_c^\infty(K_f\backslash G(\A_{F,f})/K_f,\Q_p)=C_c^\infty(K_{S}\backslash G(\A_{F,S})/K_{S},\Q_p)\otimes_{\Q_p} C_c^\infty(K_f^S\backslash G(\A_{F,f}^S)/K_f^S,\Q_p),\]
as algebras, and the second algebra in the product, namely $C_c^\infty(K_f^S\backslash G(\A_{F,f}^S)/K_f^S,\Q_p)$, is the spherical Hecke algebra away from $S$. The operator
\[1_{K_{S}}=\frac{1}{\vol(K_{S})}\chars(K_S)\]
is the identity element of the first algebra in the product, $C_c^\infty(K_{S}\backslash G(\A_{F,S})/K_{S},\Q_p)$. Thus the subalgebra
\[1_{K_S}\otimes C_c^\infty(K_f^S\backslash G(\A_{F,f}^S)/K_f^S,\Q_p)\]
sits in $C_c^\infty(K_f\backslash G(\A_{F,f})/K_f,\overline\Q_p)$, and we have required that $\mb{T}$ contains this subalgebra.

\item A natural source of Hecke families is given by Hida theory and the theory of eigenvarieties. Although, in these theories, the variation over the weight is a crucial aspect, we have chosen not to keep track of the weight in our general definition. This will not affect our main theorem, Theorem \ref{thmmainthm}, and it will also allow us to be more general; Theorem \ref{thmmainthm} is proved using the purity of the Galois representations involved and makes no use of the variation of the weights of the automorphic representations to which these Galois representations are attached.

\item Let us explain how Hida theory gives Hecke families. If $G$ is not orthogonal, then versions of Hida theory exist for $G$ and will give us a finitely generated module $\mc{M}$ over an Iwasawa algebra $\Lambda=\mc{O}[\![T_1,\dotsc,T_{n'}]\!]$ of some relative dimension $n'$, which depends on the group $G$ and the field $F$, over a finite extension $\mc{O}$ of $\Z_p$. We view $\mc{O}$ as a subring of $\overline\Q_p$. The module $\mc{M}$ is a space of $p$-adic families of ordinary modular forms varying in weight in a particular way which is parametrized by $\Lambda$, and these families are of a given level away from $p$ (i.e., of given \textit{tame level}) which we eventually take to be the component $K_f^p$ of our $K_f$.

Now let $S$ be the set of finite places of $F$ containing only those above $p$ and all the places at which the tame level $K_f^p$ is not hyperspecial, and let $K_f^S$ be the factor of the tame level subgroup $K_f^p$ away from $S$. Then the module $\mc{M}$ has a natural $\Z_p$-linear action of the spherical Hecke algebra
\[C_c^\infty(K_f^S\backslash G(\A_{F,f}^S)/K_f^S,\Z_p).\]
There is also an action at $p$ of an algebra of $U_p$-operators on these forms, which form a commutative subalgebra of the Hecke algebra at $p$ with level given by some Iwahori subgroup $I_p$. However, this action only coincides with the usual action of the Iwahori Hecke algebra on automorphic representations after a specific normalization, and since the operators at $p$ play no role in our main theorem, we will feel free to ignore them.

Now at the expense of extending scalars to a finite integral extension $\mb{I}$ of $\Lambda$, we can then find eigenfamilies for the spherical Hecke algebra $C_c^\infty(K_f^S\backslash G(\A_{F,f}^S)/K_f^S,\Z_p)$; this algebra will act by scalars such an eigenfamily $\mbf{f}$.

We can then build a Hecke family as follows. Let $I_p$ be a sufficiently deep Iwahori subgroup at the places above $p$ in $F$, and let $K_f=K_f^pI_p$. Let $\mf{X}$ be the rigid generic fiber of the $\Z_p$-formal scheme defined by $\mb{I}$, and $\Sigma\subset\mf{X}(\overline\Q_p)$ the set of points $x\in\mf{X}(\overline\Q_p)$ coming from those points $y\in\spec(\mb{I})(\overline\Q_p)$ of level $I_p$ which lie above classical weights of $\Lambda$. For such $x$, we take $V_x$ to be the $1$-dimensional $\C$-vector subspace spanned by $\iota^{-1}(\mbf{f}_{y})$ in the automorphic representation $\pi_x$ generated by $\iota^{-1}(\mbf{f}_{y})$, where $\mbf{f}_{y}$ is the specialization of $\mbf{f}$ at $y$ (and the $\iota^{-1}$ is there to signify that we are viewing $\mbf{f}_{y}$ with an automorphic form with values in $\C$ via $\iota^{-1}$). We take $\mb{T}$ to simply be the spherical Hecke algebra
\[\mb{T}=1_{K_S}\otimes C_c^\infty(K_f^S\backslash G(\A_{F,f}^S)/K_f^S,\Q_p).\]
Then we define the map
\[\Psi:\mb{T}\to\mc{O}(\mf{X})\]
to be the one that sends $\phi$ in the $\Z_p$-subalgebra
\[1_{K_S}\otimes C_c^\infty(K_f^S\backslash G(\A_{F,f}^S)/K_f^S,\Z_p)\subset\mb{T}\]
to its eigenvalue, call it $\lambda_{\phi,\mbf{f}}$, on $\mbf{f}$; we then extend $\Psi$ to $\mb{T}$ by $\Q_p$-linearity. Then we have
\[\phi\cdot\mbf{f}_y=x(\lambda_{\phi,\mbf{f}})\mbf{f}_y,\]
which is equivalent to the crucial compatibility property \eqref{eqnheckefamcomp} of the definition. So with these definitions, the tuple $(\mb{T},\mf{X},\Sigma,\Psi)$ just defined is a Hecke family. 

\item In order to apply our main theorem, one needs to include the operators constructed in Section \ref{sectypes} in the Hecke algebra $\mb{T}$. In view of the remarks just made about Hida theory, we note that Hida theory is usually constructed in such a way that it carries only an action of the spherical Hecke algebra described above. Nevertheless, one should be able to show that the module $\mc{M}$ described above carries also an action the full Hecke algebra away from $p$,
\[C_c^\infty(K_f^p\backslash G(\A_{F,f}^p)/K_f^p,\Z_p).\]
However, this algebra is generally far too large (and in particular, noncommutative) for us to expect to find simultaneous eigenfamilies $\mbf{f}$ for it. So we can instead take $\mb{T}$ the algebra generated by operators of the form
\[1_{K_p}\otimes T^S\otimes \bigotimes_{v\in S\backslash\{\textrm{primes over }p\}}T_v,\]
where $T^S\in C_c^\infty(K_f^p\backslash G(\A_{F,f}^p)/K_f^p,\Q_p)$ is spherical away from $S$, each $T_v$ is in the subalgebra of
\[C_c^\infty(K_v\backslash G(F_v)/K_v,\overline\Q_p)\]
generated by one of the Hecke operators defined in Definition \ref{defheckeoptype} (viewed with coefficients in $\overline\Q_p$ instead of $\C$ via $\iota$) for each $v\in S$ not above $p$, and where we assume that the factor $K_v$ of $K_f$ at $v$ is deep enough so that these operators are of level $K_v$. We note that, since we allow $T^S$ to have coefficients in $\overline\Q_p$, it is possible that we may have to replace the space $\mf{X}$ above with some restriction of scalars of $\mf{X}$ from a finite extension of $\Q_p$, and then lift $\Sigma$ accordingly.

In any case, this algebra $\mb{T}$ is then obviously commutative, and we can expect to find simultaneous eigenfamilies for it. If, for $v\in S$ not lying over $p$, any of the eigenvalues of the operators $T_v$ as above have nonzero eigenvalue on $\mbf{f}$, then in particular the automorphic representations $\pi_x$ for general $x\in\Sigma$ have supercuspidal type at $v$ given by one used in Definition \ref{defheckeoptype}.

\item More generally, let us allow $G$ to be odd orthogonal again as well as symplectic or unitary. Assume we are given a level subgroup $K_f^p\subset G(\A_{F,f}^p)$ which is hyperspecial away from a finite set $S$ of finite places of $F$ including the places above $p$. Then the theory of eigenvarieties gives us families of (not necessarily ordinary) automorphic eigensystems for the spherical Hecke algebra for $K_f^S$ (as well as the algebra of $U_p$ operators), varying over rigid spaces $\mf{Y}$ over $\Q_p$ which admit a finite map to a space of weights. If one uses the theory of Urban \cite{urbanev} to construct these families, then it is explained there (see \cite[Theorem 5.3.10]{urbanev} specifically) that, after passing to a finite cover of an open affinoid neighborhood of a given point $x_0\in\mf{Y}(\overline\Q_p)$, these families also interpolate the traces of any Hecke operators in $C_c^\infty(K_f^p\backslash G(\A_F^p)/K_f^p,\Q_p)$, not just spherical ones. Thus, to incorporate this aspect of the theory, we have allowed the dimension of the spaces $V_x$ for $x\in\Sigma$ in our definition of Hecke families to be arbitrary (though they are still finite dimensional, since each $\pi_{x,f}^{K_f}$ is), rather than just $1$-dimensional like for Hida families.

From this theory of eigenvarieties one can thus construct Hecke families $(\mb{T},\mf{X},\Sigma,\Psi)$ for $G$, where $\mf{X}$ is the aforementioned finite cover of a given neighborhood of $x_0$, and $\mb{T}$ is the full Hecke algebra away from $p$.

We remark here that, although the families $\mf{Y}$ of representations constructed in \cite[Theorem 5.3.10]{urbanev} are generically irreducible, their specializations at given points $x_0\in\mf{Y}(\overline\Q_p)$ can be reducible. This is why, in our definition of a strongly $M$-Eisenstein point $x_0$, we have allowed the space $V_{x_0}$ to split into $V_{x_0}'\oplus V_{x_0}''$ with $V_{x_0}'$ Eisenstein but $V_{x_0}''$ possibly non-Eisenstein.

\item In either of the examples above coming from Hida theory or eigenvarieties, if the families constructed there are generically cuspidal, then after restricting $\Sigma$ away from points of irregular weight, the corresponding Hecke families will be $\Sigma$-generic. But in general this condition of $\Sigma$-genericity is slightly weaker than generic cuspidality.

In any case, we will use this condition on the Galois side in order to know that for $x\in\Sigma$, the Galois representations attached to $\pi_x$ are all pure of the same weight.

\item Finally, we comment on the last part of the definition. The main motivation for this paper comes from studying the $p$-adic variation of cuspidal automorphic forms or representations which degenerate to Eisenstein series, i.e., from the theory of Eisenstein congruences. Such a study gives rise naturally, via the processes described above, to $\Sigma$-generic Hecke families for $G$ which contain a point (possibly many points) $x_0$ which are strongly $M$-Eisenstein. It is usually possible to choose the sections defining the relevant Eisenstein series at bad finite places $v$, and such choices will affect behavior of the Hecke actions on the spaces $V_{x_0}$.

The Galois representations attached to the representations $\pi_x$ for $x\in\Sigma$ will deform to those attached to Eisenstein series, and the Galois representations attached to Eisenstein series will be reducible. But the choice of sections for the Eisenstein series at $x_0$ at bad finite places $v$ will influence the behavior of the local components of the representations $\pi_x$ for $x\in\Sigma$ as long as we allow our families to see the actions of certain Hecke operators at $v$, such as the ones constructed in Section \ref{sectypes}. Thus, by local-global compatibility, this strategy will allow us some control over the restrictions of these Galois representations to the decomposition groups $G_{E_w}$ for $w$ in $E$ over such $v$; we will thus obtain control over the possible monodromy in the associated Weil--Deligne representations, which is very important if we are trying to use our families to construct Selmer classes via Ribet's method.
\end{enumerate}
\end{remark}

We end this section with a small lemma about Hecke families which will be useful in the proof of our main theorem.

\begin{lemma}
\label{lemdimVxisconst}
Let $\mc{F}=(\mb{T},\mf{X},\Sigma,\Psi)$ be a Hecke family, and assume $\mf{X}$ is Zariski connected. Then there exists a positive integer $D$ such that, for any $x\in\Sigma$, the space $V_x$ of Definition \ref{defheckefamilies} has dimension $D$ as a $\C$-vector space. Moreover, if $x_0\in\mf{X}(\overline\Q_p)$ is any strongly $M$-Eisenstein point, then we have as well that $\dim_\C(V_{x_0})=D$.
\end{lemma}

\begin{proof}
By assumption, the algebra $\mb{T}$ contains the unit
\[1_{K_f}=\frac{1}{\vol(K_f)}\chars(K_f)\]
of $C_c^\infty(K_f\backslash G(\A_{F,f})/K_f,\Q_p)$. Thus $\Psi(1_{K_f})$ is defined and analytic on $\mf{X}$. Let $x\in\Sigma$. Then by definition, using \eqref{eqnheckefamcomp}, we have
\begin{equation}
\label{eqntrofunitx}
x(\Psi(1_{K_f}))=\iota(\tr(\iota^{-1}\circ 1_{K_f}|V_x))=\dim_\C(V_x)\leq\dim_\C(\pi_{x,f}^{K_f}).
\end{equation}
Similarly, for any strongly $M$-Eisenstein point $x_0$, we have
\begin{equation}
\label{eqntrofunity}
x_0(\Psi(1_{K_f}))=\dim_\C(V_{x_0}).
\end{equation}
A result of Bernstein \cite[Theorem 1]{Bernstein} implies that the irreducible, finite dimensional representations of $C_c^\infty(K_f\backslash G(\A_{F,f})/K_f,\C)$ have bounded dimension, say by an integer $D'$. Thus we have $0<\dim_\C(\pi_{x,f}^{K_f})\leq D'$, the lower bound due to the fact that $V_x$ is nontrivial by definition. 

Thus by \eqref{eqntrofunitx}, the analytic function $\Psi(1_{K_f})$ takes on only finitely many possible values on the dense subset $\Sigma$, all of which are positive integers. So since we assumed $\mf{X}$ is connected, the function $\Psi(1_{K_f})$ is constant, necessarily given by some positive integer $D$. Therefore $\dim_\C(V_x)=D$ for any $x\in\Sigma$ by \eqref{eqntrofunitx}, and similarly for any strongly $M$-Eisenstein point $x_0$ by \eqref{eqntrofunity}.
\end{proof}

\section{Families of Galois representations}
\label{secgaloisfamilies}
We continue in this section to fix an isomorphism $\iota:\C\overset{\sim}{\longrightarrow}\overline{\Q}_p$. Let $K_f$ be a factorizable, open compact subgroup of $G(\A_{F,f})$, and fix also throughout this section a Hecke family $\mc{F}=(\mb{T},\mf{X},\Sigma,\Psi)$, as in Definition \ref{defheckefamilies}. Our goal now is to attach a family of Galois representations to $\mc{F}$. We will pass through the theory of pseudorepresentations to do this; we will not need to recall the full definition of pseudorepresentation here, which is given in \cite{taylorpsrep}. Suffice it to say that a pseudorepresentation of dimension $N_0$ of a group $G_0$ over a ring $R$ is, in particular, a function $T:G_0\to R$ with $T(1)=N_0$, and that if $\rho:G_0\to GL_N(R)$ is a representation then $\tr(\rho)$ is a pseudorepresentation.

We continue to denote by $N$ the dimension of the standard representation of $G^\vee$, and by $n$ the absolute rank of $G$.

\begin{proposition}
\label{propconstofpsrep}
Let $\mc{F}=(\mb{T},\mf{X},\Sigma,\Psi)$ be a Hecke family.
\begin{enumerate}[label=(\alph*)]
\item There exists a unique continuous, $N$-dimensional pseudorepresentation $T_{\mc{F}}:G_E\to\mc{O}(\mf{X})$ such that, for every $x\in\Sigma$, we have
\[x\circ T_{\mc{F}}=\tr\rho_{\pi_x},\]
where $\pi_x$ is the automorphic representation attached to the point $x$ as in Definition \ref{defheckefamilies}, and $\rho_{\pi_x}$ is the Galois representation attached to $\pi_x$ by Theorem \ref{thmconstofgalois} via $\iota$.
\item Moreover, let $x_0\in\mf{X}(\overline\Q_p)$ be a strongly $M$-Eisenstein point, say induced by some pair $(\chi,\sigma)$ for some algebraic Hecke character $\chi$ for $E$ and some discrete, strongly generic, cohomological automorphic representation $\sigma$ of $H(\A_F)$ which, if $G$ is even orthogonal, is $\std$-regular at all archimedean places of $F$ (see Definition \ref{defstdreg}). Then the Hecke characters $\chi\Vert\cdot\Vert^{(1+\delta_{\mr{eo}}-N)/2}$ and $\chi^*\Vert\cdot\Vert^{(1+\delta_{\mr{eo}}-N)/2}$ are algebraic, where $(\cdot)^*$ denotes (conjugate-)dual. Write $\rho_{\chi\Vert\cdot\Vert^{(1+\delta_{\mr{eo}}-N)/2}}$ for the character $G_E\to\overline\Q_p^\times$ associated with $\chi\Vert\cdot\Vert^{(1+\delta_{\mr{eo}}-N)/2}$ by class field theory via $\iota$, and similarly for $\rho_{\chi^*\Vert\cdot\Vert^{(1+\delta_{\mr{eo}}-N)/2}}$. Then
\[x_0\circ T_{\mc{F}}=\tr(\rho_{\sigma}(-1)\oplus \rho_{\chi\Vert\cdot\Vert^{(1+\delta_{\mr{eo}}-N)/2}}\oplus\rho_{\chi^*\Vert\cdot\Vert^{(1+\delta_{\mr{eo}}-N)/2}}),\]
where $\rho_\sigma$ is attached to $\sigma$ by Theorem \ref{thmconstofgalois}, and $\rho_{\sigma}(-1)$ denotes the twist of $\rho_{\sigma}$ by the inverse of the $p$-adic cyclotomic character.
\end{enumerate}
\end{proposition}

\begin{proof}
We first define a set $S_0$ of places of $E$ as follows, by splitting into three cases. If $G$ is unitary, let $S_0$ be the set of finite places $w'$ of $E$ which are split over a place $v'$ of $F$ at which the level $K_f$ of $\mc{F}$ is hyperspecial and at which $\mb{T}$ contains all spherical Hecke operators at $v'$. If $G=SO_{2n}^\eta$ with $\eta\ne 1$ let $S_0$ denote the set of finite places $w'=v'$ of $E=F$ at which $\eta$ is unramified, at which $K_f$ is hyperspecial, and at which $\mb{T}$ contains all spherical Hecke operators at $v'$. Finally, in all other cases, let $S_0$ denote the set of finite places $w'=v'$ of $E=F$ at which $K_f$ is hyperspecial and at which $\mb{T}$ contains all spherical Hecke operators at $v'$. Then in all cases, the set $S_0$ has density $1$ in the set of all places of $E$.

For $x\in\Sigma$, let $\rho_{\pi_x}$ be the Galois representation attached to $\pi_x$ by Theorem \ref{thmconstofgalois}, which exists because $\pi_x$ is discrete by our strong $\Sigma$-genericity hypothesis. To prove part (a) of our proposition, we will invoke \cite[Proposition 7.1.1]{chengln}, which immediately implies the statement at hand once we can show the following claim: There is a compact subring $R\subset\mc{O}(\mf{X})$ with the property that, for any $w'\in S_0$, there is an element $a_{w'}\in R$ such that
\begin{equation}
\label{eqnawxfrob}
a_{w'}(x)=\tr(\Frob_{w'}^{-1}|\rho_{\pi_x}),
\end{equation}
for any $x\in\Sigma$. Here $\Frob_w$ denotes an arithmetic Frobenius element. In fact, we will take $R=\mc{O}(\mf{X})^\circ$, the ring of power bounded elements in $\mc{O}(\mf{X})$, which is a compact subring of $\mc{O}(\mf{X})$; here, we say an element $f\in\mc{O}(\mf{X})$ is \textit{power bounded} if the subset $\sset{f^m\in\mc{O}(\mf{X})}{m\in\Z_{>0}}\subset\mc{O}(\mf{X})$ is bounded in the natural Banach topology on $\mc{O}(\mf{X})$. Actually, note that if we find functions $a_{w'}$ such that \eqref{eqnawxfrob} holds, then each $a_{w'}$ is power bounded by continuity because the traces $\tr(\Frob_{w'}^{-1}|\rho_{\pi_x})$ are necessarily $p$-adically integral. So to prove (a) it suffices now to construct functions $a_{w'}\in\mc{O}(\mf{X})$ such that \eqref{eqnawxfrob} holds.

So let $w'\in S_0$, and let $v'$ be the place of $F$ below $w'$ (so $w'=v'$ inside of $E=F$ if $G$ is not unitary). Then we have that $\pi_{v'}$ is an unramified representation of the quasisplit group $G(F_{v'})$ of rank $n$, and this group is split unless $G=SO_{2n}^\eta$ with $\eta$ having its local component $\eta_{v'}$ at $v'$ be nontrivial.

Let $T$ be the maximal $F_{v'}$-split torus in the split group $G_{/F_{v'}}$, which is defined in Section \ref{secgroups} in all cases except the unitary case; if $G$ is unitary, so that $G_{/E_{w'}}\cong GL_{n/E_{w'}}$ and so $G_{/F_{v'}}\cong GL_{n/F_{v'}}$ via the natural isomorphism $F_{v'}\cong E_{w'}$, let $T$ be the usual diagonal torus. Let $e_1^\vee$ be the cocharacter sending $t\in\GL_{1/E_{w'}}$ to $\diag(t,1,\dotsc,1)$ in the unitary case; otherwise, let $e_1^{\vee}$ be the cocharacter of $T$ which is dual to $e_1$ with respect to the basis $e_1,\dotsc,e_n$, or $e_1,\dotsc,e_{n-1}$ if $G=SO_{2n}^\eta$ with $\eta_{v'}\ne 1$, of the $E_{w'}$-rational character group as in Section \ref{secgroups}.

Let $D$ be the dimension the vector space $V_x$ associated with $x\in\Sigma$ as in Definition \ref{defheckefamilies}; this is independent of the choice of $x\in\Sigma$ by Lemma \ref{lemdimVxisconst}. Then we define the function $a_{w'}$ by
\begin{equation}
\label{eqnawxfrobjust1}
a_{w'}=D^{-1}\Psi(\iota\circ(\chars(K_{v'} e_1^\vee(\varpi_{v'})K_{v'})\otimes 1_{K_f^{v'}})),
\end{equation}
where $\varpi_{v'}$ is a fixed the uniformizer of $F_{v'}$, where $K_{v'}$ is the component of $K_f$ at $v'$, and where $K_f^{v'}$ is that away from $v'$. We now check that the relation \eqref{eqnawxfrob} holds.

To do this, we must recall some of the theory of the Satake transform. Let $X^*(T)$ denote the character group of the split torus $T$ defined just above. In all cases, one may identify this group with the quotient of the Levi subgroup of the standard minimal parabolic subgroup of $G(F_{v'})$ by its maximal compact subgroup. Let $W$ be the relative Weyl group of $T$ in $G_{/F_{v'}}$. Let $\mc{S}:C_c^{\infty}(K_{v'}\backslash G(F_{v'})/K_{v'},\C)\to \C[X^*(T)]^{W}$ denote the Satake transform, where $\C[X^*(T)]^{W}$ denotes the Weyl-invariants in the group ring $\C[X^*(T)]$. Then $\mc{S}$ is an isomorphism of $\C$-algebras.

When $G$ is not $SO_{2n}^\eta$ with $\eta_{v'}\ne 1$, so that $G_{F_{v'}}$ is split, the paragraph above is part of the standard theory of the Satake transform. But otherwise, $G_{/F_v}$ is quasisplit but not split, and the transform $\mc{S}$ is constructed in \cite[\S 4.2]{CartierCorv}, and proven to be an isomorphism.

We now wish to compute $\tr(\Frob_{w'}^{-1}|\rho_{\pi_x})$ in terms of the Satake parameters of the local component $\pi_{x,v'}$ of $\pi_{x}$ at $v'$. We will do this separately in the split and nonsplit cases.

So first assume that $G$ is not $SO_{2n}^\eta$ with $\eta_{v'}\ne 1$. Since the inclusion $F_{v'}\hookrightarrow E_{w'}$ is an isomorphism, the spherical representation $\pi_{x,v'}$ may be seen as a representation of $G(E_{w'})$ whose Satake parameter, call it $s_{w'}$, under the standard representation $\std$ of $G^\vee$ has
\[\Vert\varpi_{v'}\Vert_{v'}^{(1+\delta_{\mr{eo}}-N)/2}\std(s_{w'})=\iota^{-1}(\rho_{\pi_x}(\Frob_{w'}^{-1})^{\sss}),\]
by property (b) of Theorem \ref{thmconstofgalois}. Let $q_{w'}$ denote the cardinality of the residue field of $E_{w'}$ or, equivalently, of $F_{v'}$. We therefore have that
\begin{equation}
\label{eqnawxfrobjust2}
\tr(\Frob_{w'}^{-1}|\rho_{\pi_x})=\iota(q_{w'}^{(N-\delta_{\mr{eo}}-1)/2}\tr(s_{w'}|\std)).
\end{equation}

Now since $e_1^\vee$ is dominant and miniscule and, when viewed as a weight for $T^\vee$ in $G^\vee$, it is the highest weight of the standard representation of $G^\vee$, we have that
\begin{equation}
\label{eqnawxfrobjust3}
q_{w'}^{\langle\rho,e_1^\vee\rangle}\tr(s_{w'}|\std)=\mc{S}(\chars(K_{v'} e_1^\vee(\varpi_{v'})K_{v'}))(s_{w'});
\end{equation}
see for example \cite[Formula (3.13)]{gross} for a justification of this. But note also that
\begin{equation}
\label{eqnawxfrobjust4}
\langle\rho,e_1^\vee\rangle=\frac{N-\delta_{\mr{eo}}-1}{2}.
\end{equation}
Thus, combining \eqref{eqnawxfrobjust2}, \eqref{eqnawxfrobjust3} and \eqref{eqnawxfrobjust4}, we get
\begin{equation}
\label{eqnawxfrobjust5}
\tr(\Frob_{w'}^{-1}|\rho_{\pi_x})=\iota(\mc{S}(\chars(K_{v'} e_1^\vee(\varpi_{v'})K_{v'}))(s_{w'})),
\end{equation}
in the split case.

Now assume instead that $G=SO_{2n}^\eta$ with $\eta_{v'}\ne 1$. Then the $L$-parameter $\phi_{x,v'}$ of $\pi_{x,v'}$ may be described as follows. A maximal torus in $G_{/F_{v'}}$ is given by $T_{\mr{max}}=T\times SO_{2/F_{v'}}^\eta$. Let $\chi_{x,v'}:T_{\mr{max}}(F_{v'})\to\C^\times$ be the unramified character whose parabolic induction to $G(F_{v'})$ has $\pi_{x,v'}$ as the unramified constituent; note that this character is trivial on $SO_2^\eta(F_{v'})$ because this group is compact, as it may be identified with the set of norm $1$ elements in the quadratic extension of $F_{v'}$ cut out by $\eta_{v'}$. Then the $L$-parameter $\phi_{x,v'}$ factors through $T^\vee(\C)\times O_2(\C)\subset O_{2n}(\C)\subset GL_{2n}(\C)$ under the standard representation \eqref{eqnstdofLgps} and is determined by
\begin{multline}
\label{eqnphixvprimediag}
\phi_{x,v'}(\varphi_{v'}^{-1})=\diag(\chi_{x,v'}(e_1^\vee(\varpi_{v'})),\dotsc,\chi_{x,v'}(e_{n-1}^\vee(\varpi_{v'})),-1,1,\\
\chi_{x,v'}(e_{n-1}^\vee(\varpi_{v'}))^{-1},\dotsc,\chi_{x,v'}(e_1^\vee(\varpi_{v'}))^{-1}),
\end{multline}
where $\varphi_{v'}$ is the arithmetic Frobenius element in $W_{F_{v'}}$, and $e_1^\vee,\dotsc,e_{n-1}^\vee$ is the dual basis to $e_1,\dotsc,e_{n-1}$. This parameter is trivial on $SL_{2,\mr{D}}(\C)$. Since $\eta_{v'}$ is nontrivial and unramified by assumption, the $O_2(\C)$-factor of the image of $\varphi_{v'}$ under $\phi_{x,v'}$ must lie in the non-identity component of $O_2(\C)$; this accounts for the $(-1,1)$ in the middle of the list above.

Let $s_{w'}=s_{v'}\in T^\vee(\C)$ now be the Satake parameter of $\pi_{x,v'}$, so
\[s_{w'}=\prod_{i=1}^{n-1}e_i(\chi_{x,v'}(e_i^\vee(\varpi_{v'}))),\]
where each $e_i$ above is viewed as a cocharacter of $T^\vee$, and each $e_i^\vee$ is viewed as one for $T$. Then since the relative root system of $T^\vee$ in $G_{/F_{v'}}$ is of type $B$, the dual torus $T^\vee$ is maximal in a group of type $C_{n-1}$, which we can take to be $Sp_{2n-2}(\C)$. The trace of the standard representation on this group applied to $s_{w'}$ is thus given by
\[\sum_{i=1}^{n-1}(\chi_{x,v'}(e_i^\vee(\varpi_{v'}))+\chi_{x,v'}(e_i^\vee(\varpi_{v'}))^{-1}).\]
Thus, letting $B$ be the standard minimal parabolic in $G_{/F_{v'}}$ containing $T_{\mr{max}}$, the theory of the Satake transform still computes
\[\mc{S}(\chars(K_{v'} e_1^\vee(\varpi_{v'})K_{v'}))(s_{w'})=\delta_B(e_1^\vee(\varpi_{v'}))^{-1/2}\sum_{i=1}^{n-1}(\chi_{x,v'}(e_i^\vee(\varpi_{v'}))+\chi_{x,v'}(e_i^\vee(\varpi_{v'}))^{-1});\]
the argument used to justify \cite[Formula (3.13)]{gross} still shows that there are no other terms on the right hand side above, because $e_1^\vee$ is dominant and miniscule. The sum on the right hand side above visibly coincides with the trace of \eqref{eqnphixvprimediag}, and the term $\delta_B(e_1^\vee(\varpi_{v'}))^{-1/2}$ is given by $q_{w'}^{n-1}=q_{w'}^{(N-\delta_{\mr{eo}}-1)/2}$. Altogether, this justifies the formula \eqref{eqnawxfrobjust5} again, this time in the nonsplit case.

Now we have that
\begin{equation}
\label{eqnawxfrobjust6}
\begin{aligned}
a_{w'}(x)&=x(D^{-1}\Psi(\iota\circ(\chars(K_{v'} e_1^\vee(\varpi_{v'})K_{v'})\otimes 1_{K_f^{v'}})))\\
&=D^{-1}\iota(\tr(\chars(K_{v'} e_1^\vee(\varpi_{v'})K_{v'})\otimes 1_{K_f^{v'}}|V_x))\\
&=\iota(\mc{S}(\chars(K_{v'} e_1^\vee(\varpi_{v'})K_{v'}))(s_{w'})),
\end{aligned}
\end{equation}
the first equality by \eqref{eqnawxfrobjust1}, the second by Definition \ref{defheckefamilies}, and the third because the operator $\chars(K_{v'} e_1^\vee(\varpi_{v'})K_{v'})\otimes 1_{K_f^{v'}}$ acts on every vector in $V_x$ by the scalar
\[\mc{S}(\chars(K_{v'} e_1^\vee(\varpi_{v'})K_{v'}))(s_{w'}).\] Thus, combining \eqref{eqnawxfrobjust5} and \eqref{eqnawxfrobjust6} gives \eqref{eqnawxfrob}. As noted above, part (a) of the proposition follows.

As for (b), we first prove that the characters $\chi\Vert\cdot\Vert^{(1+\delta_{\mr{eo}}-N)/2}$ and $\chi^*\Vert\cdot\Vert^{(1+\delta_{\mr{eo}}-N)/2}$ are algebraic. Indeed, by assumption, since $x_0$ is strongly $M$-Eisenstein, the component at any archimedean place $v_\infty$ of $F$ of the induced representation $\Ind_{P(\A_F)}^{G(\A_F)}(\chi\boxtimes \sigma)$ has infinitesimal character $\lambda_{v_\infty}+\rho$ for some integral weight $\lambda_{v_\infty}$. Thus, by the proof of Proposition \ref{propparamliftcoho}, its infinitesimal character becomes integral after adding $(\tfrac{1+\delta_{\mr{eo}}-N}{2},\tfrac{1+\delta_{\mr{eo}}-N}{2},\dotsc,\tfrac{1+\delta_{\mr{eo}}-N}{2})$. This implies that the character $\chi\Vert\cdot\Vert^{(1+\delta_{\mr{eo}}-N)/2}$ is algebraic. Thus 
\[(\chi\Vert\cdot\Vert^{(1+\delta_{\mr{eo}}-N)/2})^*\Vert\cdot\Vert^{1+\delta_{\mr{eo}}-N}=\chi^*\Vert\cdot\Vert^{(1+\delta_{\mr{eo}}-N)/2}\]
is also algebraic.

Now for $w'\in S_0$ lying above some $v'$ in $F$, write $t_{w'}$ for the Satake parameter at $w'$ of $\sigma_{v'}$, similarly as in the proof of (a) above. Let us also write $\tilde{t}_{w'}$ for that of the unramified constituent of $\Ind_{P(\A_F)}^{G(\A_F)}(\chi\boxtimes \sigma)_{v'}$. Then
\begin{equation}
\label{eqnttildesatake}
\std(\tilde{t}_{w'})=\diag(\chi(\varpi_{w'}),\chi^*(\varpi_{w'}),t_{w'}).
\end{equation}
The argument above used to prove \eqref{eqnawxfrobjust6} in part (a) still shows that
\[a_{w'}(x_0)=\iota(\mc{S}(\chars(K_v e_1^\vee(\varpi_{v'})K_v))(\tilde{t}_{w'}))=\iota(q_{w'}^{(N-\delta_{\mr{eo}}-1)/2}\tr(\tilde{t}_{w'}|\std)).\]
By (a), the relation 
\[x(T_{\mc{F}}(\Frob_{w'}^{-1}))=a_{w'}(x),\]
holds for every $x\in\Sigma$. Since $\Sigma$ is dense, we therefore get
\[x_0(T_{\mc{F}}(\Frob_{w'}^{-1}))=a_{w'}(x_0)=\iota(q_{w'}^{(N-\delta_{\mr{eo}}-1)/2}\tr(\tilde{t}_{w'}|\std)).\]
Thus we have, by \eqref{eqnttildesatake}, that
\begin{equation}
\label{eqnalmostpartb}
\begin{aligned}
x_0(T_{\mc{F}}(\Frob_{w'}^{-1}))=&\,\iota(q_{w'}^{(N-\delta_{\mr{eo}}-1)/2}\chi(\varpi_{w'})+q_{w'}^{(N-\delta_{\mr{eo}}-1)/2}\chi^*(\varpi_{w'})\\
&+q_{w'}^{(N-\delta_{\mr{eo}}-1)/2}\tr(t_{w'}|\std))\\
=&\,\rho_{\chi\Vert\cdot\Vert^{(1+\delta_{\mr{eo}}-N)/2}}(\Frob_{w'}^{-1})+\rho_{\chi^*\Vert\cdot\Vert^{(1+\delta_{\mr{eo}}-N)/2}}(\Frob_{w'}^{-1})\\
&\,+\iota(q_{w'}^{(N-\delta_{\mr{eo}}-1)/2}\tr(t_{w'}|\std)).
\end{aligned}
\end{equation}
But from property (b) of Theorem \ref{thmconstofgalois}, we have
\[q_{w'}^{(N-\delta_{\mr{eo}}-1)/2}\tr(t_{w'}|\std)=q_{w'}q_{w'}^{(N-\delta_{\mr{eo}}-3)/2}\tr(t_{w'}|\std)=\iota^{-1}(\chi_{\cyc}^{-1}(\Frob_{w'}^{-1})\rho_{\sigma}(\Frob_{w'}^{-1})),\]
where $\chi_{\cyc}$ is the $p$-adic cyclotomic character. Combining this with \eqref{eqnalmostpartb} and invoking Chebotarev density gives part (b) of the proposition.
\end{proof}

\begin{proposition}
\label{propgalfamonfreemod}
With $\mc{F}=(\mb{T},\mf{X},\Sigma,\Psi)$ as above, fix $x_0\in\mf{X}(\overline\Q_p)$. Then there is a Zariski open affinoid neighborhood $\mf{X}_0$ of $x_0$ in $\mf{X}$, a reduced affinoid rigid space $\mf{Y}$ over $\Q_p$, a surjective and generically finite map $q:\mf{Y}\to\mf{X}_0$, and a continuous semisimple representation $\rho_{\mf{Y}}:G_E\to GL_N(\mc{O}(\mf{Y}))$ such that the following holds: For any $x\in\mf{X}_0(\overline\Q_p)$ and any $y\in\mf{Y}$ above $x$, and for any $g\in G_E$, we have
\begin{equation}
\label{eqntracefromY}
\tr(\rho_{\mf{Y}}(g))(y)=(T_{\mc{F}}(g))(x),
\end{equation}
where $T_\mc{F}$ is the pseudorepresentation of Proposition \ref{propconstofpsrep}.
\end{proposition}

\begin{proof}
This follows from \cite[Lemma 7.8.11]{BCbook}; in fact \textit{loc. cit.} constructs a rigid space, call it $\mf{Y}'$, which is finite and dominant over $\mf{X}$, and a blow-up $\mf{Y}''$ of $\mf{Y}'$ with a locally free coherent sheaf of $\mc{O}_{\mf{Y}''}$-modules of rank $N$ satisfying a property analogous to \eqref{eqntracefromY}. Our proposition follows from this upon taking Zariski open affinoid neighborhoods.
\end{proof}

We remark that in the above proposition, it may happen that the fiber in $\mf{Y}$ above $x_0$ has positive dimension.

\section{Weil--Deligne representations}
\label{secWDreps}
We begin this section with some generalities about Weil--Deligne representations. We fix throughout this section a nonarchimedean local field $L$ of characteristic $0$ and an arithmetic Frobenius element $\varphi\in W_L$. We let $\lambda$ denote the order of the residue field of $L$, and we fix a square root $\lambda^{1/2}$ of $\lambda$ in $\overline\Q_p$. We denote the inertia group of $L$ as usual by $I_L$.

We will have occasion below to consider Weil--Deligne representations over rings which are not fields; specifically, we will consider $p$-adic families of Weil--Deligne representations which we will view as Weil--Deligne representations over affinoid rings. For clarity, we make the following definitions.

\begin{definition}
\label{defWDrepoverA}
Let $k$ be a field of characteristic $0$ and $A$ a $k$-algebra. By a \textit{Weil--Deligne representation} for $L$ over $A$ we shall mean an $A$-module $M$, together with an action
\[r:W_L\to\aut_A(M)\]
with open kernel, and with a nilpotent operator
\[N\in\End_A(M)\]
such that, for all $g\in W_L$, we have
\[r(g)N=\lambda^{v(g)}Nr(g).\]
Here, we recall that $v:W_L\to\Z$ is the unique homomorphism factoring through $W_L/I_L$ and such that $v(\varphi)=1$. Such a Weil--Deligne representation will be denoted by $(r,N,M)$, or usually just by $(r,N)$ when $M$ is understood.
\end{definition}

There is then an obvious notion of a morphism, and hence of an isomorphism, of such Weil--Deligne representations, as well as a notion of quotients of such, and of Weil--Deligne subrepresentations.

\begin{definition}
Let $k$ be a field of characteristic $0$ and $A$ a $k$-algebra. Let $(r,N,M)$ be a Weil--Deligne representation for $L$ over $A$ in the sense of Definition \ref{defWDrepoverA}. Given an $A$-algebra $B$, we define the base changed Weil--Deligne representation $(r,N,M)\otimes_A B$ for $L$ over $B$ by
\[(r,N,M)\otimes_A B=(r\otimes\id_B,N\otimes\id_B,M\otimes_A B),\]
where $\id_B:B\to B$ is the identity map. It is easy to verify then that $(r,N,M)\otimes_A B$ is indeed a Weil--Deligne representation for $L$ over $B$.
\end{definition}

We now set some notation which will be used later.

\begin{notation}
\label{notnorderonNs}
Let $k=\bar{k}$ be an algebraically closed field of characteristic $0$. Let $N,N'$ be nilpotent operators on a finite dimensional $k$-vector space $V$. Then we write $N\sim N'$ if $N$ and $N'$ are conjugate by $GL(V)$. We write $N\prec N'$ if $N$ is in the Zariski closure of the $GL(V)$-orbit of $N'$.
\end{notation}

\begin{notation}
\label{notnisotypic}
Let $k=\bar{k}$ be an algebraically closed field of characteristic $0$. Let $\tau$ be an irreducible, finite dimensional representation of the inertia group $I_L$ with open kernel and $k$-coefficients. Let $(r,N)$ be a finite dimensional Weil--Deligne representation for $L$ over $k$, say acting on a finite dimensional $k$-vector space $V$. We assume $(r,N)$ is Frobenius-semisimple, that is, the action of any arithmetic Frobenius element via $r$ is semisimple.

Then we write $r[\tau]$ for the $\tau$-isotypic component of $r$, by which we mean restriction of $r$ to the span, call it $V[\tau]$, of all subspaces of $V$ on which $r|_{I_L}$ acts by $\tau$. We also write $N[\tau]$ for the restriction of $N$ to this space $V[\tau]$. Then we write
\[(r,N)[\tau]=(r[\tau],N[\tau]),\]
which is again a finite-dimensional, Frobenius-semisimple Weil--Deligne representation for $L$ over $k$.
\end{notation}

Thus, in the notation above, we have
\[(r,N)\cong\bigoplus_\tau(r,N)[\tau],\]
where the sum is over all isomorphism classes of irreducible, finite dimensional representations $\tau$ of $I_L$ over $k=\bar{k}$ with open kernel.

We now give a general lemma about irreducible representations of $W_L$.

\begin{lemma}
\label{lemweilrepsuptounrtw}
Let $k=\bar{k}$ be an algebraically closed field of characteristic $0$. let $d$ be a positive integer, and let $r:W_L\to GL_d(k)$ be representation with open kernel. Then $r$ is irreducible if and only if $r|_{I_L}$ is. Moreover, in this case, if $r':W_L\to GL_d(k)$ is another representation with $r'|_{I_L}\cong r|_{I_L}$, then there is an $\alpha\in k^\times$ such that $r'\cong r\otimes\chi_\alpha$; here $\chi_\alpha$ has a similar meaning as in Section \ref{secparams}, namely it is the unramified character $W_L\to k^\times$ with $\chi_\alpha(\varphi)=\alpha$.
\end{lemma}

\begin{proof}
This lemma is well known and is a consequence of Schur's lemma; we omit the proof.
\end{proof}

With this setup, we can state the following lemma; it will be the only place in this section where the setting from the previous sections plays a role.

\begin{lemma}
\label{lemexistsWDforrhoY}
Let the setting be as in Proposition \ref{propgalfamonfreemod}, so that in particular there is a reduced affinoid rigid space $\mf{Y}$ and a continuous, semisimple representation $\rho_{\mf{Y}}:G_E\to GL_N(\mc{O}(\mf{Y}))$. Let $\iota:\C\to\overline\Q_p$ we the isomorphism that was fixed in the previous section in order to ultimately define the representation $\rho_{\mf{Y}}$. Let
\[(r_y,N_y)=\WD((y\circ\rho_\mf{Y})|_{G_L})^{\varphi-\sss}\otimes\chi_{\iota(q_w^{(N-\delta_{\mr{eo}}-1)/2})}\]
be the $\chi_{\iota(q_w^{(N-\delta_{\mr{eo}}-1)/2})}$-twist of Frobenius-semisimple Weil--Deligne representation attached to the restriction of $\rho_y$ to $G_{E_w}$. Then there exists a Weil--Deligne representation $(r_{\mf{Y}},N_{\mf{Y}})$ into the same free $\mc{O}(\mf{Y})$-module of rank $N$ on which $\rho_{\mf{Y}}$ acts with the property that, for any $y\in\mf{Y}(\overline{\Q}_p)$, we have
\[((r_{\mf{Y}},N_{\mf{Y}})\otimes_y\overline\Q_p)^{\varphi-\sss}\cong(r_y,N_y).\]
\end{lemma}

\begin{proof}
This follows (without the twist, which may be inserted afterward) from \cite[Lemmas 7.8.12 and 7.8.14]{BCbook}.
\end{proof}

We remark that the twist that was inserted above was put there so that the representations $(r_y,N_y)$ would be pure of weight $0$ for $y$ in the set $\Sigma$ given to us by the Hecke family in Proposition \ref{propgalfamonfreemod}; this purity assertion follows immediately from Proposition \ref{proppurityofgal}.

We now shift focus a bit and fix for the rest of this section:
\begin{itemize}
\item A positive integer $d$;
\item A reduced, irreducible, affinoid rigid space $\mf{Y}$ over $\Q_p$;
\item A Weil--Deligne representation $(r_{\mf{Y}},N_{\mf{Y}})$ on a free $\mc{O}(\mf{Y})$-module of rank $d$.
\end{itemize}
We remark that, as far as the rest of this section is concerned, the objects $\mf{Y}$ and $(r_{\mf{Y}},N_{\mf{Y}})$ just fixed need not be the ones from Lemma \ref{lemexistsWDforrhoY} above; however, in the application to our main theorem later, we will take them to be so.

We set the following useful piece of notation.

\begin{notation}
\label{notnmathcalT}
We write $\mc{T}$ for the set of isomorphism classes of irreducible, finite dimensional representations of $I_L$ with open kernel and $\overline\Q_p$-coefficients. For a class $\tau\in\mc{T}$, we often abusively write $\tau$ as well for a fixed representative of this class.
\end{notation}

Let us write $(r_{\bar{\eta}},N_{\bar{\eta}})$ for the Frobenius-semisimplified, base changed representation,
\[(r_{\bar{\eta}},N_{\bar{\eta}})=((r_{\mf{Y}},N_{\mf{Y}})\otimes\overline{\Frac(\mc{O}(\mf{Y}))})^{\varphi-\sss}.\]
Then there is an isomorphism
\begin{equation*}
\label{eqnWDisotodirsumorig}
(r_{\bar{\eta}},N_{\bar{\eta}})\cong\bigoplus_{\tau\in\mc{T}} \bigoplus_{i\in\mc{I}^\tau}r_i^\tau\otimes\chi_{\alpha_i^\tau}\otimes\Sp(m_i^\tau),
\end{equation*}
where
\begin{itemize}
\item For each $\tau\in\mc{T}$, the set $\mc{I}^\tau$ is a finite index set;
\item For each $\tau\in\mc{T}$ and each $i\in\mc{I}^\tau$, the factor $r_{i}^\tau$ is a fixed, irreducible, finite dimensional representation of $W_L$ with $\overline{\Q}_p$-coefficients such that $r_i^\tau|_{I_L}\cong\tau$;
\item Each element $\alpha_i^\tau\in\overline{\Frac(\mc{O}(\mf{Y}))}^\times$;
\item For each $\tau\in\mc{T}$ and each $i\in\mc{I}^\tau$, the number $m_i^\tau$ is a positive integer.
\end{itemize}
Note we are using Lemma \ref{lemweilrepsuptounrtw} in order to assume each $r_0^\tau$ has $\overline{\Q}_p$-coefficients instead of coefficients in $\overline{\Frac(\mc{O}(\mf{Y}))}$. Thus there is a direct sum decomposition of $W_L$-representations with coefficients in $\overline{\Frac(\mc{O}(\mf{Y}))}$ as follows:
\begin{equation}
\label{eqnWDisotodirsum}
r_{\bar{\eta}}=\bigoplus_{\tau\in\mc{T}}\bigoplus_{i\in\mc{I}^\tau}\bigoplus_{\substack{1-m_i^\tau \leq j\leq m_i^\tau-1\\ j\equiv m_i^\tau-1\modulo{2}}} r_{\bar{\eta},i,j}^\tau\quad\textrm{with}\quad r_{\bar{\eta},i,j}^\tau\cong r_i^\tau\otimes\chi_{\lambda^{j/2}\alpha_i^\tau}.
\end{equation}
We recall here that the special representation $\Sp(m)$, for $m$ a positive integer, was defined in Section \ref{secparams}, and we note that it may depend on the choice of square root $\lambda^{1/2}$ of $\lambda$ that we fixed at the beginning of this section.

We would like to be able to pass from the decomposition \eqref{eqnWDisotodirsum} above to similar decompositions for the specializations of $r_\mf{Y}$ at various points $y\in\mf{Y}(\overline\Q_p)$. But this will not work directly over $\mf{Y}$, the main issue being that the elements $\alpha_i^\tau$ are not necessarily in $\mc{O}(\mf{Y})$. However, the following proposition resolves this issue after making a benign base change.

\begin{proposition}
\label{propexistenceofmfz}
There is a reduced, irreducible, affinoid rigid space $\mf{Z}$ admitting a finite, surjective map $\pi:\mf{Z}\to\mf{Y}$ with the following properties:
\begin{enumerate}[label=(\roman*)]
\item For each $\tau\in\mc{T}$ and each $i\in\mc{I}^\tau$, the ring $\mc{O}(\mf{Z})$ contains the fields of definition of $r_i^\tau$;
\item For each $\tau\in\mc{T}$ and each $i\in\mc{I}^\tau$, we have that $\alpha_i^\tau\in\mc{O}(\mf{Z})^\times$;
\item We have that $\lambda^{1/2}\in \mc{O}(\mf{Z})$.
\end{enumerate}
\end{proposition}

\begin{proof}
Recall we have our fixed lift $\varphi$ of Frobenius in $W_L$. Let $\lambda_{\mf{Y},k}\in\overline{\Frac(\mc{O}(\mf{Y}))}$ be the eigenvalues of $r_{\mf{Y}}(\varphi)$ in any chosen order, where $1\leq k\leq d=\dim(r_{\bar{\eta}})$, and let $\lambda_{0,k}\in\overline\Q_p$ be those of the representation with coefficients in $\overline{\Q}_p$ given by
\begin{equation}
\label{eqnWeilsumwithoutalphas}
\bigoplus_{\tau\in\mc{T}}\bigoplus_{i\in\mc{I}^\tau}\bigoplus_{\substack{1-m_i^\tau \leq j\leq m_i^\tau-1\\ j\equiv m_i^\tau-1\modulo{2}}} r_i^\tau\otimes\chi_{\lambda^{j/2}}
\end{equation}
again in any chosen order. We take $\mf{Z}$ be the normalization of $\mf{Y}$ in the finite extension field of $\Frac(\mc{O}(\mf{Y}))$ obtained from adjoining to it the following: All of the elements $\lambda_{\mf{Y},k}$ and $\lambda_{0,k}$, the field of definition of each $r_i^\tau$, and the element $\lambda^{1/2}$. Then the natural map $\mf{Z}\to\mf{Y}$ is finite and it is surjective by going up. Thus we need to show that for any $k$, we have that $\lambda_{\mf{Y},k},\lambda_{0,k}\in\mc{O}(\mf{Z})^\times$, as well as $\lambda^{1/2}\in\mc{O}(\mf{Z})^\times$.

Of course the elements $\lambda_{0,k}$ and $\lambda^{1/2}$ are in $\mc{O}(\mf{Z})^\times$ because they are algebraic over $\Q_p$ and nonzero, and $\mc{O}(\mf{Y}')$ is integrally closed containing $\Q_p$. As for $\lambda_{\mf{Y},k}$, note first that
\[\tr(r_\mf{Y}(\varphi)^n)\in\mc{O}(\mf{Y})\]
for all integers $n$. So by the theory of symmetric polynomials, the coefficients of the characteristic polynomials of $r_\mf{Y}(\varphi)^{\pm 1}$ are in $\mc{O}(\mf{Y})$. Thus the roots $\lambda_{\mf{Y},k}^{\pm 1}$ are in $\mc{O}(\mf{Z})$, again because $\mc{O}(\mf{Z})$ is integrally closed, proving the claim.

It follows that the elements $\alpha_i^\tau$ are in $\mc{O}(\mf{Z})^\times$ as desired, since they are given by certain quotients of eigenvalues of $r_{\mf{Y}}(\varphi)$ by those of the representation of \eqref{eqnWeilsumwithoutalphas}.
\end{proof}

Now let us write
\[(r_\mf{Z},N_\mf{Z})=(r_\mf{Y},N_\mf{Y})\otimes_{\mc{O}(\mf{Y})}\mc{O}(\mf{Z})\]
where $\mf{Z}$ is as in the proposition above. Then $(r_\mf{Z},N_\mf{Z})$ is a Weil--Deligne representation into a free $\mc{O}(\mf{Z})$-module of rank $d$, the same rank as that of $(r_\mf{Y},N_\mf{Y})$. It has the property that for any point $y\in\mf{Y}(\overline{\Q}_p)$ and any $z\in\mf{Z}(\overline{\Q}_p)$ over $y$, we have
\[(r_{\mf{Z}},N_{\mf{Z}})\otimes_{z}\overline\Q_p\cong(r_{\mf{Y}},N_{\mf{Y}})\otimes_{y}\overline\Q_p.\]

\begin{proposition}
\label{propdecompofrzprime}
Let the setup be as above. For any $z\in\mf{Z}(\overline\Q_p)$, we have that
\[(r_\mf{Z}\otimes_{z}\overline\Q_p)^{\sss}\cong \bigoplus_{\tau\in\mc{T}} \bigoplus_{i\in\mc{I}^\tau}\bigoplus_{\substack{1-m_i^\tau \leq j\leq m_i^\tau-1\\ j\equiv m_i^\tau-1\modulo{2}}} r_i^\tau\otimes\chi_{\lambda^{j/2}\alpha_i^\tau(z)},\]
where the semisimplification on the right is as $W_L$-representations (or, equivalently, it is Frobenius semisimplification).
\end{proposition}

\begin{proof}
Note that
\[\tr(r_{\mf{Z}})=\tr(r_{\bar{\eta}})=\sum_{\tau\in\mc{T}}\sum_{i\in\mc{I}^\tau}\sum_{\substack{1-m_i^\tau \leq j\leq m_i^\tau-1\\ j\equiv m_i^\tau-1\modulo{2}}} \tr(r_i^\tau\otimes\chi_{\lambda^{j/2}\alpha_i^\tau}),\]
the latter equality following from \eqref{eqnWDisotodirsum}. Thus for any $z\in\mf{Z}(\overline\Q_p)$, we have that
\[z(\tr(r_\mf{Z}))=\sum_{\tau\in\mc{T}}\sum_{i\in\mc{I}^\tau}\sum_{\substack{1-m_i^\tau \leq j\leq m_i^\tau-1\\ j\equiv m_i^\tau-1\modulo{2}}} \tr(r_i^\tau\otimes\chi_{\lambda^{j/2}\alpha_i^\tau(z)}).\]
The proposition follows, since both sides of the asserted isomorphism are semisimple representations of $W_L$ with coefficients in $\overline\Q_p$.
\end{proof}

\begin{proposition}
\label{propwdvari}
Notation as set above, for any $z\in\mf{Z}(\overline\Q_p)$, let us also write
\[(r_z,N_z)=((r_{\mf{Z}},N_{\mf{Z}})\otimes_{z}\overline\Q_p)^{\varphi-\sss}.\]
\begin{enumerate}[label=(\alph*)]
\item For any $z\in\mf{Z}(\overline\Q_p)$, we have
\[(r_{z}|_{I_L})\otimes_{\overline{\Q}_p}\overline{\Frac(\mc{O}(\mf{Z}))}\cong (r_{\mf{Z}}|_{I_L})\otimes_{\mc{O}(\mf{Z})}\overline{\Frac(\mc{O}(\mf{Z}))}.\]
In particular, for any $z_1,z_2\in\mf{Z}(\overline\Q_p)$, we have $r_{z_1}|_{I_L}\cong r_{z_2}|_{I_L}$.
\item There is a Zariski open and dense subset $\mf{U}$ of $\mf{Z}$ with the following properties: First, for any $z_1$ and $z_2$ in $\mf{U}(\overline\Q_p)$ and any continuous, irreducible, finite dimensional representation $\tau$ of $I_L$ with finite image, if we identify the space $r_{z_1}[\tau]$ (see Notation \ref{notnisotypic}) with $r_{z_2}[\tau]$ (which can be done by (a)) then $N_{z_1}[\tau]\sim N_{z_2}[\tau]$ (Notation \ref{notnorderonNs}). Moreover, for any $z$ in $\mf{U}(\overline\Q_p)$ and any such $\tau$, we have, under the obvious identification, that
\[N_{z}[\tau]\otimes_{\overline{\Q}_p}\overline{\Frac(\mc{O}(\mf{Z}))}\sim N_{\mf{Z}}[\tau]\otimes_{\mc{O}(\mf{Z})}\overline{\Frac(\mc{O}(\mf{Z}))}.\]
Finally, for any $z\in\mf{U}(\overline\Q_p)$ and any $z_0\in\mf{Z}(\overline\Q_p)$, then identifying the space of $r_{z}[\tau]$ with $r_{z_0}[\tau]$, we have $N_{z_0}[\tau]\prec N_z[\tau]$ (again see Notation \ref{notnorderonNs}).
\end{enumerate}
\end{proposition}

\begin{proof}
Part (a) follows from Proposition \ref{propdecompofrzprime} (and more generally, from \cite[Proposition 7.8.17]{BCbook}). Part (b) follows from \cite[Proposition 7.8.19]{BCbook}.
\end{proof}

In the next section, we study the degeneration of the operator $N_{\mf{Z}}$ at points not necessarily in $\mf{U}$ under certain purity assumptions.

\section{Combinatorics of weights}
\label{secweights}
We continue to study many of the objects we studied in the previous section. For the convenience of the reader, we summarize the relevant objects in play, which were all fixed after Lemma \ref{lemexistsWDforrhoY}.

We have first the characteristic $0$ nonarchimedean local field $L$ whose residue field contains $\lambda$ elements and is of characteristic different from $p$. We fixed a square root $\lambda^{1/2}$ of $\lambda$ in $\overline\Q_p$. Then we also have the affinoid rigid space $\mf{Y}$ over $\Q_p$. The ring $\mc{O}(\mf{Y})$ is an integral domain, and we have a Weil--Deligne representation $(r_\mf{Y},N_\mf{Y})$ of $W_L$ acting on a free $\mc{O}(\mf{Y})$-module of finite rank $d$.

We have the Frobenius-semisimplified base change $(r_{\bar{\eta}},N_{\bar{\eta}})$ of $(r_{\mf{Y}},N_{\mf{Y}})$ to $\overline{\Frac(\mc{O}(\mf{Y}))}$, which decomposes as
\begin{equation}
\label{eqndecompofretaNeta}
(r_{\bar{\eta}},N_{\bar{\eta}})\cong\bigoplus_{\tau\in\mc{T}} \bigoplus_{i\in\mc{I}^\tau}r_i^\tau\otimes\chi_{\alpha_i^\tau}\otimes\Sp(m_i^\tau);
\end{equation}
we refer the reader to above \eqref{eqnWDisotodirsum} for a description of the pieces of this sum and its summands, suffice it to say that $\mc{T}$ is defined in Notation \ref{notnmathcalT}, each $\mc{I}^\tau$ is a finite set, each $m_i^\tau$ is a positive integer, and $\alpha_i^\tau\in\overline{\Frac(\mc{O}(\mf{Y}))}^\times$. This gave us a direct sum decomposition of $r_{\bar{\eta}}$ into irreducible $W_L$-representations \eqref{eqnWDisotodirsum}.

Now we have the reduced, irreducible, affinoid rigid space $\mf{Z}$ admitting a finite, surjective map $\pi:\mf{Z}\to\mf{Y}$ from Proposition \ref{propexistenceofmfz}. We continue to write $(r_\mf{Z},N_\mf{Z})$ for the base change of $(r_\mf{Y},N_\mf{Y})$ to $\mc{O}(\mf{Z})$. The rigid space $\mf{Z}$ has a Zariski dense and open subset $\mf{U}$ on which the monodromy operator $N_\mf{Z}$ is ``as large as possible;'' see Proposition \ref{propwdvari} (b) for a precise statement.

Fix two points $z,z_0\in\mf{Z}(\overline\Q_p)$. We now make the following assumptions about the behavior of $(r_\mf{Z},N_{\mf{Z}})$ at these two points.

\begin{assumption}
\label{assumptionsonzandz0}
Write
\[(r_z,N_z)=((r_\mf{Z},N_\mf{Z})\otimes_z\overline\Q_p)^{\varphi-\sss},\]
and
\[((r_{z_0},N_{z_0})=(r_\mf{Z},N_\mf{Z})\otimes_{z_0}\overline\Q_p)^{\varphi-\sss}.\]
We assume that
\begin{enumerate}[label=(\arabic*)]
\item The point $z$ is in $\mf{U}(\overline\Q_p)$;
\item The Weil--Deligne representation $(r_{z},N_{z})$ is pure of weight $0$;
\item There exists a second nilpotent operator $\overline{N}_0$ on $r_{z_0}$ with the property that $(r_{z_0},\overline{N}_0)$ is still a Weil--Deligne representation and, for any $\tau\in\mc{T}$, we have $\overline{N}_0[\tau]\prec N_{z_0}[\tau]$;
\item There exist:
\begin{itemize}
\item Elements $\tau_{0,+},\tau_{0,-}\in\mc{T}$ which are not necessarily distinct;
\item An integer $k_0\geq 0$;
\item Two irreducible, finite dimensional representations $r_{\tau_{0,+}}$ and $r_{\tau_{0,-}}$ of $W_L$ with $\overline{\Q}_p$-coefficients which are pure of respective weights $k_0$ and $-k_0$, such that $r_{\tau_{0,\pm}}|_{I_L}\cong \tau_{0,\pm}$;
\item A Frobenius-semisimple Weil--Deligne representation $(r_0',N_0')$ for $L$ with coefficients in $\overline{\Q}_p$ which is pure of weight $0$;
\end{itemize}
such that
\[(r_{z_0},\overline{N}_0)\cong (r_0',N_0')\oplus r_{\tau_{0,+}}\oplus r_{\tau_{0,-}}.\]
\end{enumerate}
\end{assumption}

With the setup concluded, we remark that in our application, the Weil--Deligne representation $(r_z,N_z)$ will be that attached a member in $\Sigma$ of a strongly $\Sigma$-generic Hecke family, and $(r_{z_0},N_{z_0})$ will be that attached to a strongly $M$-Eisenstein point on this family (see Definition \ref{defheckefamilies} for these notions). In this case, the representations $r_{\tau_{0,\pm}}$ above will be characters of $W_L$. The other Weil--Deligne representation $(r_{z_0},\overline{N}_0)$ will be the one coming from the \textit{semisimple} Galois representation attached to the Eisenstein representation corresponding to $z_0$; thus $(r_{z_0},N_{z_0})$ will come from the restriction to $G_L$ of a reducible, global Galois representation, and $(r_{z_0},\overline{N}_0)$ will come from the restriction to $G_L$ of the semisimplification of that Galois representation, which is why we have $\overline{N}_0\prec N_{z_0}$.

\begin{lemma}
\label{lemlambdaweil}
For any $\tau\in\mc{T}$ and any $i\in\mc{I}^\tau$, the number $\alpha_i^\tau(z_0)/\alpha_i^\tau(z)$ is a $\lambda$-Weil number.
\end{lemma}

\begin{proof}
Take $z'=z$ or $z'=z_0$ in Proposition \ref{propdecompofrzprime}. Then conditions (2) and (4) of Assumption \ref{assumptionsonzandz0} imply that every eigenvalue of $\varphi$ on $r_i^\tau\otimes\chi_{\lambda^{j/2}\alpha_i^\tau(z')}$, for $i$ and $j$ as in the sums of Proposition \ref{propdecompofrzprime}, are $\lambda$-Weil numbers. The lemma follows then by taking quotients of these $\lambda$-Weil numbers.
\end{proof}

We may thus set the following notation.

\begin{notation}
For any $\tau\in\mc{T}$ and any $i\in\mc{I}^\tau$, we write $a_i^\tau$ for the $\lambda$-Weil number $\alpha_i^\tau(z_0)/\alpha_i^\tau(z)$, and we write $w_i^\tau$ for its weight.
\end{notation}

The main goal of the rest of this section is to prove the following.

\begin{proposition}
\label{proppossibsforWDdegen}
Let the setting be as above.
\begin{enumerate}[label=(\alph*)]
\item For all $\tau\in\mc{T}$ and all $i\in\mc{I}^\tau$, we have that $-k_0\leq w_i^\tau\leq k_0$.
\item If there exists both $i_+\in \mc{I}^{\tau_{0,+}}$ with $w_{i_+}^{\tau_{0,+}}=k_0$ and $i_-\in \mc{I}^{\tau_{0,-}}$ with $w_{i_-}^{\tau_{0,-}}=-k_0$, then $\overline{N}_0[\tau]\sim N_z[\tau]$ for all $\tau$.
\end{enumerate}
\end{proposition}

To prove this proposition, it will be convenient to have the following lemma; despite its specificity, it will be used in different ways in proving the proposition above.

\begin{lemma}
\label{lemspecificlemma}
Let $\tau\in\mc{T}$. Let $(r_1,N_1)$ and $(r_2,N_2)$ be two finite dimensional Weil--Deligne representations with $\overline\Q_p$-coefficients. Assume $(r_1,N_1)=(r_1,N_1)[\tau]$ and $(r_2,N_2)=(r_2,N_2)[\tau]$. Assume $(r_1,N_1)$ is pure of weight $0$, and write
\begin{equation}
\label{eqnWDr1decomp}
r_1\cong\bigoplus_{i\in\mc{I}_1}r_{1,i}\otimes\Sp(m_i)
\end{equation}
for some finite index set $\mc{I}_1$ and some irreducible $W_L$-representations $r_{1,i}$ of weight $0$ with $r_{1,i}|_{I_L}\cong\tau$, and some positive integers $m_i$. Assume moreover that there are Weil numbers $a_i$ for each $i\in\mc{I}_1$ such that
\begin{equation}
\label{eqnWDr2decomp}
r_2\cong\bigoplus_{i\in\mc{I}_1}\bigoplus_{\substack{1-m_i\leq j\leq m_i-1 \\ j\equiv m_i-1\modulo{2}}} r_{1,i}\otimes\chi_{a_i\lambda^{j/2}},
\end{equation}
as $W_L$-representations. This implies $r_1$ and $r_2$ have the same dimension, and we assume that under one (equivalently, any) identification of the underlying vector spaces, we have $N_2\prec N_1$.

Now, assume that $(r_2,N_2)$ splits as
\[(r_2,N_2)=(r_2',N_2')\oplus (r_2'',N_2'')\]
with $(r_2',N_2')$ pure of weight $0$ and $(r_2'',N_2'')$ a (possibly empty) direct sum of Weil--Deligne representations which are all pure (not necessarily of weight $0$). Let $w_{\mr{max}}''$ be the largest weight (possibly $-\infty$) of an eigenvalue of $\varphi_L$ on $r_2''$, and $w_{\mr{min}}''$ the smallest weight (possibly $\infty$) of an eigenvalue of $\varphi_L$ on $r_2''$. Then for any $i\in\mc{I}_1$ such that $m_i-1\geq \max\{w_{\mr{max}}'',-w_{\mr{min}}''\}$, the Weil number $w_i$ is of weight $0$, and
\[\bigoplus_{\substack{i\in\mc{I}_1\\ m_i-1\geq \max\{w_{\mr{max}}'',-w_{\mr{min}}''\}}}r_{1,i}\otimes\chi_{a_i}\otimes\Sp(m_i)\]
is a direct summand of $(r_2',N_2')$.
\end{lemma}

\begin{proof}
Let $\mc{I}_1'$ be the set $i\in\mc{I}_1$ such that $m_i-1\geq \max\{w_{\mr{max}}'',-w_{\mr{min}}''\}$, and let $i\in\mc{I}_1'$ be such that $m_i$ is maximal. Then it follows from \eqref{eqnWDr2decomp} that $r_{1,i}\otimes\chi_{a_i\lambda^{(m_i-1)/2}}$ and $r_{1,i}\otimes\chi_{a_i\lambda^{-(m_i-1)/2}}$ are direct summands of $r_{2}$. 

Now assume for sake of contradiction that the weight $w_i$ of $a_i$ is strictly positive. Then the weight of $r_{1,i}\otimes\chi_{a_i\lambda^{(m_i-1)/2}}$ is strictly larger than $m_i-1>w_{\mr{max}}''$, and so we must have that $r_{1,i}\otimes\chi_{a_i\lambda^{(m_i-1)/2}}$ is a summand of $r_2'$. Then since $r_2'$ is pure of weight $0$, there must be a positive integer $M\geq m_i+w_i$ with $M\equiv m_i+w_i\modulo{2}$, such that $r_{1,i}\otimes\chi_{a_i\lambda^{-w_i}}\otimes\Sp(M)$ is a direct summand of $(r_2',N_2')$. But since $M>m$, this contradicts $N_2\prec N_1$. So $w_i\leq 0$, and completely symmetric argument then shows that $w_i=0$.

Thus the weight of the summand $r_{1,i}\otimes\chi_{a_i\lambda^{(m_i-1)/2}}$ equals $m_i-1>w_{\mr{max}}''$, so we still have that it is a summand of $r_2'$. Again we find that there must be a positive integer $M\geq m_i$ with $M\equiv m_i\modulo{2}$, such that $r_{a_i}\otimes\Sp(M)$ is a direct summand of $(r_2',N_2')$. And similarly as above, if actually $M>m_i$, then this would contradict the fact that $N_2\prec N_1$. So $r_{1,i}\otimes\chi_{a_i}\otimes\Sp(m_i)$ is a direct summand of $(r_1',N_1')$.

Now for this $i$, consider the quotient representations
\[(r_1,N_1)/(r_{1,i}\otimes\Sp(m_i))\quad\textrm{and}\quad(r_2,N_2)/(r_{1,i}\otimes\chi_{a_i}\otimes\Sp(m_i)).\]
Then these respective representations satisfy the same list of hypotheses that $(r_1,N_1)$ and $(r_2,N_2)$ do in the statement of the lemma. Thus an easy induction on $\#\mc{I}_1'$ suffices to complete the proof.
\end{proof}

\begin{proof}[Proof (of Proposition \ref{proppossibsforWDdegen})]
We now proceed in several steps.

\textit{Step 1}. We first observe the following consequence of Lemma \ref{lemspecificlemma} along with the decomposition \eqref{eqndecompofretaNeta}: We have that
\begin{equation}
\label{eqndecompofrzNz}
(r_{z},N_{z})\cong\bigoplus_{\tau\in\mc{T}}\bigoplus_{i\in\mc{I}^\tau}r_i^\tau\otimes\chi_{\alpha_i^\tau(z)}\otimes\Sp(m_i^\tau).
\end{equation}
To see this, for any $\tau\in\mc{T}$, apply Lemma \ref{lemspecificlemma} with $(r_1,N_1)=(r_{z},N_{z})[\tau]$ and $(r_2,N_2)=(r_2',N_2')$ given by the $\tau$-isotypic component of the right hand side above. We can take $\mc{I}_1=\mc{I}^\tau$ there, as well as $a_i=1$ for all $i$. Then $(r_2'',N_2'')$ is trivial and both $(r_1,N_1)$ and $(r_2,N_2)$ are pure of weight $0$, and we can take $w_{\mr{max}}''=-\infty$ and $w_{\mr{min}}''=\infty$. We conclude that the $\tau$-component of the right hand side of \eqref{eqndecompofrzNz} is a summand of that of the left. They are thus isomorphic since they have the same dimension. Now sum over all $\tau$.

\textit{Step 2}. We now make a preliminary reduction. Let $\tau$ be such that $\tau\ne\tau_{0,+},\tau_{0,-}$. Then by item (4) of Assumption \ref{assumptionsonzandz0}, we have that $(r_{z_0},\overline{N}_0)[\tau]$ is pure of weight $0$. Moreover, by Proposition \ref{propwdvari} (b) along with item (3) of Assumption \ref{assumptionsonzandz0}, we also have that $\overline{N}_0[\tau]\prec N_{z}[\tau]$. We may thus apply Lemma \ref{lemspecificlemma} with $(r_1,N_1)=(r_{z},N_{z})[\tau]$ and $(r_2,N_2)=(r_2',N_2')=(r_{z_0},\overline{N}_0)[\tau]$ there, taking $\mc{I}_1=\mc{I}^\tau$ along with $r_{1,i}=r_i^\tau\otimes\chi_{\alpha_i^\tau(z)}$, $a_i=a_i^\tau$, and $m_i=m_i^\tau$ for each $i\in\mc{I}^\tau$; note that the hypothesis \eqref{eqnWDr2decomp} of Lemma \ref{lemspecificlemma} is satisfied by Proposition \ref{propdecompofrzprime}. Thus this tells us that
\[\bigoplus_{i\in\mc{I}^\tau}r_i^\tau\otimes\chi_{a_i^\tau}\otimes\Sp(m_i^\tau)\]
is a direct summand of $(r_{z_0},\overline{N}_0)[\tau]$, and hence they are equal for dimension reasons. Moreover, Lemma \ref{lemspecificlemma} also tells us that $w_i^\tau=0$ for all $i\in\mc{I}^\tau$. Thus part (a) of the proposition holds for $\tau\ne\tau_{0,+},\tau_{0,-}$, and we have moreover that $\overline{N}_0[\tau]\sim N_z[\tau]$.

\textit{Step 3}. Thus it suffices to study the case when $\tau=\tau_{0,+}$ or $\tau=\tau_{0,-}$. Consider the former case. We now apply Lemma \ref{lemspecificlemma} with $(r_1,N_1)=(r_{z},N_{z})[\tau_{0,+}]$ and $(r_2,N_2)=(r_{z_0},\overline{N}_0)[\tau_{0,+}]$ there, taking now $(r_2',N_2')=(r_0',N_0')[\tau_{0,+}]$ and taking $(r_2'',N_2'')$ to be $r_{\tau_{0,+}}$ if $\tau_{0,+}\ne\tau_{0,-}$, and otherwise taking $(r_2'',N_2'')$ to be $r_{\tau_{0,+}}\oplus r_{\tau_{0,-}}$ (see item (4) of Assumption \ref{assumptionsonzandz0}); we take $\mc{I}_1$ and each $r_{1,i}$, $a_i$ and $m_i$ like in Step 2 but with $\tau=\tau_{0,+}$. Then the number $\max\{w_{\mr{max}}'',-w_{\mr{min}}''\}$ of Lemma \ref{lemspecificlemma} is equal to $k_0$, and we obtain following: Let $\mc{I}_{+}'$ be the set of $i\in\mc{I}^{\tau_{0,+}}$ such that $m_i^{\tau_{0,+}}-1>k_0$. Then
\[\bigoplus_{i\in\mc{I}_+'}r_i^{\tau_{0,+}}\otimes\chi_{a_i^{\tau_{0,+}}}\otimes\Sp(m_i^{\tau_{0,+}})\]
is a direct summand of $(r_{z_0},\overline{N}_0)[\tau_{0,+}]$, and moreover $w_i^{\tau_{0,+}}=0$ for every $i\in\mc{I}_+'$. Thus, taking the quotient of $(r_{z_0},\overline{N}_0)[\tau_{0,+}]$ by this summand, and the quotient of $(r_z,N_z)$ by its summand
\[\bigoplus_{i\in\mc{I}_+'}r_i^{\tau_{0,+}}\otimes\chi_{\alpha_i^{\tau_{0,+}}(z)}\otimes\Sp(m_i^{\tau_{0,+}}),\]
we may, and will, therefore assume $\mc{I}_+'$ is empty, i.e., that $m_i^{\tau_{0,+}}-1\leq k_0$ for every $i\in\mc{I}^{\tau_{0,+}}$.

A completely symmetric argument allows us to assume that $m_i^{\tau_{0,-}}-1\leq k_0$ for every $i\in\mc{I}^{\tau_{0,-}}$, and we do assume this from now on.

\textit{Step 4}. With the assumptions just made, we now note the following: If $\tau=\tau_{0,+}$ or $\tau=\tau_{0,-}$, then there are no irreducible direct summands of $r_0[\tau]$ with weight $w$ having $\vert w\vert>k_0$. Indeed, the argument is similar to above, as such a summand would force the existence of a direct summand of $(r_{z_0},\overline{N}_0)[\tau]$ of the form $R\otimes\Sp(M)$ for some $W_L$-representation $R$ and some $M$ with $M-1>k_0+1$. But this is impossible since $\overline{N}_0[\tau]\prec N_z[\tau]$.

\textit{Step 5}. We now prove (a). Assume for sake of contradiction that there is an $i\in\mc{I}^{\tau_{0,+}}$ such that $w_i^{\tau_{0,+}}>k_0$. Then
\[r_i^{\tau_{0,+}}\otimes\chi_{\alpha_i^{\tau_{0,+}}(z)}\otimes\chi_{a_i^{\tau_{0,+}}\lambda^{(m_i^{\tau_{0,+}}-1)/2}}\]
is a direct summand of $r_{z_0}[\tau_{0,+}]$, and visibly this summand has weight $w_i^{\tau_{0,+}}+m_i^{\tau_{0,+}}-1>k_0$. But this contradicts the conclusion of Step 4. A completely symmetric argument shows also that $w_i^{\tau_{0,+}}\geq -k_0$, and similarly for $\tau_{0,-}$ in place of $\tau_{0,+}$. This suffices to complete the proof of (a).

\textit{Step 6}. Now we tackle (b). Assume from now on that there exists both $i_+\in \mc{I}^{\tau_{0,+}}$ with $w_{i_+}^{\tau_{0,+}}=k_0$ and $i_-\in \mc{I}^{\tau_{0,-}}$ with $w_{i_-}^{\tau_{0,-}}=-k_0$. We first claim that $m_{i_+}^{\tau_{0,+}}=m_{i_-}^{\tau_{0,-}}=1$. Indeed, otherwise, say if $m_{i_+}^{\tau_{0,+}}>1$, then
\[r_{i_+}^{\tau_{0,+}}\otimes\chi_{\lambda^{(m_{i_+}^{\tau_{0,+}}-1)/2}\alpha_{i_+}^{\tau_{0,+}}(z)}\]
is a direct summand of $r_z[\tau_{0,+}]$, and hence
\[r_{i_+}^{\tau_{0,+}}\otimes\chi_{\lambda^{(m_{i_+}^{\tau_{0,+}}-1)/2}a_{i_+}^{\tau_{0,+}}(z)}\]
is a direct summand of $r_{z_0}[\tau_{0,+}]$. But this representation has weight $k_0+m_{i_+}^{\tau_{0,+}}-1>k_0$, contradicting Step 4. So $m_{i_+}^{\tau_{0,+}}=1$, and a symmetric argument shows $m_{i_-}^{\tau_{0,-}}=1$.

\textit{Step 7}. For clarity, we note that, as at the beginning of Step 3, we currently may assume, and are assuming, that $r_z$ is the span of $r_z[\tau_{0,+}]$ and $r_z[\tau_{0,-}]$, and similarly for $r_0'$. Now for any integer $k$ with $-k_0\leq k \leq k_0$, let $n_{0}'(k,\pm)$ (resp. $n_{z}(k,\pm)$) be the nonnegative integer such that the span of all the weight $k$ summands of $r_0'[\tau_{0,\pm}]$ (resp. of $r_z/((r_{i_+}^{\tau_{0,+}}\otimes\chi_{\alpha_{i_+}^{\tau_{0,+}}(z)})\oplus (r_{i_-}^{\tau_{0,-}}\otimes\chi_{\alpha_{i_-}^{\tau_{0,-}}(z)}))[\tau_{0,\pm}]$) is the direct sum of $n_{0}'(k,\pm)$ irreducible $W_L$-representations. We claim that for all $k$, we have $n_{0}'(k,\pm)=n_z(k,\pm)$ (same sign). We will prove this simultaneously with the claim that for all $i\in\mc{I}^{\tau_{0,+}}\cup\mc{I}^{\tau_{0,-}}\backslash\{i_+,i_-\}$ (the union is disjoint if $\tau_{0,+}\ne\tau_{0,-}$), we have that the number $w_i^{\tau_{0,+}}$ or $w_i^{\tau_{0,-}}$, whichever makes sense, is $0$. We will prove these claims by descending induction on $\vert k\vert$ from $\vert k\vert =k_0$. This step will accomplish the base case.

For ease of notation, let us write $\widetilde{\mc{I}}=\mc{I}^{\tau_{0,+}}\cup\mc{I}^{\tau_{0,-}}\backslash\{i_+,i_-\}$, and let us drop the superscripts on $r_i^{\tau_{0,\pm}}$, $\alpha_i^{\tau_{0,\pm}}$ $a_i^{\tau_{0,\pm}}$, $m_i^{\tau_{0,\pm}}$, and $w_i^{\tau_{0,\pm}}$, and simply write $r_i$, $\alpha_i$, $a_i$, $m_i$, and $w_i$, respectively, when $i\in\widetilde{\mc{I}}$.

First, note that by Steps 3 and 4, we are assuming that there are no direct summands of $r_0'$ or $r_z$ of weight $k$ with $\vert k\vert >k_0$. Let $i\in\widetilde{\mc{I}}$ be such that $r_i\otimes\chi_{\alpha_i(z)}\otimes\Sp(m_i)$ has a direct summand of weight $k_0$. Then by weight $0$ purity of $r_z$ along with the assumption just noted, we have $m_i-1=k_0$. If, for this $i$, we had $w_i>0$, then the summand $r_i\otimes\chi_{\lambda^{k_0/2}\alpha_i(z_0)}$ of $r_0'$ would have weight $k_0+w_i>k_0$, contrary to our assumptions. So $w_i\leq 0$, and symmetrically, we must then have $w_i=0$.

Thus, invoking Proposition \ref{propdecompofrzprime}, we find that along with $r_{\tau_{0,+}}$, there are at least $n_z(k_0,\pm)$ other distinct summands of $r_0[\tau_{0,\pm}]$ of weight $k$ which are in direct sum. Allowing for one of these to be isomorphic to $r_{\tau_{0,+}}$, we find then $r_0'[\tau_{0,\pm}]$ contains at least $n_z(k_0,\pm)$ such summands i.e., $n_0'(k_0,\pm)\geq n_z(k_0,\pm)$. It follows from the purity of $r_0'$, along with the fact that $\overline{N}_0[\tau_{0,\pm}]\prec N_z[\tau_{0,\pm}]$, that actually we must have $n_0'(k_0,\pm)= n_z(k_0,\pm)$.

As usual, a symmetric argument implies that if $i\in\widetilde{\mc{I}}$ is such that $r_i\otimes\chi_{\alpha_i(z)}\otimes\Sp(m_i)$ has a direct summand of weight $-k_0$, then $w_i=0$, and then also implies the fact that $n_0'(-k_0,\pm)=n_z(-k_0,\pm)$. This completes the base case of the aforementioned induction.

\textit{Step 8}. We now carry out the induction step of the argument explained in the previous step.

If $k_0=0$, we are done with the induction. Otherwise, let $k$ be such that $0\leq\vert k\vert<k_0$. Assume that for all $k'$ with $\vert k\vert<\vert k'\vert\leq k_0$, we have that if $i\in\widetilde{\mc{I}}$ is such that $r_i\otimes\chi_{\alpha_i(z)}\otimes\Sp(m_i)$ has a direct summand of weight $k'$, then $w_i=0$ and $n_0'(k',\pm)=n_z(k',\pm)$. Using the notation of Step 7, for any integer $K$ with $0\leq K\leq k_0$, we let
\[(r_z,N_z)[K]=\bigoplus_{\substack{i\in\widetilde{\mc{I}}\\ m_i-1\geq K}}r_i\otimes\chi_{\alpha_i(z)}\otimes\Sp(m_i).\]
This is a direct summand of $(r_z,N_z)/((r_{i_+}^{\tau_{0,+}}\otimes\chi_{\alpha_{i_+}^{\tau_{0,+}}(z)})\oplus (r_{i_-}^{\tau_{0,-}}\otimes\chi_{\alpha_{i_-}^{\tau_{0,-}}(z)}))$, which makes sense by Step 6. The underlying $W_L$ representation is
\[\bigoplus_{\substack{i\in\widetilde{\mc{I}}\\ m_i-1\geq K}}\bigoplus_{\substack{1-m_i^\tau \leq j\leq m_i^\tau-1\\ j\equiv m_i^\tau-1\modulo{2}}}r_i\otimes\chi_{\lambda^{j/2}\alpha_i(z)},\]
and it contains all summands of $r_z$ with weight at least $K$ or at most $-K$. The corresponding $W_L$-subrepresentation of $r_0'$, let us call it $r_0'[K]$, coming from Proposition \ref{propdecompofrzprime} is
\[r_0'[K]=\bigoplus_{\substack{i\in\widetilde{\mc{I}}\\ m_i-1\geq K}}\bigoplus_{\substack{1-m_i^\tau \leq j\leq m_i^\tau-1\\ j\equiv m_i^\tau-1\modulo{2}}}r_i\otimes\chi_{\lambda^{j/2}\alpha_i(z_0)}.\]
By the induction hypothesis, if $K>\vert k\vert$, then the direct sum of this representation with $(r_{i_+}^{\tau_{0,+}}\otimes\chi_{\alpha_{i_+}^{\tau_{0,+}}(z_0)})\oplus (r_{i_-}^{\tau_{0,-}}\otimes\chi_{\alpha_{i_-}^{\tau_{0,-}}(z_0)})$ contains all summands of $r_{z_0}$ with weight at least $K$ or at most $-K$.

Thus, if there is an $i\in\widetilde{\mc{I}}$ is such that $r_i\otimes\chi_{\alpha_i(z)}\otimes\Sp(m_i)$ has a direct summand of weight $k$ and $w_i\ne 0$, then first we must have $m_i-1=\vert k\vert$ by weight $0$ purity along with the induction hypothesis. Moreover, we would have then that the $W_L$-representation $(r_0'[\vert k\vert])[\tau_{0,\pm}]$ would have either more than $n_0(\vert k\vert+\vert w_i\vert,\pm)$ summands of weight $\vert k\vert+\vert w_i\vert$, or more than $n_0(-(\vert k\vert+\vert w_i\vert),\pm)$ summands of weight $-(\vert k\vert+\vert w_i\vert)$, all signs depending on the signs of $k$ and $w_i$ and the underlying $I_L$-representation of $r_i$, and this contradicts the induction hypothesis. Thus $w_i=0$, and so Proposition \ref{propdecompofrzprime} implies that $n_0'(k,\pm)\geq n_z(k,\pm)$.

Now since the representations $(r_z,N_z)/((r_{i_+}^{\tau_{0,+}}\otimes\chi_{\alpha_{i_+}^{\tau_{0,+}}(z)})\oplus (r_{i_-}^{\tau_{0,-}}\otimes\chi_{\alpha_{i_-}^{\tau_{0,-}}(z)}))[\tau_{0,\pm}]$ and $(r_0',N_0')[\tau_{0,\pm}]$ are pure of weight $0$, it follows that the number of disjoint summands of either of these representations of the form $r\otimes\Sp(m)$, for some irreducible $W_L$-representation $r$ and fixed positive integer $m$, equals $n_z(m-1,\pm)-n_z(m+1,\pm)=n_z(-(m-1),\pm)-n_z(-(m+1),\pm)$ and respectively $n_0'(m-1,\pm)-n_0'(m+1,\pm)=n_0'(-(m-1),\pm)-n_0'(-(m+1),\pm)$. Thus the induction hypothesis along with the relation $\overline{N}_0\prec N_z$ shows that the inequality $n_0'(k,\pm)\geq n_z(k,\pm)$ just proved is actually an equality. This completes the induction.

\textit{Step 9}. We now complete the proof of (b). By Steps 7 and 8, we know that $n_0'(k,\pm)=n_z(k,\pm)$ for all $k$ with $0\leq k\leq k_0$. Like at the end of Step 8, this implies that for any integer $m>0$, the number of disjoint summands of either $(r_0',N_0')[\tau_{0,\pm}]$ or of
\[(r_z,N_z)/((r_{i_+}^{\tau_{0,+}}\otimes\chi_{\alpha_{i_+}^{\tau_{0,+}}(z)})\oplus (r_{i_-}^{\tau_{0,-}}\otimes\chi_{\alpha_{i_-}^{\tau_{0,-}}(z)}))[\tau_{0,\pm}]\]
of the form $r\otimes\Sp(m)$, for some irreducible $W_L$-representation $r$, are equal. It follows that that the monodromy operators on these Weil--Deligne representations are conjugate. From this it immediately follows that $\overline{N}_0[\tau_{0,\pm}]\sim N_z[\tau_{0,\pm}]$, as desired.
\end{proof}

\section{Controlling monodromy}

We now continue with our group $G$ over our totally real number field $F$, along with the number field $E$ which equals $F$ if $G$ is not unitary, and otherwise is an imaginary quadratic extension of $F$. We state and prove our main theorem in full.

\begin{theorem}
\label{thmmainthm}
Fix an isomorphism $\iota:\C\overset{\sim}{\longrightarrow}\overline{\Q}_p$. Let $K_f$ be a factorizable, compact open subgroup of $G(\A_{F,f})$, and let $\mc{F}=(\mb{T},\mf{X},\Sigma,\Psi)$ be a Hecke family of level $K_f$ for $G$ which is strongly $\Sigma$-generic, in the sense of Definition \ref{defheckefamilies}, defined using $\iota$. For $x\in\Sigma$, let $\pi_x$ be the automorphic representation for $G$ attached to $x$. Assume there is a point $x_0\in\mf{X}(\overline\Q_p)$ which is strongly $M$-Eisenstein for $\mc{F}$ and induced by some $(\chi,\sigma)$, again as in Definition \ref{defheckefamilies}.

Fix a finite place $v$ of $F$ not lying above $2$ or $p$, and if $G$ is unitary, assume $v$ is inert or ramified in $E$. Let $w$ be the place above $v$ in $E$ if $G$ is unitary, and otherwise let $w=v$ in $E=F$.

Let $(r_{\sigma},N_{\sigma})$ be the $(N-2)$-dimensional, Frobenius-semisimple Weil--Deligne representation for $W_{E_w}$ over $\overline\Q_p$ associated with the local $L$-parameter of $\sigma$ at $v$ and $\iota$, and for any $x\in\Sigma$, let $(r_{\pi_x},N_{\pi_x})$ be that associated with the local $L$-parameter of $\pi_x$ at $v$ and $\iota$. Let us abusively write $\chi_v$ for both the local component of $\chi$ at $v$ and for the character $W_{E_w}\to\C^\times$ attached to it by class field theory. Write $\chi_v^*$ for its (conjugate-)dual.

Assume $\mf{X}$ is Zariski irreducible. Then there are:
\begin{itemize}
\item A compact open subgroup $K_v'\subset G(F_v)$;
\item A Hecke operator $\phi_v\in C_c^\infty(K_v'\backslash G(F_v)/K_v',\overline\Q_p)$;
\item A $\iota^{-1}(\phi_v\otimes 1_{K_f^v})$-eigenvector $u_v\in\Ind_{P(F_v)}^{G(F_v)}(\chi_v\boxtimes \sigma_v)^{K_v'}$, say with eigenvalue $\lambda(u_v)$;
\end{itemize}
such that, if we assume
\begin{enumerate}[label=(\arabic*)]
\item Writing $K_{f,v}$ for the component of $K_f$ at $v$ and $K_f^v$ for that away from $v$, we have that $K_{f,v}\subset K_v'$;
\item The operator $\phi_v\otimes 1_{K_f^v}\in\mb{T}$;
\item If $V_{x_0}'$ is the constituent of $(\Ind_{P(\A_F)}^{G(\A_F)}(\chi\boxtimes\sigma)_f^{K_f})^{\mb{T}-\sss}$ given to us by the strong $M$-Eisenstein property of $x_0$ in Definition \ref{defheckefamilies}, then there is a vector $u\in V_{x_0}'$ such that $\iota^{-1}(\phi_v\otimes 1_{K_v'})u=\lambda(u_v)u$;
\end{enumerate}
then for any $x\in\Sigma$ outside a proper Zariski closed subset of $\mf{X}$, the underlying $I_{E_w}$-representations in
\[(r_\sigma\oplus(\iota\circ\chi_v)\oplus(\iota\circ\chi_v^*),N_\sigma\oplus 0\oplus 0)\quad \textrm{and}\quad (r_{\pi_x},N_{\pi_x})\]
are isomorphic, and for any continuous, irreducible, finite dimensional representation $\tau$ of $I_{E_w}$ with $\overline\Q_p$-coefficients and open kernel, we have that the monodromy operators above satisfy
\[(N_\sigma\oplus 0\oplus 0)[\tau]\sim N_{\pi_x}[\tau].\]
Here, if we write $(r_\sigma\oplus(\iota\circ\chi_v)\oplus(\iota\circ\chi_v^*),N_\sigma\oplus 0\oplus 0)[\tau]$ and $(r_{\pi_x},N_{\pi_x})$ for the $\tau$-isotypic components of these respective Weil--Deligne representations, then $(N_\sigma\oplus 0\oplus 0)[\tau]$ and $N_{\pi_x}[\tau]$ denote the respective monodromy operators in these $\tau$-isotypic components.
\end{theorem}

\begin{proof}
Throughout the proof we will identify $\C\cong\overline\Q_p$ via $\iota$ and drop all instances of $\iota$ from the notation. Fix throughout a uniformizer $\varpi_w$ in $E_w$. We proceed in several steps.

\textit{Step 1}. Let $q_w$ denote the cardinality of the residue field of $E_w$. We first note that the $W_{E_w}$-character $\chi_v$ is pure of some weight $k_v\in\Z$, i.e., that $\chi_v(\varpi_w^{-1})$ is a $q_w$-Weil number of weight $k_v$. Indeed, we already saw in Proposition \ref{propconstofpsrep} (b) that $\chi\Vert\cdot\Vert^{(1-\delta_{\mr{eo}}-N)/2}$ is algebraic, which implies this immediately.

\textit{Step 2}. In this step, we use the content of Section \ref{sectypes} to define the objects $K_v'$, $\phi_v$, and $u_v$ whose existence is claimed in the theorem.

First, let $(\tau_v,L)$ be a pair with $L$ an $F_v$-Levi subgroup of $G_{/F_v}$ contained in $M_{/F_v}$, and $\tau_v$ a supercuspidal representation of $L(F_v)$ such that $\chi_v\boxtimes\sigma_v$ is a constituent of an induced representation of the form $\Ind_{L(F_v)}^{M(F_v)}((\chi_v\boxtimes\tau_v)\otimes\psi_v)$ for some unramified character $\psi_v$ of $L(F_v)$. Let $\varrho$ be a supercuspidal type for $(\chi_v\boxtimes\tau_v,L)$ in $G_{/F_v}$ satisfying the conclusions of Proposition \ref{propexistenceofcovers}; this is where we use $v\nmid 2$.

We then take $K_v'$ to be the kernel of $\varrho$. We also let $\phi_v=\phi_{\varrho}$, where $\phi_{\varrho}$ is the operator of Definition \ref{defheckeoptype} defined with respect to this type, using the fixed uniformizer $\varpi_w$ in $E_w$. We define $u_v$ below now.

First, let $\varepsilon\in\{\emptyset,*\}$ be the symbol determined by
\[\varepsilon=\begin{cases}
*&\textrm{if }k_v\textrm{ is nonnegative;}\\
\emptyset&\textrm{if }k_v\textrm{ is negative.}
\end{cases}\]
We then write
\[\chi_v^\varepsilon=\begin{cases}
\chi_v&\textrm{if }\varepsilon=\emptyset;\\
\chi_v^*&\textrm{if }\varepsilon=*.
\end{cases}\]
This symbol is defined so that $\chi_v^{\varepsilon}(\varpi_w)$ is a $q_w$-Weil number of nonnegative weight $\vert k_v\vert$ always.

Now, identifying $M(F_v)=E_w^\times\times H(F_v)$, we write $\varpi_w\times 1\in M(F_v)$ for the corresponding element of $E_w^\times\times H(F_v)$. Note that by Corollary \ref{corjacofeisenstein}, the Jacquet module
\[\Jac_{M}(\chi_v\rtimes\sigma_v)\]
contains the constituent $\chi_v^{\varepsilon}\Vert\cdot\Vert_w^{(N-\delta_{\mr{eo}}-1)/2}\boxtimes\sigma$ with $k_v$ the nonzero integer defined in Step 1. Let $e_{\varrho,M}$ be the idempotent associated with the restriction of $\varrho$ to $M(F_v)$, as in Proposition \ref{propexistenceofcovers} (ii), and similarly for $e_{\varrho,G}$. Then by the transitivity of covers as in Proposition \ref{propexistenceofcovers} (i), the idempotent $e_{\varrho,M}$ acts in a nonzero way on the fixed vector subspace
\[(\chi_v^{\varepsilon}\Vert\cdot\Vert_w^{(N-\delta_{\mr{eo}}-1)/2}\boxtimes\sigma_v)^{K_v'\cap M(F_v)}.\]
Now for any vector
\[\bar{u}\in(\chi_v^{\varepsilon}\Vert\cdot\Vert_w^{(N-\delta_{\mr{eo}}-1)/2}\boxtimes\sigma_v)^{K_v'\cap M(F_v)},\]
we have
\[\frac{\chars((\varpi_w\times 1)(K_v'\cap M(F_v)))}{\vol(K_v'\cap M(F_v))}\bar{u}=\chi_v^{\varepsilon}(\varpi_w)\Vert\varpi_w\Vert_w^{(N-\delta_{\mr{eo}}-1)/2}\bar{u}.\]
It follows that the operator $\bar{\phi}_v$ given by
\[\bar\phi_v=e_{\varrho,M}*\frac{\chars((\varpi_w\times 1)(K_v'\cap M(F_v)))}{\vol(K_v'\cap M(F_v))}*e_{\varrho,M}\]
acts by the scalar $\chi_v^{\varepsilon}(\varpi_w)\Vert\varpi_w\Vert_w^{(N-\delta_{\mr{eo}}-1)/2}$ on the space
\[e_{\varrho,M}*(\chi_v^{\varepsilon}\Vert\cdot\Vert_w^{(N-\delta_{\mr{eo}}-1)/2}\boxtimes\sigma_v).\]
Therefore, the space
\[e_{\varrho,M}*\Jac_{M}(\chi_v\rtimes\sigma_v)\]
contains an eigenvector for $\bar{\phi}_v$ with eigenvalue
\[\chi_v^{\varepsilon}(\varpi_w)\Vert\varpi_w\Vert_w^{(N-\delta_{\mr{eo}}-1)/2}=q_w^{(1+\delta_{\mr{eo}}-N)/2}\chi_v^{\varepsilon}(\varpi_w).\]
By definition, we have $\phi_v=t_{M,G}(\bar{\phi}_v)$, where $t_{M,G}$ is the Bushnell--Kutzko transfer map of Proposition \ref{propexistenceofcovers} (ii), and so it follows from Proposition \ref{propexistenceofcovers} (ii) and (iii) that the space
\[e_{\varrho,G}*\Ind_{P(F_v)}^{G(F_v)}(\chi_v\boxtimes \sigma_v)\] contains an eigenvector for $\phi_v$ with eigenvalue $q_w^{(1+\delta_{\mr{eo}}-N)/2}\chi_v^{\varepsilon}(\varpi_w)$; we take $u_v$ to be any such eigenvector, and hence
\[\lambda(u_v)=q_w^{(1+\delta_{\mr{eo}}-N)/2}\chi_v^{\varepsilon}(\varpi_w).\]

\textit{Step 3}. We now set up the objects on the Galois side we will need.

For $x\in\Sigma$, let $\rho_x=\rho_{\pi_x}$ be the semisimple Galois representation given to us by Theorem \ref{thmconstofgalois}. Applying Proposition \ref{propgalfamonfreemod}, we get the following: There is a Zariski open affinoid neighborhood $\mf{X}_0$ of $x_0$ in $\mf{X}$, a reduced affinoid rigid space $\mf{Y}$ over $\Q_p$, a surjective and generically finite map $\mf{Y}\to\mf{X}_0$, and a continuous semisimple representation $\rho_{\mf{Y}}:G_\Q\to GL_N(\mc{O}(\mf{Y}))$ such that the following holds: For any $x\in\mf{X}_0(\overline\Q_p)$ and any $y\in\mf{Y}$ above $x$, and for any $g\in G_E$, we have
\begin{equation}
\label{eqntracefromYinpf}
\tr(\rho_{\mf{Y}}(g))(y)=\tr(\rho_{\pi_x}(g));
\end{equation}
the above relation follows from combining \eqref{eqntracefromY} with the defining property of the pseudorepresentation $T_\mc{F}$ of Proposition \ref{propconstofpsrep}.

Now as in Lemma \ref{lemexistsWDforrhoY}, write
\[(r_y,N_y)=\WD((y\circ\rho_\mf{Y})|_{G_L})^{\varphi-\sss}\otimes\chi_{q_w^{(N-\delta_{\mr{eo}}-1)/2}}.\]
Then by that lemma, we get a Frobenius-semisimple Weil--Deligne representation $(r_{\mf{Y}},N_{\mf{Y}})$ into the same free $\mc{O}(\mf{Y})$-module of rank $N$ on which $\rho_{\mf{Y}}$ acts with the property that, for any $y\in\mf{Y}(\overline{\Q}_p)$, we have
\[((r_{\mf{Y}},N_{\mf{Y}})\otimes_y\overline\Q_p)^{\varphi-\sss}\cong(r_y,N_y).\]

As we did after Lemma \ref{lemexistsWDforrhoY}, we consider the base change
\[(r_{\bar{\eta}},N_{\bar{\eta}})=((r_{\mf{Y}},N_{\mf{Y}})\otimes_{\mc{O}(\mf{Y})}\overline{\Frac(\mc{O}(\mf{Y}))})^{\varphi-\sss}\]
and write
\begin{equation}
\label{eqqnretabardecompinpf}
(r_{\bar{\eta}},N_{\bar{\eta}})\cong\bigoplus_{\tau\in\mc{T}} \bigoplus_{i\in\mc{I}^\tau}r_i^\tau\otimes\chi_{\alpha_i^\tau}\otimes\Sp(m_i^\tau),
\end{equation}
for some finite index set $\mc{I}^\tau$, some irreducible, finite dimensional representation of $W_L$ with $\overline{\Q}_p$-coefficients with $r_i^\tau|_{I_L}\cong\tau$, some elements $\alpha_i^\tau\in\overline{\Frac(\mc{O}(\mf{Y}))}^\times$, and some positive integers $m_i^\tau$; we recall that $\mc{T}$ is as in Notation \ref{notnmathcalT}. Actually, by Proposition \ref{propexistenceofmfz}, there is a reduced, irreducible, affinoid rigid space $\mf{Z}$ admitting a finite, surjective map $\pi:\mf{Z}\to\mf{Y}$ such that $\alpha_i^\tau\in\mc{O}(\mf{Z})^\times$ for all $\tau$ and $i$.

We also write
\[(r_\mf{Z},N_\mf{Z})=(r_\mf{Y},N_\mf{Y})\otimes_{\mc{O}(\mf{Y})}\mc{O}(\mf{Z}),\]
and, for $z\in\mf{Z}(\overline\Q_p)$, we write
\[(r_z,N_z)=((r_{\mf{Z}},N_{\mf{Z}})\otimes_z\overline\Q_p)^{\varphi-\sss}.\]
Then, of course, we have that
\[(r_z,N_z)=(r_{\pi(z)},N_{\pi(z)}).\]

We observe here that the pseudorepresentation $T_{\mc{F}}$ is (conjugate-)self dual by interpolation, because each $\rho_{\pi_x}$ is for $x\in\Sigma$; that is, we have
\[T_{\mc{F}}(cgc)=T_{\mc{F}}(g^{-1})\]
for all $g\in G_E$, where $c$ is any complex conjugation in $G_F$. Thus the representations $\rho_\mf{Y}$ and $r_{\bar{\eta}}$ are (conjugate-)self dual as well. This will be useful later.

\textit{Step 4}. We claim now that for any point $z_0\in\mf{Z}(\overline\Q_p)$ above $x_0$, we have that
\[r_{z_0}\cong r_\sigma\oplus(\iota^{-1}\circ\chi_v)\oplus(\iota^{-1}\circ\chi_v^*),\]
and for any $\tau\in\mc{T}$, we have
\[(N_\sigma\oplus 0\oplus 0)[\tau]\prec N_{z_0}[\tau].\]
Indeed, Proposition \ref{propconstofpsrep} (b) tells us that for $y_0=\pi(z_0)\in\mf{Y}(\overline\Q_p)$, we have
\[(y_0\circ\rho_{\mf{Y}})^{\sss}=\rho_{\sigma}(-1)\oplus \rho_{\chi\Vert\cdot\Vert^{(1+\delta_{\mr{eo}}-N)/2}}\oplus\rho_{\chi^*\Vert\cdot\Vert^{(1+\delta_{\mr{eo}}-N)/2}},\]
where the semisimplification on the left is as $G_E$-representations and $\rho_{\chi\Vert\cdot\Vert^{(1+\delta_{\mr{eo}}-N)/2}}$ and $\rho_{\chi^*\Vert\cdot\Vert^{(1+\delta_{\mr{eo}}-N)/2}}$ are the $G_E$-characters associated with the Hecke characters $\chi\Vert\cdot\Vert^{(1+\delta_{\mr{eo}}-N)/2}$ and $\chi^*\Vert\cdot\Vert^{(1+\delta_{\mr{eo}}-N)/2}$ by class field theory via $\iota$. Therefore it follows that
\begin{equation}
\label{eqnWDofssrep}
\WD((y_0\circ\rho_\mf{Y})^{\sss}|_{G_L})^{\varphi-\sss}\otimes\chi_{q_w^{(N-\delta_{\mr{eo}}-1)/2}}=(r_\sigma\oplus\chi_v\oplus\chi_v^*,N_\sigma\oplus 0\oplus 0).
\end{equation}
But by definition, we have
\[\WD((y_0\circ\rho_\mf{Y})|_{G_L})^{\varphi-\sss}\otimes\chi_{q_w^{(N-\delta_{\mr{eo}}-1)/2}}=(r_{z_0},N_{z_0}),\]
whose left hand side is the same as that of \eqref{eqnWDofssrep} except that we did not semisimplify first. The claim then follows easily.

We note that the statement about $I_{E_w}$-representations in our theorem follows from here by Proposition \ref{propdecompofrzprime} or Proposition \ref{propwdvari} (a), because for any $x\in\Sigma\cap\mf{X}_0(\overline\Q_p)$ and any points $z,z_0\in\mf{Z}(\overline\Q_p)$ above $x$ and $x_0$, respectively, these propositions tell us in particular that the $I_{E_w}$-representations underlying $(r_{z_0},N_{z_0})$ and $(r_{z},N_{z})$ agree; but we also have that, if $y=\pi(z)$, then
\begin{equation}
\label{eqnWDofssrep2}
\WD((y\circ\rho_\mf{Y})^{\sss}|_{G_L})^{\varphi-\sss}\otimes\chi_{q_w^{(N-\delta_{\mr{eo}}-1)/2}}=(r_{\pi_x},N_{\pi_x}).
\end{equation}
by Proposition \ref{propconstofpsrep} (b). So this
does follow in view of \eqref{eqnWDofssrep}.

\textit{Step 5}. For each $x\in\Sigma$, let $V_x$ be the subspace of $(\pi_{x,f}^{K_f})^{\mb{T}-\sss}$ given to us by Definition \ref{defheckefamilies}. Then by Lemma \ref{lemdimVxisconst}, there is a positive integer $D$ such that $\dim_\C(V_x)=D$ for all $x\in\Sigma$. We now compute the eigenvalues of the operator $\phi_v\otimes 1_{K_f^v}$ constructed in Step 2 on $V_x$ for various $x\in\Sigma$.

First, let $\mf{U}_1\subset\mf{Y}$ be the open locus on which $(y\circ\rho_{\mf{Y}})^{\sss}=y\circ\rho_{\mf{Y}}$ for every $y\in\mf{U}_1(\overline{\Q}_p)$. We note that for any $x\in\Sigma \cap\mf{X}_0(\overline\Q_p)$, if $x$ has a point $y$ which is in $\mf{U}_1(\overline{\Q}_p)$ lying above it, then $\rho_{\pi_x}\cong y\circ\rho_{\mf{Y}}$, since both are then semisimple and they have the same trace by \eqref{eqntracefromYinpf}. It follows then from Theorem \ref{thmconstofgalois} (b) that for such $x$ and $y$, we have that the local $L$-parameter of $\pi_x$ at $v$ for such $x$ is given by the Weil--Deligne representation $(r_y,N_y)$ via $\iota$. We will use this below and compute the eigenvalues of $\phi_v\otimes 1_{K_f^v}$ for these points $x$.

For $x\in\Sigma$, by Proposition \ref{propexistenceofcovers} (iii) and \eqref{eqnBuKuJac} of Proposition \ref{propexistenceofcovers} (ii), the operator $\phi_v\otimes 1_{K_f^v}$ acts on $\pi_{x,f}^{K_f}$ as the operator $\bar{\phi}$ defined by
\[\bar{\phi}=\left(e_{\varrho,M}*\frac{\chars((\varpi_w^\nu\times 1)(K_v'\cap M(F_v)))}{\vol(K_v'\cap M(F_v))}*e_{\varrho,M}\right)\otimes 1_{K_f^v}\]
acts on
\[e_{\varrho,M}*\Jac_{M}(\pi_{x,v})\otimes(\pi_{x,f}^{v})^{K_f^v}.\]
If $x$ is such that there is a point $y\in\mf{U}_1(\overline\Q_p)$ above it, then by Proposition \ref{propjacoftempered}, the possible constituents of this Jacquet module can be read off the local $L$-parameter for $\pi_{x}$ at $v$, hence from the Weil--Deligne representation $(r_y,N_y)$; they are all of the form
\[(((r_i^\tau\otimes\chi_{\alpha_i^\tau(x)})\circ\mr{Art}_w^{-1})\cdot\Vert\cdot\Vert_w^{b})\boxtimes\varsigma\quad\textrm{or}\quad(((r_i^\tau\otimes\chi_{\alpha_i^\tau(x)})^{*}\circ\mr{Art}_w^{-1})\cdot\Vert\cdot\Vert_w^{b})\boxtimes\varsigma,\]
for some $\tau\in\mc{T}$ which is $1$-dimensional and some $i\in\mc{I}^\tau$, where $\mr{Art}_w$ is the local Artin map, the character $r_i^\tau$ and the function $\alpha_i^\tau$ are as in the decomposition \eqref{eqqnretabardecompinpf} of Step 3, the number $b$ is in $\frac{1}{2}\Z$ with $\frac{N-\delta_{\mr{eo}}-1}{2}\leq b_i\leq \frac{N-\delta_{\mr{eo}}-1}{2}+\frac{n_v-1}{2}$ (we remind the reader that $n_v$ is the $F_v$-rank of $G_{/F_v}$), and $\varsigma$ is a smooth admissible representation of $H(F_v)$. We obtain the following: For each such $x$, there is a finite set $\tilde{\mc{J}}_x$ consisting of quintuples $(\tau,i,\varepsilon_i,b_i,n_i)$ where
\begin{itemize}
\item $\tau\in\mc{T}$ is $1$-dimensional such that $\mc{I}^\tau\ne\emptyset$;
\item $i\in\mc{I}^\tau$;
\item $\varepsilon_i\in\{\emptyset,*\}$;
\item $b_i\in\frac{1}{2}\Z$ with $\frac{N-\delta_{\mr{eo}}-1}{2}\leq b_i\leq \frac{N-\delta_{\mr{eo}}-1}{2}+\frac{n_v-1}{2}$;
\item $n_i$ is a positive integer;
\end{itemize}
such that $\sum_i n_i\leq D$, and such that, as a $GL_1(E_w)$-module, we have
\[e_{\varrho,M}*\Jac_{M}(\pi_{x,v})\otimes(\pi_{x,f}^{v})^{K_f^v}\cong\bigoplus_{(\tau,i,\varepsilon_i,b_i,n_i)\in\tilde{\mc{J}}_x}(((r_i^\tau\otimes\chi_{\alpha_i^\tau(x)})^{\varepsilon_i}\circ\mr{Art}_w^{-1})\cdot\Vert\cdot\Vert_w^{b_i})^{\oplus n_i}.\]
Now for each such $i$, let us write $e_i\in\{\pm 1\}$ for the number
\[e_i=\begin{cases}
1&\textrm{if }\varepsilon_i=\emptyset;\\
-1&\textrm{if }\varepsilon_i=*.
\end{cases}\]
Then we have
\[(((r_i^\tau\otimes\chi_{\alpha_i^\tau(x)})^{\varepsilon_i}\circ\mr{Art}_w^{-1})\cdot\Vert\cdot\Vert_w^{b_i})(\varpi_w)=q_w^{-b_i}(r_i^\tau)^{\varepsilon_i}(\varphi^{-1})(\alpha_i^\tau(x))^{-e_i},\]
where $\varphi=\mr{Art}_w^{-1}(\varpi_w)^{-1}$ is the arithmetic Frobenius element corresponding to $\varpi_w^{-1}$. It follows, since $V_x$ is a constituent of $(\pi_{x,f}^{K_f})^{\sss}$, that there is a subset $\mc{J}_x\subset\tilde{\mc{J}}_x$ such that the generalized eigenvalues of $\phi_v\otimes 1_{K_f^v}$ on $V_x$ are given by
\[q_w^{-b_i}(r_i^\tau)^{\varepsilon_i}(\varphi^{-1})(\alpha_i^\tau(x))^{-e_i}\textrm{ with multiplicity }n_i\textrm{ for each }(\tau,i,\varepsilon_i,b_i,n_i)\in\mc{J}_x,\]
for $x$ as above; this accounts for $\sum_{i}n_i$ eigenvalues, and the rest are $0$.

\textit{Step 6}. We just computed the $D$ eigenvalues of $\phi_v\otimes 1_{K_f^v}$ on the space $V_x$; the space $V_{x_0}$ given to us by Definition \ref{defheckefamilies} also has dimension $D$ by Lemma \ref{lemdimVxisconst}, and we know one of its eigenvalues, namely $\lambda(u_v)$, by Step 2. We now isolate and interpolate these eigenvalues.

By the crucial property \eqref{eqnheckefamcomp} of Definition \ref{defheckefamilies}, we have that
\[x(\Psi(\phi_v\otimes 1_{K_f^v})^\nu)=\tr((\phi_v\otimes 1_{K_f^v})^\nu|V_x),\]
for $x\in\Sigma\cup\{x_0\}$ and any integer $\nu>0$. By the theory of symmetric polynomials, there is a polynomial whose coefficients are polynomial combinations, which depend only on the integer $D$, of the $x(\Psi(\phi_v\otimes 1_{K_f^v})^\nu)$, and whose roots are exactly the eigenvalues of $\phi_v\otimes 1_{K_f^v}$ on $V_x$. There is thus a finite extension $\mc{L}/\Frac(\mc{O}(\mf{X}))$ such that, writing $\mf{Y}'$ for the normalization of $\mf{X}$ in $\mc L$, we have that $\mc{O}(\mf{Y}')$ contains elements $\beta_1,\dotsc,\beta_D$ with the property that, for any $y'\in\mf{Y}'(\overline\Q_p)$ above $x\in\Sigma\cup\{x_0\}$, we have that $\beta_1(y'),\dotsc,\beta_D(y')$ are the eigenvalues of $\phi_v\otimes 1_{K_f^v}$ acting on $V_x$ with multiplicity. After possibly rearranging, and by our crucial assumption (3) in the statement of the theorem, we may and will assume that $\beta_1$ is such that there is a $y_0'\in\mf{Y}'(\overline\Q_p)$ above $x_0$ such that
\[\beta_1(y_0)=\lambda(u_v)=q_w^{(1+\delta_{\mr{eo}}-N)/2}\chi_v^{\varepsilon}(\varpi_w),\]
the second equality above coming from Step 2.

Let $\widetilde{\Sigma}$ be the preimage in $\mf{Y}'$ of those elements in $\Sigma$ which have a preimage in the set $\mf{U}_1$ of Step 5. Then $\widetilde{\Sigma}$ is also Zariski dense in $\mf{Y}'$ because the map $\mf{Y}'\to\mf{X}'$ is finite. Also, note that there are only finitely many possible quintuples $(\tau,i,\varepsilon_i,b_i,n_i)$ that can appear in any of the sets $\mc{J}_x$ from Step 5. Thus there is a Zariski dense subset $\widetilde{\Sigma}'\subset\widetilde{\Sigma}$ such that, for all $y'\in\widetilde{\Sigma}'$, we have that
\begin{equation}
\label{eqnbeta1eval}
\beta_1(y')=q_w^{-b_{i_0}}(r_{i_0}^{\tau_0})^{\varepsilon_{i_0}}(\varphi^{-1})(\alpha_{i_0}^{\tau_0}(x))^{-e_{i_0}},
\end{equation}
for some \textit{fixed} quintuple $(\tau_0,i_0,\varepsilon_{i_0},b_{i_0},n_{i_0})$ (though the entry $n_{i_0}$ plays no role here), where $z\in\mf{Z}(\overline\Q_p)$ is any point whose image in $\mf{X}(\overline\Q_p)$ coincides with that of $y'$.

Now let $\mf{Z}'$ be the normalization of $\mf{Z}$ in any compositum of $\mc L$ and $\Frac(\mc{O}(\mf{Z}))$. Then $\mc{O}(\mf{Z}')$ contains the element $\alpha_{i_0}^{\tau_0}$. Therefore, the analytic functions defined by either side of \eqref{eqnbeta1eval} coincide, and so we have
\begin{equation}
\label{eqnbetaandalpha}
(r_{i_0}^{\tau_0})^{\varepsilon_{i_0}}(\varphi^{e_{i_0}})\alpha_{i_0}^{\tau_0}=(q_w^{b_{i_0}}\beta_1)^{-e_{i_0}}.
\end{equation}

\textit{Step 7}. With \eqref{eqnbetaandalpha} in hand, we now finish the proof. Let $z$ be any point in $\mf{Z}(\overline\Q_p)$ above a point in $x\in\Sigma$ and whose image in $\mf{Y}$ is in the open set $\mf{U}_1$ of Step 5. Then $(r_z,N_z)=(r_{\pi_x},N_{\pi_x})$, which is pure of weight $0$ by Proposition \ref{proppurityofgal}. Thus the character $r_{i_0}^{\tau_0}\otimes\chi_{\alpha_{i_0}^{\tau_0}(z)}$ of $W_{E_w}$ is pure of weight $0$. Hence $(r_{i_0}^{\tau_0})^{\varepsilon_{i_0}}(\varphi^{e_{i_0}})\alpha_{i_0}^{\tau_0}(z)$ is a Weil number of weight $0$, as it equals the evaluation of $r_{i_0}^{\tau_0}\otimes\chi_{\alpha_{i_0}^{\tau_0}(z)}$ on either $\varphi$ or possibly on $c\varphi c$, for some complex conjugation $c$ in $G_F$, when $G$ is unitary. However, for any $z_0'\in\mf{Z}(\overline\Q_p)$ over the point $y_0'$ from Step 6, writing $z_0$ for the image of $z_0'$ in $\mf{Z}$, we have by \eqref{eqnbetaandalpha} that
\[(r_{i_0}^{\tau_0})^{\varepsilon_{i_0}}(\varphi^{e_{i_0}})\alpha_{i_0}^{\tau_0}(z_0)=(q_w^{b_{i_0}}\lambda(u_v))^{-e_{i_0}}=(q_w^{b_{i_0}}q_w^{(1+\delta_{\mr{eo}}-N)/2}\chi_v^{\varepsilon}(\varpi_w))^{-e_{i_0}}.\]
The right hand side is a Weil number of weight
\[w_0=-e_{i_0}\left(b_{i_0}-\frac{N-\delta_{\mr{eo}}-1}{2}+\vert k_v\vert\right).\]
Because $\frac{N-\delta_{\mr{eo}}-1}{2}\leq b_{i_0}\leq \frac{N-\delta_{\mr{eo}}-1}{2}+\frac{n_v-1}{2}$, we therefore have that $\vert w_0\vert\geq \vert k_v\vert$ with equality holding if and only if $b_{i_0}=\frac{N-\delta_{\mr{eo}}-1}{2}$; this crucial numerical comparison is at the heart of the matter.

In any case, it follows from above that $\alpha_{i_0}^{\tau_0}(z_0)/\alpha_{i_0}^{\tau_0}(z)$ is a Weil number of weight $w_0$. By the (conjugate-)self duality observed at the end of Step 3, there is also another index $i_1\in\mc{I}^{\tau_1}$ for some $\tau_1\in\mc{T}$ for which $\alpha_{i_1}^{\tau_1}(z_0)/\alpha_{i_1}^{\tau_1}(z)$ is a Weil number of weight $-w_0$.

Now we invoke Proposition \ref{proppossibsforWDdegen}, taking $r_{\tau_{0,+}}=\chi_v^{\varepsilon}$ and $r_{\tau_{0,-}}=(\chi_v^{\varepsilon})^*$ there, and taking $\overline{N}_0=N_\sigma\oplus 0\oplus 0$, as well as $i_+=i_0$ and $i_-=i_1$, or vice-versa, depending on the sign of $w_0$, and finally taking $k_0=k_v$. (If $k_v=0$, the choice of sign does not matter, and we can make one arbitrarily.) The hypothesis we need that $(N_\sigma\oplus 0\oplus 0)[\tau]\prec N_{z_0}[\tau]$ for any $\tau\in\mc{T}$ is satisfied by the claim proved in Step 4. Then Proposition \ref{proppossibsforWDdegen} (a) says that $\vert w_0\vert\leq\vert k_v\vert$. By our crucial numerical comparison above, this implies $\vert w_0\vert=\vert k_v\vert$, and so the hypotheses of Proposition \ref{proppossibsforWDdegen} (b) are satisfied. We conclude that $(N_\sigma\oplus 0\oplus 0)[\tau]\sim N_z[\tau]=N_{\pi_x}[\tau]$ for such $z$.

Since the map $\mf{Y}\to\mf{X}_0$ is surjective, the locus of points in $\mf{X}_0$ which have a point in $\mf{U}_1$ above them contains a Zariski open in $\mf{X}_0$. Thus the conclusion
\[(N_\sigma\oplus 0\oplus 0)[\tau]\sim N_{\pi_x}[\tau]\]
is satisfied for $x\in\Sigma$ away from a proper Zariski closed subset, proving the theorem.
\end{proof}

\printbibliography

\end{document}